\documentclass[12pt,a4paper]{amsart}
\usepackage[czech,english]{babel}
\usepackage{amssymb,amsxtra,eucal} 
\usepackage{tikz-cd}
\usepackage{hyperref}

\calclayout

\newcommand{\+}{\protect\nobreakdash-}
\newcommand{\<}{\protect\nobreakdash--}
\renewcommand{\:}{\colon}

\newcommand{\rarrow}{\longrightarrow}
\newcommand{\ot}{\otimes}
\newcommand{\ocn}{\odot}
\newcommand{\tim}{\rightthreetimes}

\newcommand{\lrarrow}{\mskip.5\thinmuskip\relbar\joinrel\relbar\joinrel
 \rightarrow\mskip.5\thinmuskip\relax}

\DeclareMathOperator{\Hom}{Hom}
\DeclareMathOperator{\Ext}{Ext}
\DeclareMathOperator{\End}{End}
\DeclareMathOperator{\Id}{Id}
\DeclareMathOperator{\id}{id}
\DeclareMathOperator{\coker}{coker}

\DeclareMathOperator{\rad}{rad}
\DeclareMathOperator{\tp}{top}

\newcommand{\modl}{{\operatorname{\mathsf{--mod}}}}
\newcommand{\modr}{{\operatorname{\mathsf{mod--}}}} 
\newcommand{\vect}{{\operatorname{\mathsf{--vect}}}}
\newcommand{\contra}{{\operatorname{\mathsf{--contra}}}}
\newcommand{\discr}{{\operatorname{\mathsf{discr--}}}} 

\newcommand{\Sets}{\mathsf{Sets}}
\newcommand{\Funct}{\mathsf{Funct}}
\newcommand{\Add}{\mathsf{Add}}
\newcommand{\Prod}{\mathsf{Prod}}
\newcommand{\Rex}{\mathsf{Rex}}
\newcommand{\Ab}{\mathsf{Ab}}
\newcommand{\inj}{\mathsf{inj}}
\newcommand{\proj}{\mathsf{proj}}
\newcommand{\fp}{\mathsf{fp}}

\newcommand{\rop}{{\mathrm{op}}}
\newcommand{\sop}{{\mathsf{op}}}
\newcommand{\sss}{{\mathsf{ss}}}

\newcommand{\Mat}{\mathfrak{Mat}}

\newcommand{\A}{\mathfrak A}
\newcommand{\B}{\mathfrak B}
\newcommand{\C}{\mathfrak C}
\newcommand{\D}{\mathfrak D}
\newcommand{\E}{\mathfrak E}
\newcommand{\F}{\mathfrak F}
\newcommand{\HH}{\mathfrak H}
\newcommand{\I}{\mathfrak I}
\newcommand{\J}{\mathfrak J}
\newcommand{\K}{\mathfrak K}
\newcommand{\R}{\mathfrak R}
\renewcommand{\S}{\mathfrak S}
\newcommand{\T}{\mathfrak T}
\newcommand{\U}{\mathfrak U}
\newcommand{\V}{\mathfrak V}
\newcommand{\MM}{\mathfrak M}

\newcommand{\M}{\mathcal M}
\newcommand{\N}{\mathcal N}
\newcommand{\cV}{\mathcal V}

\newcommand{\sA}{\mathsf A}
\newcommand{\sB}{\mathsf B} 
\newcommand{\sC}{\mathsf C}
\newcommand{\sD}{\mathsf D}

\newcommand{\boZ}{\mathbb Z}
\newcommand{\boQ}{\mathbb Q}
\newcommand{\boT}{\mathbb T}

\newcommand{\Section}[1]{\bigskip\section{#1}\medskip}
\theoremstyle{plain}
\newtheorem{thm}{Theorem}[section]
\newtheorem{qst}[thm]{Question}
\newtheorem{conj}[thm]{Conjecture}
\newtheorem{lem}[thm]{Lemma}
\newtheorem{prop}[thm]{Proposition}
\newtheorem{cor}[thm]{Corollary}
\theoremstyle{definition}
\newtheorem{ex}[thm]{Example}
\newtheorem{exs}[thm]{Examples}
\newtheorem{rem}[thm]{Remark}
\newtheorem{rems}[thm]{Remarks}

\begin{document}

\title{Topologically semisimple and topologically perfect
topological rings}

\author{Leonid Positselski}

\address{Leonid Positselski, Institute of Mathematics
of the Czech Academy of Sciences \\
\v Zitn\'a~25, 115~67 Prague~1 \\ Czech Republic; and
\newline\indent Laboratory of Algebra and Number Theory \\
Institute for Information Transmission Problems \\
Moscow 127051 \\ Russia} 

\email{positselski@yandex.ru}

\author{Jan \v S\v tov\'\i\v cek}

\address{Jan {\v S}{\v{t}}ov{\'{\i}}{\v{c}}ek, Charles University
in Prague, Faculty of Mathematics and Physics, Department of Algebra,
Sokolovsk\'a 83, 186 75 Praha, Czech Republic}

\email{stovicek@karlin.mff.cuni.cz}

\begin{abstract}
 Extending the Wedderburn--Artin theory of (classically) semisimple
associative rings to the realm of topological rings with right linear
topology, we show that the abelian category of left contramodules
over such a ring is split (equivalently, semisimple) if and only if
the abelian category of discrete right modules over the same ring is
split (equivalently, semisimple).
 Our results in this direction complement those of
Iovanov--Mesyan--Reyes.
 An extension of the Bass theory of left perfect rings to
the topological realm is formulated as a list of conjecturally
equivalent conditions, many equivalences and implications between
which we prove.
 In particular, all conditions are equivalent for topological rings
with a countable base of neighborhoods of zero and for topologically
right coherent topological rings.
 Considering the rings of endomorphisms of modules as topological rings
with the finite topology, we establish a close connection between
the concept of a topologically perfect topological ring
and the theory of modules with perfect decomposition.
 Our results also apply to endomorphism rings and direct sum
decompositions of objects in certain additive categories more general
than the categories of modules; we call them
\emph{topologically agreeable} categories.
 We show that any topologically agreeable split abelian category is
Grothendieck and semisimple.
 We also prove that a module $\Sigma$\+coperfect over its endomorphism
ring has a perfect decomposition provided that either the endomorphism
ring is commutative or the module is countably generated, partially
answering a question of Angeleri H\"ugel and Saor\'\i n.
\end{abstract}

\maketitle

\tableofcontents

\section*{Introduction}
\medskip

% Part 1. The main goal.
 An abelian category $\sA$ is called \emph{semisimple} if all its
objects are (possibly infinite) coproducts of simple objects.
 For the category of modules over an associative ring $\sA=S\modl$,
this can be equivalently restated as the condition that all short
exact sequences in $\sA$ split.
 This property is left-right symmetric: the category of left modules
over an associative ring $S$ is semisimple if and only if the category
of right $S$\+modules is.
 Such rings $S$ are called \emph{classically semisimple} (or
``semisimple Artinian'').
 The classical Wedderburn--Artin theorem describes them as finite
products of rings of matrices (of some finite size) over skew-fields.
  
 An associative ring $R$ is said to be \emph{left perfect} if all
flat left $R$\+modules are projective, or equivalently, all descending
chains of cyclic right $R$\+modules terminate.
 The equivalence of these two and several other conditions describing
left perfect rings was established in the classical paper of
Bass~\cite[Theorem~P]{Bas}.
 In particular, one of these conditions characterizes perfect rings by
their structural properties: a ring $R$ is left perfect if and only if
its Jacobson radical $H$ is left T\+nilpotent and the quotient ring
$R/H$ is classically semisimple.

 In this paper we consider complete, separated topological associative
rings with a right linear topology (which means that open right ideals
form a base of neighborhoods of zero).
 Our aim is to extend the theorems of Wedderburn--Artin and Bass to such
topological rings.

% End of Part 1.

\smallskip

% Part 2. Motivation.

 Right linear topological rings are ubiquitous, be it in commutative
algebra, where one often considers the $I$-adic topology induced by
an ideal $I$, the $R$\+topology introduced by Matlis on a commutative
domain (or even an arbitrary commutative ring) $R$, non-commutative
generalizations of both the adic topologies and the $R$\+topology in
the form of Gabriel topologies, or the naturally arising pseudocompact
topology on the vector-space dual of a coalgebra over a field.

 It is beyond our knowledge and beyond the scope of this introduction
to give an adequate historical overview of topological structures in
the theory of rings and modules, or in additive categories.
 The reader will find small elements of such historical discussion
below in this introduction, as well as in
Remarks~\ref{Q/R-endomorphisms-remark} and~\ref{almost-perfect-remark}
and elsewhere in the main body of the paper.

 The key class of examples for us is provided
by endomorphism rings of (possibly infinitely generated) modules.
A classical construction equips any such endomorphism ring with
the \emph{finite topology}. In fact, a similar construction applies
to endomorphism rings in any locally finitely generated Grothendieck
category, providing us with even more examples.
Conversely, as we explain in this paper, every complete, separated topological ring
with right linear topology can be obtained as
the ring of endomorphisms of a module equipped with the finite topology.

 This directly relates the theory of topological
rings to the theory of (decompositions of) modules over ordinary rings,
objects in locally finitely generated Grothendieck categories or, more
generally, objects in what we call \emph{topologically agreeable} categories.
Most remarkably, however, this interaction between the theory of
topological rings with right linear topology and the theory of direct
sum decompositions of modules brings applications in both directions.

% End of Part 2.

\smallskip

% Part 3. Summary of results on topologically semisimple rings.

 One of the distinctive features of working with
right linear topological rings, as opposed to ordinary rings, is that 
a symmetry between the categories of left and right modules is lost
from the outset.
With a complete, separated topological associative ring $\S$ with right
linear topology, we associate the abelian category $\S\contra$ of left
\emph{$\S$\+contramodules} and the abelian category $\discr\S$ of
\emph{discrete} right \emph{$\S$\+modules}.
These are two abelian categories of quite different nature: while
$\discr\S$ is a hereditary pretorsion class in $\modr\S$ and
a Grothendieck abelian category, $\S\contra$ is a locally presentable
abelian category with enough projective objects~\cite{PR}.

 Then it turns that $\S\contra$ is a semisimple Grothendieck abelian
category if and only if the abelian category $\discr\S$ is semisimple.
 Moreover, the above two equivalent properties of a topological ring $\S$
are also equivalent to the seemingly weaker conditions that all short
exact sequences are split in $\S\contra$, or that all short exact
sequences are split in $\discr\S$.
 Topological rings $\S$ satisfying these equivalent conditons are called
\emph{topologically semisimple}.
 We describe them as the infinite topological products of the topological
rings of infinite-sized, row-finite matrices over skew-fields.

 An extension of the Wedderburn--Artin theory to topological rings
was also studied in the paper of Iovanov, Mesyan, and Reyes~\cite{IMR},
and the same class of topological rings (up to the passage to
the opposite ring) was obtained as the result, characterized by a list
of equivalent conditions different from ours.
 The authors of~\cite{IMR} discuss pseudo-compact modules, while we
prefer to consider discrete modules.
 There are no contramodules in~\cite{IMR}, so the topological
semisimplicity is described in~\cite{IMR} in terms of modules on one
side only, while we have both left and right modules of two different
kinds.
 There are many equivalent characterizations of topologically
semisimple topological rings in~\cite[Theorem~3.10]{IMR}, with
the proof of the equivalence substantially based on the preceding
results of the book of Warner~\cite{War}.
 So our results on topologically semisimple topological rings complement
those of Iovanov, Mesyan, and Reyes by providing further conditions
equivalent to the ones on their list.

% End of Part 3.

\smallskip

% Part 4. Summary of results on topologically perfect rings.

 Extending Bass' theory of left perfect rings to the topological realm
is a harder task, at which we only partially succeed.
 Given a complete, separated topological associative ring $\R$ with
right linear topology, we show that projectivity of all flat left
$\R$\+contramodules implies the descending chain condition for cyclic
discrete right $\R$\+modules.
 The converse implication is equivalent to a positive answer to
a certain open question in the theory of direct sum decompositions of
modules (Question~\ref{as-main-question} below), as we explain, and
we show that it holds under various additional assumptions
including the cases where
\begin{enumerate}
\item the topological ring $\R$ has a countable base of
neighborhoods of zero (proved in this paper),
\item the topological ring $\R$ is topologically right coherent
(this is an interpretation of a result by Roos in~\cite{Ro}, see below), or
\item the underlying ring of $\R$ is commutative
(this follows from results in~\cite{Pproperf}).
\end{enumerate}

Using a combination of contramodule-theoretic techniques developed
in the paper~\cite{Pproperf} with results on direct sum decompositions
of modules, we further show that
all flat left contramodules over a topological ring $\R$ are projective
if and only if $\R$ has a certain set of structural properties.
Namely, the topological Jacobson radical $\HH$ of the topological
ring $\R$ has to be \emph{topologically left T\+nilpotent} and
\emph{strongly closed in\/ $\R$}, and the topological quotient ring
$\S=\R/\HH$ needs to be topologically semisimple.
We call such (complete, separated, right linear) topological rings
$\R$ \emph{topologically left perfect}.

 The classical Govorov--Lazard theorem tells that flat modules over
a ring are precisely the direct limits of (finitely generated)
projective modules.
In the contramodule context, it is easy to prove that direct limits
of projective contramodules are flat, but it is not known whether
an analogue of the Govorov--Lazard theorem holds.%
\footnote{A counterexample is available now: \cite[Example~10.2]{PPT}.}
Nevertheless, we prove that \emph{if} the class of projective left
$\R$\+contramodules is closed under direct limits, \emph{then} all
flat left $\R$\+contramodules are projective.

 Further conditions on a topological ring $\R$ which, as we show,
are equivalent to the topological perfectness, are formulated in
terms of \emph{projective covers} in the abelian category $\R\contra$.
In fact, the main result of the paper~\cite{BPS} tells that
a direct limit of projective contramodules is projective whenever
it has a projective cover.
On the other hand, one can show that all left contramodules over
a topologically left perfect topological ring have projective covers.
Thus all flat left $\R$\+contramodules are projective if and only if
all flat left $\R$\+contramodules have projective covers, if and only
if all left $\R$\+contramodules have projective covers, and if and
only if the topological ring $\R$ is topologically left perfect.

% small historical discussion

\smallskip

 It is worth mentioning that we essentially never consider topological
modules in this paper, but only topological rings and topologies on
additive categories.
 One reason for that is because topological modules rarely form
abelian categories (pseudo-compact modules are a notable exception,
but we prefer to work with discrete modules).
 The following historical examples illustrate the point.

 Harrison~\cite{Harr} called an abelian group $C$ ``co-torsion'' if
$\Hom_\boZ(\boQ,C)=0=\Ext^1_\boZ(\boQ,C)$ (in the present-day language,
such abelian groups would be called ``reduced cotorsion'').
 The category of abelian groups with these properties is abelian.
 Matlis~\cite{Mat} extended this definition to modules over
a commutative domain~$R$: an $R$\+module $C$ was called ``cotorsion''
in~\cite{Mat} if $\Hom_R(Q,C)=0=\Ext^1_R(Q,C)$, where $Q$ is the field
of fractions of $R$ (in the modern language, such modules might be
called ``h\+reduced Matlis cotorsion'').
 The category of $R$\+modules with these properties is abelian
whenever the projective dimension of the $R$\+module $Q$ does not
exceed~$1$; such commutative domains are now known as \emph{Matlis
domains}.

 The so-called ``$R$\+topology'' on the domain $R$ and all $R$\+modules
was also defined and studied in the memoir~\cite{Mat}.
 One of the main results~\cite[Theorem~6.10]{Mat} was that a bounded
torsion $R$\+module is cotorsion if and only if it is complete in
the $R$\+topology.
 This result was extended to arbitrary commutative rings $R$
in the book~\cite[Corollary~2.3]{Mat2} and further to commutative rings
$R$ with a fixed multiplicative subset $S\subset R$ in
the paper~\cite[Theorem~2.5]{PMat}, where the localization $S^{-1}R$
plays the role of the field/ring of fractions~$Q$.
 The ``$S$\+topology'' on any $R$\+module $M$ (including the ring $R$
itself) is defined by the rule that the submodules $sM\subset M$,
where $s\in S$, form a base of neighborhoods of zero.

 In the paper~\cite{PMat}, the ``$S$\+contramodules'' terminology is
used for what Matlis called cotorsion modules.
 These form a different, but related category of contramodules as
compared to the ones studied in the present paper.
 An $R$\+module with bounded $S$\+torsion is an $S$\+contramodule if
and only if it is complete in the $S$\+topology.
 When the projective dimension of the $R$\+module $S^{-1}R$ does not
exceed~$1$, the category of $S$\+contramodules is
abelian~\cite[Theorem~3.4]{PMat} (while the category of $R$\+complete
or $S$\+complete modules is usually not abelian).

Thus, there is a strong connection between topological completeness
of modules and certain homological properties of modules in this context,
but the homological conditions are more convenient to work with, as
they define better-behaved module categories.
% Contramodules over topological rings, which we study in this paper,
%are almost generalizations of such classes of modules defined by
%homological conditions (see~\cite[Examples~2.4(3) and~5.4(2)]{Pper} for
%a discussion of the comparison).
 Contramodules over topological rings,
which we study in this paper, can be viewed as one step further
in this direction. Their categories in the above cases
are very closely related, and often equivalent, to classes
of modules defined by such homological conditions
(see~\cite[Examples~2.4(3) and~5.4(2)]{Pper} for
a discussion of the comparison).
Their advantage, however, is a broader generality and the fact that
the concept is intrinsic to a complete, separated topological
ring with right linear topology, without any reference to specific
classes of modules.

% End of Part 4.

\smallskip

% Part 5. Connections (to direct sum decomposition theory and others).

 Let us now look closer at the connection of our results to the existing
literature and open problems in module theory and category theory.
 The theory of direct sum decompositions of modules goes back to
the classical Krull--Schmidt--Remak--Azumaya uniqueness
theorem~\cite[Section~V.5]{St}, \cite[Section~2]{Fac}.
There exists an extensive literature on this subject now, with lots of
known results and open problems~\cite{CJ,Har,MM,GG2,GT}.
The fact that the finite topology on the endomorphism ring of a module
is relevant in the study of its direct sum decompositions is
well understood~\cite{CN}, but contramodules over the topological rings
of endomorphisms have not been used in such studies yet.
This is a new technique which we bring to bear on the subject
(using the approach originally developed in our previous
paper~\cite{PS}).

In particular, the above-mentioned open question was posed by
Angeleri H\"ugel and Saor\'\i n as~\cite[Question~1 in Section~2]{AS}:

\begin{qst} \label{as-main-question}
Let $A$ be an associative ring and $M$ be an $A$\+module.
Denote by $R$ the ring of $A$\+linear endomorphisms of~$M$.
Assume that $M$ is\/ $\Sigma$\+coperfect over~$R$.
Does it follow that the $A$\+module $M$ has a perfect decomposition?
\end{qst}

Here an $R$\+module is said to be \emph{coperfect} if it satisfies
the descending chain condition on cyclic (equivalently, finitely
generated) $R$\+submodules.
An $R$\+module $M$ is \emph{$\Sigma$\+coperfect} if the countable direct
sum $M^{(\omega)}$ of copies of $M$ is a coperfect $R$\+module.
It follows from our results that the answer to
Question~\ref{as-main-question} is positive whenever either the ring
$R$ is commutative, or it is topologically right coherent in
the finite topology, or the $A$\+module $M$ is countably generated.

It was shown in the same paper~\cite[Theorem~1.4]{AS} that
an $A$\+module $M$ has perfect decomposition if and only if, for
any direct system of $A$\+modules $M_i\in\Add(M)$ indexed by
a linearly ordered set of indices~$i$, the natural surjective
$A$\+module morphism $\bigoplus_i M_i\rarrow\varinjlim_i M_i$ is split.
In this paper we extend this result to objects of (above mentioned)
topologically agreeable additive categories.

We stress again that interaction between the theory of
topological rings with right linear topology and the theory of direct
sum decompositions of modules brings applications in both directions.
In particular, the above-mentioned results concerning flat
contramodules are applications of module theory to topological algebra.
Another such application is the theorem that any split abelian
category admitting a topologically agreeable structure is Grothendieck 
and semisimple.
Our results concerning the question of Angeleri H\"ugel and Saor\'\i n,
on the other hand, are applications of topological algebra to
module theory.

 Applications of topological rings to abelian categories,
more specifically to certain classes of Grothendieck abelian categories,
were initiated by Gabriel already in his dissertation~\cite{Gab}.
 This approach was further developed by Roos in~\cite{Ro}.
 In fact, Gabriel described locally finite Grothendieck categories
in terms of pseudo-compact
topological rings~\cite[n$^{\mathrm{os}}$~IV.3\<4]{Gab}.
 Roos described locally Noetherian Grothendieck categories and their
conjugate counterparts, the locally coperfect locally coherent
Grothendieck categories, in terms of topologically coperfect and
coherent topological rings~\cite[Theorem~6]{Ro}.
 It follows from Roos' result that all flat left $\R$\+contramodules are
projective whenever the category of discrete right $\R$\+modules is
locally coherent and satisfies the descending chain condition for
coherent (or finitely generated, or cyclic) objects/modules.
This is what is behind one of the cases above where
the descending chain condition on cyclic discrete right $\R$\+modules
characterizes topological perfectness.

Finally, the theory of topological rings provides new cases
where the following open problem due to Enochs (\cite[Section~5.4]{GT})
can be answered affirmatively:

\begin{qst} \label{enochs-question}
Let $A$ be an associative ring and $\sC$ be a covering class of modules. Is $\sC$ closed under direct limits?
\end{qst}

The idea is again to study the problem for classes of the form $\sC=\Add(M)$, where $M$ is a module, and translate the problem under suitable hypotheses to the question whether suitable contramodules over the endomorphism ring $\R=\End(M)^\sop$ have projective covers and, thus, whether $\R$ is topologically left perfect. This direction is already beyond the scope of this paper and we refer to~\cite{BP2,BPS} for a detailed account.

% End of Part 5.

\medskip\noindent
\textbf{Acknowledgement.}
 The authors are grateful to Jan Trlifaj, Pavel P\v r\'\i hoda,
Pace Nielsen, and Manuel Reyes for very helpful discussions and
communications.
 The first-named author is supported by research plan
RVO:~67985840. 
 The second-named author was supported by the Czech Science
Foundation grant number 17-23112S.

\Section{Preliminaries on Topological Rings}

 We mostly refer to~\cite[Section~2]{Pproperf}
(see also~\cite[Section~2]{Pcoun}) and the references therein
for the preliminary material, so the section below only contains
a brief sketch of the key definitions and constructions.

 All \emph{topological abelian groups} in this paper are presumed to
have a base of neighborhoods of zero consisting of open subgroups.
 Subgroups, quotient groups, and products of topological groups are
endowed with the induced/quotient/product topologies.
 The \emph{completion} of a topological abelian group $A$ is
the abelian group $\A=\varprojlim_U A/U$ (where $U$ ranges over
the open subgroups in~$A$) endowed with
the \emph{projective limit topology}.
 A topological abelian group $A$ is \emph{complete} if the completion
morphism $A\rarrow\A$ is surjective, and \emph{separated} if this map
is injective.
 The completion $\A$ of any topological abelian group $A$ is complete
and separated~\cite[Sections~2.1\<2 and~8]{Pproperf}.

 Unless otherwise mentioned, all \emph{rings} are presumed to be
associative and unital.
 The Jacobson radical of a ring $R$ is denoted by $H(R)$.
 A topological ring $R$ is said to have a \emph{right linear topology}
if open right ideals form a base of neighborhoods of zero in~$R$.
 The completion $\R$ of a topological ring $R$ with right linear
topology is again a topological ring with right linear topology, and
the completion map $R\rarrow\R$ is a continuous ring
homomorphism~\cite[Section~2.3]{Pproperf}.

 Let $R$ be a topological ring with right linear topology and $\R$
be the completion of~$R$.
 A right $R$\+module $\N$ is said to be \emph{discrete} if
the annihilator of every element of $\N$ is an open right ideal in~$R$.
 The full subcategory of discrete right $R$\+modules $\discr R
\subset\modr R$ is closed under submodules, quotients, and infinite
direct sums in the category of right $R$\+modules $\modr R$;
so $\discr R$ is a Grothendieck abelian category.
 The categories of discrete right modules over a topological ring $R$
and its completion $\R$ are naturally equivalent, $\discr R\cong
\discr\R$ \,\cite[Section~2.4]{Pproperf}.

 Given an abelian group $A$ and a set $X$, we use the notation
$A[X]=A^{(X)}$ for the direct sum of $X$ copies of the group~$A$.
 For a complete, separated topological abelian group~$\A$, we put
$\A[[X]]=\varprojlim_{\U\subset\A}(\A/\U)[X]$, where the projective
limit is taken over all the open subgroups $\U\subset\A$.
 The set $\A[[X]]$ is interpreted as the set of all \emph{infinite
formal linear combinations} $\sum_{x\in X} a_xx$ of elements of $X$
with the coefficients $a_x\in\A$ forming an $X$\+indexed family
of elements $(a_x)_{x\in X}$ \emph{converging to zero} in
the topology of $\A$.
 The latter condition means that, for every open subgroup $\U\subset\A$,
the set of all $x\in X$ for which $a_x\notin\U$ is finite.
 A closed subgroup $\U$ in a complete, separated topological abelian
group $\A$ is said to be \emph{strongly closed} if the quotient
group $\A/\U$ is complete and the map $\A[[X]]\rarrow(\A/\U)[[X]]$
induced by the morphism $\A\rarrow\A/\U$ is surjective for every
set~$X$ \,\cite[Sections~2.5 and~2.11]{Pproperf}.

 The assignment of the set $\A[[X]]$ to a set $X$ is naturally extended
to a covariant functor $X\longmapsto\A[[X]]\:\Sets\rarrow\Sets$
from the category of sets to itself (or, if one wishes, to the category
$\Ab$ of abelian groups).
 Given a complete, separated topological associative ring $\R$ with
right linear topology, the functor $X\longmapsto\R[[X]]$ has
a natural structure of a \emph{monad} on the category of sets
\cite[Section~2.6]{Pproperf}.
 This means that for every set $X$, one also considers a natural map
$\epsilon_X\:X\rarrow\R[[X]]$ (called the ``point measure'' map
and defined in terms of the zero and unit elements in $\R$) and
a natural map $\phi_X\:\R[[\R[[X]]]]\rarrow\R[[X]]$ (called
the ``opening of parentheses'' map and defined in terms of
the multiplication of pairs of elements and infinite sums of
zero-converging families of elements in~$\R$), which satisfy
certain associativity and unitality conditions.

 A \emph{left contramodule} over a topological ring $\R$ is, by
the definition, a module (or, in the more conventional terminology,
an ``algebra'') over this monad on $\Sets$.
 In other words, a left $\R$\+contramodule $\C$ is a set endowed
with a \emph{left contraaction} map $\pi_\C\:\R[[\C]]\rarrow\C$
satisfying the associativity and unitality equations with respect
to the natural transformations $\phi$ and~$\epsilon$.
 The \emph{free} left $\R$\+contramodule $\R[[X]]$ spanned by a set $X$
is the free module/algebra over the monad $X\longmapsto\R[[X]]$ on
$\Sets$.
 The category of left $\R$\+contramodules $\R\contra$ is a locally
presentable abelian category with enough projective objects; the latter
are precisely the direct summands of the free left $\R$\+contramodules
$\R[[X]]$.
 The free left $\R$\+contramodule with one generator $\R=\R[[*]]$ is
a projective generator of the abelian category $\R\contra$.
 The underlying set of a left $\R$\+contramodule carries a natural
left $\R$\+module structure, which provides a faithful, exact,
limit-preserving forgetful functor $\R\contra\rarrow\R\modl$
\,\cite[Section~2.7]{Pproperf}, \cite[Sections~1.1\<2 and~5]{PR}.

 A class of examples of left $\R$\+contramodules is constructed by
dualizing discrete right $\R$\+modules.
 Let $A$ be an associative ring and $\N$ be an $A$\+$\R$\+bimodule
whose underlying right $\R$\+module is discrete.
 Let $V$ be a left $A$\+module.
 Then the abelian group $\D=\Hom_A(\N,V)$ has a natural structure of
left $\R$\+contramodule with the contraaction map given by the rule
$$
 \pi_\D\left(\sum\nolimits_{d\in\D}r_dd\right)(b)
 =\sum\nolimits_{d\in\D}d(br_d)
 \qquad\text{for all $b\in\N$}.
$$
 Here the sum in the right-hand side is finite because the annihilator
of~$b$ is open in $\R$ and the family of elements $(r_d\in\R)_{d\in\D}$
converges to zero in~$\R$ \,\cite[Section~2.8]{Pproperf},
\cite[Section~2.8]{Pcoun}.

 The \emph{contratensor product} $\N\ocn_\R\C$ of a discrete right
$\R$\+module $\N$ and a left $\R$\+contramodule $\C$ is an abelian
group constructed as the cokernel of (the difference of) the natural
pair of abelian group homomorphisms $\N\ot_\boZ\R[[\C]]
\rightrightarrows\N\ot_\boZ\C$.
 Here one of the two maps $\N\ot_\boZ\R[[\C]]\rarrow\N\ot_\boZ\C$ is
simply induced by the contraaction map $\pi_\C\:\R[[\C]]\rarrow\C$,
while the other one is constructed in terms of the right action of
$\R$ in $\N$ (using the assumption that this right action is
\emph{discrete} in combination with the description of $\R[[\C]]$
as the set of all formal linear combinations of elements of $\C$
with \emph{zero-convergent} families of coefficients in~$\R$).
 The contratensor product $\N\ocn_\R\C$ is, generally speaking,
a quotient group of the tensor product $\N\ot_\R\C$.

 For any $A$\+$\R$\+bimodule $\N$ whose underlying right $\R$\+module
is discrete, any left $A$\+module $V$, and any left $\R$\+contramodule
$\C$, there is a natural isomorphism of abelian
groups~\cite[Section~2.8]{Pproperf}, \cite[Section~2.8]{Pcoun}
$$
 \Hom^\R(\C,\Hom_A(\N,V))\,\cong\,\Hom_A(\N\ocn_\R\C,\>V),
$$
where $\Hom^\R(\C,\D)$ is the notation for the group of all
morphisms $\C\rarrow\D$ in the abelian category $\R\contra$.

 For any left $\R$\+contramodule $\C$ and any set $X$, there is
a natural isomorphism of abelian groups~\cite[Section~2.7]{Pproperf}
$$
 \Hom_\R(\R[[X]],\C)\,\cong\,\Hom_\Sets(X,\C).
$$
 For any discrete right $\R$\+module $\N$ and any set $X$, there is
a natural isomorphism of abelian groups~\cite[Section~2.8]{Pproperf}
$$
 \N\ocn_\R\R[[X]]\,\cong\,\N[X]=\N^{(X)}.
$$

\Section{Split and Semisimple Abelian Categories}

 A nonzero object in an abelian category is \emph{simple} if it has
no nonzero proper subobjects.
 An object is \emph{semisimple} if it is a coproduct of simple objects.
 An abelian category is called \emph{semisimple} if all its objects
are semisimple.
 We will say that an abelian category $\sA$ is \emph{split} if all
short exact sequences in $\sA$ split.

\begin{lem}
 Every semisimple abelian category is split.
\end{lem}

\begin{proof}
 The following proof was suggested to us by J.~Rickard~\cite{R-MO}.
 Let $\sA$ be a semisimple abelian category, and let $f\:X\rarrow Y$
be an epimorphism in~$\sA$.
 By assumption, the object $Y$ is a coproduct of a family of simple
objects, $Y=\coprod_i S_i$.
 Let $f_i\:X_i\rarrow S_i$ be the pullback of the morphism~$f$ along
the split monomorphism $S_i\rarrow Y$.
 In order to show that the epimorphism~$f$ has a section, it suffices
to check that so does the epimorphism~$f_i$ for every index~$i$
(see~\cite[Proposition~A.1]{CS}).

 By assumption, the object $X_i$ is a coproduct of a family of simple
objects, too: $X_i=\coprod_j T_{ij}$.
 The morphism $f_i\:X_i\rarrow S_i$ corresponds to a family of
morphisms $f_{ij}\:T_{ij}\rarrow S_i$.
 By the Schur Lemma, every morphism~$f_{ij}$ is either zero or
an isomorphism.
 If $f_{ij}=0$ for every~$j$, then $f_i=0$, which is impossible for
an epimorphism~$f_i$ with a nonzero codomain~$S_i$.
 Hence there exists an index~$j$ for which $f_{ij}$ is an isomorphism.
 Now the composition of the inverse morphism $f_{ij}^{-1}:S_i
\rarrow T_{ij}$ with the split monomorphism $T_{ij}\rarrow X_i$ provides
a section of the epimorphism~$f_i$.
\end{proof}

 The material below in this section is essentially well-known.
 We include it for the sake of completeness of the exposition.

 An abelian category is \emph{Ab5} if it is cocomplete and has exact
functors of direct limits (\,$=$~filtered colimits).
 A \emph{Grothendieck category} is an abelian category which is Ab5
and has a set of generators (or equivalently, a single generator).

 A split Grothendieck abelian category is called
\emph{spectral}~\cite{GO}.

\begin{rem} \label{spectral-remark}
 The theory of spectral categories is surprisingly complicated
(as compared to the na\"\i ve expectation that all split abelian
categories should be semisimple).
 The terminology ``spectral category'' refers to the spectral theory
of operators in infinite-dimensional topological vector spaces (in
functional analysis); in particular, a reference to the functional
analysis concepts of ``discrete'' and ``continuous'' spectrum
is presumed.
 A spectral category is called \emph{discrete} if it is semisimple,
and \emph{continuous} if it has no simple objects.
 Every spectral category has a unique decomposition into
the Cartesian product of a discrete and a continuous spectral
category~\cite[Section~V.6]{St}.
 Furthermore, spectral categories $\sA$ with a chosen generator $G$
correspond bijectively to left self-injective von~Neumann
regular rings $R$ in the following way.
 To a pair $(\sA,G)$ the (opposite ring to) the ring of endomorphisms
of the generator, $R=\Hom_\sA(G,G)^\rop$, is assigned (so $R$ acts
in $G$ on the right).
 To a left self-injective von~Neumann regular associative ring $R$,
the full subcategory $\sA=\Prod(R)\subset R\modl$ of all the direct
summands of products copies of the (injective) free left $R$\+module $R$ 
is assigned, with the chosen generator $G=R$.
 The category $\sA$ can be also interpreted as a quotient category
of $R\modl$ \,\cite[Section~XII.1]{St}.
 A good exposition of some aspects of this theory can be found
in~\cite[Chapter~I]{GB}.
 A concrete class of examples of continuous spectral categories is 
described below in Example~\ref{complete-boolean}.
\end{rem}

 The following theorem characterizes and describes the semisimple
Grothendieck (\,$=$~discrete spectral) abelian categories
(cf.~\cite[Proposition~V.6.7]{St}).

\begin{thm} \label{semisimple-category}
 Let $\sA$ be an abelian category with set-indexed coproducts and
a generator.
 Then the following conditions are equivalent:
\begin{enumerate}
\item $\sA$ is Ab5 and every object of $\sA$ is the sum of its simple
subobjects;
\item $\sA$ is Ab5, split, and every nonzero object of $\sA$ has
a simple subquotient object;
\item every object of\/ $\sA$ is a coproduct of simple objects, and for
every simple object $S\in\sA$ the functor\/ $\Hom_\sA(S,{-})\:\sA
\rarrow\Ab$ preserves coproducts;
\item there is a set $X$ and an $X$\+indexed family of division rings
(\,$=$~skew-fields) $D_x$, \,$x\in X$, such that the category\/ $\sA$ is
equivalent to the Cartesian product of the categories of vector spaces
over~$D_x$,
$$
 \sA\,\cong\,\mathop{\text{\huge $\times$}}_{x\in X} D_x\modl.
$$
\end{enumerate}
\end{thm}
 
\begin{rem} \label{discrete-modules-split-iff-semisimple}
 The third condition in Theorem~\ref{semisimple-category}\,(2), saying
that every nonzero object of $\sA$ has a simple subquotient, is always
satisfied for the abelian category of modules over an associative ring
$\sA=A\modl$ (because every nonzero $A$\+module has a nonzero cyclic
submodule, which in turn has a maximal proper submodule).
 Thus the category $A\modl$ is split/spectral if and only if it is
semisimple.
 Moreover, for any topological ring $R$ with right linear topology,
the same applies to the abelian category $\sA=\discr R$ of discrete
right $R$\+modules, which is also Grothendieck and has the property
that every nonzero object has a simple subquotient.
\end{rem}

 The following lemma shows that in Ab5-categories, as in the categories
of modules, the sum of a family of subobjects is direct whenever
the sum of any finite subfamily of these objects is.

\begin{lem} \label{Ab5-direct-sum}
 Let $\sA$ be an Ab5-category, $M\in\sA$ be an object, and
$(N_x\subset M)_{x\in X}$ be a family of subobjects in~$M$.
 Suppose that for every finite subset $Z\subset X$ the induced
morphism\/ $\coprod_{z\in Z} N_z\rarrow M$ is a monomorphism.
 Then the induced morphism\/ $\coprod_{x\in X} N_x\rarrow M$ is
a monomorphism, too.
\end{lem}

\begin{proof}
 In any cocomplete category, one has $\coprod_{x\in X} N_x=
\varinjlim_{Z\subset X}\coprod_{z\in Z} N_z$ (where the direct
limit is taken over all the finite subsets $Z\subset X$).
 In an Ab5-category, the direct limit of a diagram of monomorphisms
is a monomorphism.
\end{proof}

\begin{cor}[{\cite[Proposition~V.6.2]{St}}] \label{simple-Zorn}
 Let\/ $\sA$ be an Ab5-category, $M\in\sA$ be an object, and
$L\subset M$ be a suboboject.
 Assume that $M$ is the sum of a family $(S_x\subset M)_{x\in X}$
of simple subobojects of $M$ (i.~e., no proper subobject of $M$
contains $S_x$ for all $x\in X$).
 Then there exists a subset $Y\subset X$ such that $M=L\oplus
\coprod_{y\in Y}S_y$.
\end{cor}

\begin{proof}
 Provable by a standard Zorn lemma argument based on
Lemma~\ref{Ab5-direct-sum}.
\end{proof}

 The next lemma says that simple objects in Ab5-categories are
\emph{finitely generated} in the sense of~\cite[Section~V.3]{St}
or~\cite[Section~1.E]{AR}.

\begin{lem} \label{simple-finitely-generated}
 Let\/ $\sA$ be an Ab5-category and $S\in\sA$ be a simple object.
 Then the functor\/ $\Hom_\sA(S,{-})\:\sA\rarrow\Ab$ preserves
direct limits of diagrams of monomorphisms.
\end{lem}

\begin{proof}
 Let $(m_{z,w}\:M_w\to M_z)_{w<z}$ be a diagram of objects in $\sA$
and monomorphisms between them, indexed by some directed poset~$Z$.
 Set $M=\varinjlim_{z\in Z}M_z$.
 Since $\sA$ is Ab5, the natural morphisms $M_z\rarrow M$ are also
monomorphisms; so one can consider the objects $M_z$ as subobjects
in~$M$.
 Furthermore, the condition that $\sA$ is Ab5 can be equivalently
expressed by saying that for any subobject $K\subset M$ one has
$K=\varinjlim_{z\in Z}(K\cap M_z)$ \,\cite[Section~III.1]{Mit}.

 Now let $f\:S\rarrow M$ be a morphism in~$\sA$ and $f(S)\subset M$
be its image.
 Since $S$ is simple, $f(S)$ is either simple or zero, and it
follows that for every $z\in Z$ the intersection $f(S)\cap M_z$
is either the whole $f(S)$ or zero.
 We have $f(S)=\varinjlim_{z\in Z}f(S)\cap M_z$; so if $f(S)\cap M_z=0$
for all $z\in Z$, then $f=0$ and there is nothing to prove.
 Otherwise, there exists $z\in Z$ such that $f(S)\subset M_z$.
 Hence the morphism~$f$ factorizes through the monomorphism
$M_z\rarrow M$, as desired.
\end{proof}

 In particular, it follows from Lemma~\ref{simple-finitely-generated}
that simple objects in Ab5-categories are \emph{weakly finitely
generated} in the sense of~\cite[Section~9.2]{PS}.

\begin{cor} \label{simple-wfg}
 Let\/ $\sA$ be an Ab5-category and $S\in\sA$ be a simple object.
 Then the functor\/ $\Hom_\sA(S,{-})\:\sA\rarrow\Ab$ preserves
coproducts. \qed
\end{cor}

\begin{proof}[Proof of Theorem~\ref{semisimple-category}]
 (1)\,$\Longrightarrow$\,(2)\,\&\,(3) follows from
Corollaries~\ref{simple-Zorn} and~\ref{simple-wfg}.

 (2)\,$\Longrightarrow$\,(1) Given an object $M\in\sA$, consider
its socle (\,$=$~the sum of all simple subobjects) $N\subset M$.
 If $N\ne M$, then the quotient object $M/N$ has a simple
subquotient object.
 Since $\sA$ is split, this leads to a simple subobject in $M$
not contained in~$N$.
 The contradiction proves that $N=M$.

 (3)\,$\Longrightarrow$\,(4) It is important here that in any
abelian category with a generator the isomorphism classes
of simple objects form a set.
 Indeed, if $G\in\sA$ is a generator and $S\in\sA$ is simple, then
there exists a nonzero morphism $G\rarrow S$, so $S$ is a quotient
of~$G$.
 Now the subobjects of $G$ (hence also the quotient objects of~$G$)
form a set of the cardinality not exceeding that of the powerset
of $\Hom_\sA(G,G)$.
 (If $\sA$ is split, all subobjects and quotient objects of $G$ are
direct summands, and their cardinality does not exceed that of
the set of all idempotent endomorphisms of~$G$.)

 Now let $(S_x)_{s\in X}$ be a set of representatives of all
the isomorphism classes of simple objects in~$\sA$.
 Put $D_x=\Hom_\sA(S_x,S_x)^\rop$ (by the Schur lemma, $D_x$~are
division rings).
 The desired equivalence of categories is provided by the functor
$F\:\sA\rarrow\mathop{\text{\Large $\times$}}_{x\in X} D_x\modl$
taking an object $M\in\sA$ to the collection of left $D_x$\+vector
spaces $(\Hom_\sA(S_x,M))_{x\in X}$.
 The inverse functor $G$ takes a collection of left $D_x$\+vector spaces
$(V_x\in D_x\modl)_{x\in X}$ to the object $\coprod_{x\in X}
(S_x\ot_{D_x}V_x)\in\sA$ (where $S_x\ot_{D_x}{-}$ is a functor
taking the coproduct $D_x^{(Y)}$ of $Y$ copies of $D_x\in D_x\modl$ to
the coproduct $S_x^{(Y)}\in\sA$ of $Y$ copies of the object $S_x$,
for any set~$Y$).

 Essentially, the first condition in~(3) describes the objects of
the category $\sA$, and the second condition fully describes its
morphisms.
 This allows to prove that the functors $F$ and $G$ are mutually
inverse equivalences.
 
 The implication (4)\,$\Longrightarrow$\,(1) is obvious.
\end{proof}

\begin{ex} \label{complete-boolean}
 Let us give an explicit example (or, rather, a class of examples) of
\emph{continuous} spectral categories $\sA$ with a chosen generator~$G$.
 These examples have an additional advantage that the ring
$R=\Hom_\sA(G,G)^\rop$ is \emph{commutative}.

 Following~\cite[Section~XII.1]{St}, \cite[Chapter~I]{GB},
and/or Remark~\ref{spectral-remark} above, spectral categories $\sA$
with a chosen generator $G$ are described by left self-injective
von~Neumann regular rings~$R$.
 Isomorphism classes of simple objects in $\sA$ correspond to ring
direct factors in $R$ isomorphic to the (opposite ring of) the ring of
endomorphisms of a vector space over a skew-field (the dimension of
the vector space being equal to the multiplicity with which the simple
object occurs in the chosen generator).
 So continuous spectral categories $\sA$ are described by left
self-injective von~Neumann regular rings $R$ which have no such
ring direct factors.
 In particular, if we want the ring $R$ to be commutative, then such
continuous spectral categories $\sA$ with a generator $G$ are described
by self-injective commutative von~Neumann regular rings $R$ such that
no ring direct factor of $R$ is a field.

 Now let us restrict to the following particular case.
 Let $R$ be a Boolean ring, that is an associative unital ring in
which all the elements are idempotent.
 Then $R$ is a commutative algebra over the field $\boZ/2\boZ$ and
a von~Neumann regular ring.
 Furthermore, there is a natural partial order on $R$ which makes $R$ 
a distributive lattice with complements (this structure is called
a \emph{Boolean algebra})~\cite[Section~III.4]{St}, \cite{GH}.
 According to~\cite[Section~XII.3]{St}, a Boolean ring is
self-injective if and only if its Boolean algebra is \emph{complete}
(that is, complete as a lattice).
 A Boolean ring has no field direct factors (i.~e., ring direct
factors isomorphic to $\boZ/2\boZ$) if and only if its Boolean algebra
has no \emph{atoms}.

 A discussion of complete Boolean algebras can be found
in~\cite[Chapter~38]{GH}; they are classified by \emph{extremally
disconnected} compact Hausdorff topological spaces (which means
compact Hausdorff topological spaces in which the closure of every
open subset is open; to such a space $Z$, the algebra of
all its clopen subsets is assigned).
 Moving to a specific example of a complete Boolean algebra, choose
a Hausdorff topological space $X$ without isolated points, and
consider the Boolean algebra/ring $R$ of all open or closed subsets
in $X$ considered \emph{up to nowhere dense subsets}.
 The key observation is that, viewing nowhere dense subsets
as negligible, there is \emph{no difference} between open and closed
subsets (since for any open subset $U\subset X$ with the closure
$\overline{U}\subset X$, the complement $\overline{U}\setminus U$ is
nowhere dense in~$X$).
 Alternatively, the usual approach is to consider \emph{regular
open sets}, i.~e., open subsets in $X$ which coincide with
the interior of their closure~\cite[Chapter~10]{GH}.
 These form the desired complete Boolean ring $R$ without atoms,
which is consequently self-injective commutative von~Neumann regular
without field direct factors.
\end{ex}

\Section{Topologically Agreeable Additive Categories}
\label{topologically-agreeable-secn}

 The following definitions and construction were suggested in
the manuscript~\cite{Cor}.
 Let $\sA$ be an additive category with set-indexed coproducts.
 If set-indexed products exist in the category $\sA$, one says
that $\sA$ is \emph{agreeable} if for every family of
objects $N_x\in\sA$ (indexed by elements~$x$ of some set~$X$)
the natural morphism from the coproduct to the product
$\coprod_{x\in X} N_x\rarrow \prod_{x\in X} N_x$ is
a monomorphism in~$\sA$.

 This condition can be reformulated so as to avoid the assumption
of existence of products in $\sA$.
 For every object $M$ and a family of objects $N_x\in\sA$, consider
the natural map of abelian groups
$$\textstyle
 \eta\:\Hom_\sA\bigl(M,\coprod_{x\in X} N_x\bigr)\lrarrow
 \prod_{x\in X}\Hom_\sA(M,N_x),
$$
assigning to a morphism $f\:M\rarrow\coprod_{x\in X}N_x$ the collection
of its compositions $\eta(f)=(\pi_y\circ f)_{y\in X}$ with
the projection morphisms $\pi_y\:\coprod_{x\in X}N_x\rarrow N_y$.
 An additive category $\sA$ with set-indexed coproducts is said to be
\emph{agreeable} if, for all objects $M$ and $N_x\in\sA$,
the map~$\eta$ is injective.
 It is the latter, more general definition that was formulated
in~\cite{Cor} and that we will use in the sequel.

 Let $\sA$ be an agreeable category, $M\in\sA$ be an object, and
$(N_x\in\sA)_{x\in X}$ be a family of objects.
 A family of morphisms $(f_x\:M\to N_x)_{x\in X}$ is said to be
\emph{summable} if there exists a (necessarily unique, by assumption)
morphism $f\:M\rarrow\coprod_{x\in X}N_x$ whose image under
the map~$\eta$ is equal to the element~$(f_x)_{x\in X}\in
\prod_{x\in X}\Hom_\sA(M,N_x)$.

 The particular case when all the objects $N_x$ are one and the same,
$N_x=N$, is important.
 Let $M$ and $N$ be two fixed objects in~$\sA$ and
$(f_x\:M\to N)_{x\in X}$ be a summable family of morphisms
between them.
 The \emph{sum} $\sum_{x\in X} f_x\:M\rarrow N$ of the summable family
of morphisms $(f_x)_{x\in X}$ is defined as the composition
$$
 M\overset f\lrarrow N^{(X)}\overset\Sigma\lrarrow N
$$
of the related morphism $f\:M\rarrow N^{(X)}=\coprod_{x\in X}N$ with
the natural summation morphism $\Sigma\:N^{(X)}\rarrow N$.
 The morphism~$\Sigma$ is defined by the condition that its
composition $\Sigma\iota_x\:N\rarrow N^{(X)}\rarrow N$
with the coproduct injection $\iota_x\:N\rarrow N^{(X)}$ is
the identity morphism $N\rarrow N$ for every $x\in X$.

\begin{ex} \label{grothendieck-agreeable}
 Any Grothendieck abelian category is agreeable.
 In fact, if $\sA$ is a complete, cocomplete abelian category with
exact direct limits, then the natural morphism
$\coprod_{x\in X}N_x\rarrow\prod_{x\in X}N_x$ is a monomorphism
for every family of objects $N_x\in\sA$, since it is the direct
limit of the split monomorphisms $\prod_{z\in Z}N_z\rarrow
\prod_{x\in X}N_x$ taken over the directed poset of all finite
subsets $Z\subset X$.
\end{ex}

\begin{rems}
 (1)~More generally, any abelian category satisfying Ab5 is agreeable.
 Indeed, let $\sA$ be an Ab5\+category, $(N_x)_{x\in X}$ be
a family of objects in~$\sA$, and $M\in\sA$ be an object.
 Given a nonzero morphism $f\:M\rarrow\coprod_{x\in X}N_x$, consider
its image~$f(M)$.
 The object $\coprod_{x\in X}N_x$ is the direct limit of its
subobjects $\coprod_{z\in Z}N_z$, where $Z$ ranges over all the finite
subsets of~$X$.
 Hence the subobject $f(M)\subset\coprod_{x\in X}N_x$ is the direct
limit of its subobjects $f(M)\cap\coprod_{z\in Z}N_z$
(cf.\ the proof of Lemma~\ref{simple-finitely-generated}).
 Since $f(M)\ne0$, there exists a finite subset $Z\subset X$ such that
the object $f(M)\cap\coprod_{z\in Z}N_z$ is nonzero.
 It follows that the composition of the morphism~$f$ with
the projection $\coprod_{x\in X}N_x\rarrow\coprod_{z\in Z}N_z=
\prod_{z\in Z}N_z$ is nonzero.
 Thus there exists $z\in Z$ for which the morphism $\pi_z\circ f\:
M\rarrow N_z$ is nonzero, so $\eta(f)\ne0$.

 (2)~Conversely, any agreeable abelian category $\sA$ satisfies Ab4,
i.~e., the functors of infinite coproducts in $\sA$ are exact
(cf.~\cite[Section~III.1]{Mit}).
 Indeed, let $g_x\:K_x\rarrow L_x$ be a family of monomorphisms
in~$\sA$; we have to prove that the morphism $\coprod_{x\in X}g_x\:
\coprod_{x\in X}K_x\rarrow\coprod_{x\in X}L_x$ is a monomorphism.
 Let $f\:M\rarrow\coprod_{x\in X}K_x$ be a nonzero morphism.
 Then there exists $y\in X$ such that the composition of~$f$ with
the projection map $\pi_y\:\coprod_{x\in X}K_x\rarrow K_y$ is nonzero.
 Hence the composition $g_y\pi_yf\:M\rarrow K_y\rarrow L_y$ is nonzero,
too.
 Denoting by~$\rho_y$ the projection map $\coprod_{x\in X}L_x\rarrow
L_y$, we have $\rho_y\circ\coprod_{x\in X}g_x=g_y\pi_y$.
 Hence $\rho_y\circ\coprod_{x\in X}g_x\circ f=g_y\pi_yf\ne0$, and it
follows that the composition of~$f$ with the morphism
$\coprod_{x\in X}g_x$ is nonzero.

 (3)~Moreover, any complete agreeable abelian category with
an injective cogenerator satisfies Ab5.
 This is the result of the paper~\cite{PSab5}.
\end{rems}

 In this paper, we will be mostly interested in a more special
class of additive categories, which we call \emph{topologically
agreeable}.
 In fact, a topologically agreeable category is an additive
category with the following additional structure.

 A \emph{right topological} additive category $\sA$ is an additive
category in which, for every pair of objects $M$ and $N\in\sA$,
the abelian group $\Hom_\sA(M,N)$ is endowed with a topology in
such a way that the following two conditions are satisfied:
\begin{enumerate}
\renewcommand{\theenumi}{\roman{enumi}}
\item the composition maps
$$
 \Hom_\sA(L,M)\times\Hom_\sA(M,N)\lrarrow\Hom_\sA(L,N)
$$
are continuous (as functions of two arguments) for all objects
$L$, $M$, $N\in\sA$;
\item open $\Hom_\sA(N,N)$\+submodules form a base of neighborhoods
of zero in $\Hom_\sA(M,N)$ for any two objects $M$, $N\in\sA$.
\end{enumerate}

 A right topological additive category $\sA$ is said to be
\emph{complete} (resp., \emph{separated}) if the topological
abelian group $\Hom_\sA(M,N)$ is complete (resp., separated)
for every pair of objects $M$ and $N\in\sA$.

 For any object $M$ in a right topological additive category $\sA$,
the topology on the group of endomorphisms $\Hom_\sA(M,M)$ makes it
a topological ring with a \emph{left} linear topology.
 Here the notation presumes that the ring $\Hom_\sA(M,M)$ acts on
the object $M\in\sA$ on the left.
 We will usually consider the opposite ring $\R=\Hom_\sA(M,M)^\rop$,
which acts on $M$ on the right.
 Hence $\R$ is a topological ring with a right linear topology.
 When the topological additive category $\sA$ is complete (resp.,
separated), so is the ring~$\R$.

\begin{lem} \label{right-topological-coproduct-lemma}
 Let\/ $\sA$ be a right topological additive category, let $M$ and
$N\in\sA$ be two objects, and let $X$ be a set such that a coproduct
$N^{(X)}$ of $X$ copies of $N$ exists in\/~$\sA$.
 Let $(f_x\:M\to N)_{x\in X}$ be a family of morphisms converging to
zero in the topology of the abelian group\/ $\Hom_\sA(M,N)$.
 Then the family of morphisms $(\iota_xf_x)_{x\in X}$ converges to zero
in the group\/ $\Hom_\sA(M,N^{(X)})$.
\end{lem}

\begin{proof}
 For any two elements $x$ and $y\in X$, denote by $\sigma_{x,y}\:N^{(X)}
\rarrow N^{(X)}$ the automorphism permuting the coordinates~$x$ and~$y$.
 In particular, $\sigma_{x,x}=\id_{N^{(X)}}$; and for any $x$, $y\in X$
we have $\sigma_{x,y}\iota_y=\iota_x\:N\rarrow N^{(X)}$.
 Choose a fixed element $x_0\in X$.
 Since the family of morphisms $f_x\:M\rarrow N$ converges to zero in
the topology of $\Hom_\sA(M,N)$, it follows from the continuity
axiom~(i) that the family of morphisms
$(\iota_{x_0}f_x\:M\to N^{(X)})_{x\in X}$ converges to zero in
the topology of the group $\Hom_\sA(M,N^{(X)})$.
 By axiom~(ii) (applied to the objects $M$ and $N^{(X)}\in\sA$),
we can conclude that the family of morphisms
$\iota_xf_x=\sigma_{x,x_0}\iota_{x_0}f_x\:M\rarrow N^{(X)}$ also
converges to zero in the topology of $\Hom_\sA(M,N^{(X)})$.
\end{proof}

\begin{lem} \label{agreeable-agreement-lemma}
 Let\/ $\sA$ be an additive category that is simultaneously agreeable
and complete, separated right topological.
 Let $M$ and $N\in\sA$ be two objects.
 Then any family of morphisms $f_x\in\Hom_\sA(M,N)$ converging to
zero in the topology of the abelian group\/ $\Hom_\sA(M,N)$ is summable
in the agreeable category\/~$\sA$.
 Moreover, the sum\/ $\sum^{\mathrm{top}}_{x\in X}f_x\in\Hom_\sA(M,N)$
defined as the limit of finite partial sums in the topology of
the group\/ $\Hom_\sA(M,N)$ coincides with the sum\/
$\sum_{x\in X} f_x=\sum_{x\in X}^{\mathrm{agr}}f_x$ computed in
the agreeable category\/ $\sA$ (so our notation is unambiguous).
\end{lem}

\begin{proof}
 By Lemma~\ref{right-topological-coproduct-lemma}, the family of
elements $(\iota_xf_x)_{x\in X}$ converges to zero in the topological
abelian group $\Hom_\sA(M,N^{(X)})$.
 Since this group is complete and separated by assumption,
the sum
$$\textstyle
 f=\sum^{\mathrm{top}}_{x\in X}\iota_xf_x\:M\rarrow N^{(X)},
$$
understood as the limit of finite partial sums, is well-defined.
 Using the continuity of composition again, one can see that
$$\textstyle
 \pi_y\circ f=\pi_y\circ\sum^{\mathrm{top}}_{x\in X}(\iota_x\circ f_x)=
 \sum^{\mathrm{top}}_{x\in X}(\pi_y\circ\iota_x\circ f_x)=f_y,
 \qquad y\in X,
$$
so $\eta(f)=(f_x)_{x\in X}$.
 This proves that the family of morphisms~$(f_x)$ is summable in
the agreeable category~$\sA$.
 Finally, the same continuity axiom~(i) implies that
$$\textstyle
 \sum^{\mathrm{agr}}_{x\in X}f_x=\Sigma\circ f=
 \sum^{\mathrm{top}}_{x\in X}(\Sigma\circ\iota_x\circ f_x)
 =\sum^{\mathrm{top}}_{x\in X}f_x,
$$
so the two notions of infinite summation agree.
\end{proof}

 Notice that the converse assertion to
Lemma~\ref{agreeable-agreement-lemma} certainly does \emph{not} hold
in general.
 In fact, endowing all the abelian groups $\Hom_\sA(M,N)$ with
the discrete topology defines a complete, separated right topological
additive category structure on any additive category $\sA$ in such a way
that no infinite family of nonzero morphisms converges to zero in
$\Hom_\sA(M,N)$.
% So a right topological structure on an additive category $\sA$ has to
%be nontrivial enough in order to be useful.

 A \emph{topologically agreeable} category $\sA$ is an agreeable
additive category endowed with a complete, separated right topological
additive category structure in such a way that, for any two objects
$M$, $N\in\sA$, every summable family of morphisms $f_x\:M\rarrow N$
converges to zero in the topology of $\Hom_\sA(M,N)$.

\begin{exs} \label{full-subcategory-top-agreeable}
 (1)~Any full subcategory closed under coproducts in an agreeable
category is agreeable. \par
 (2)~Any full subcategory closed under coproducts in a topologically
agreeable category is topologically agreeable.
\end{exs}

 An additive category $\sA$ is said to be \emph{idempotent-complete}
if all the idempotent endomorphisms of objects in $\sA$ have their
images in~$\sA$.
 Given an additive category $\sA$, the additive category obtained by
adjoining to $\sA$ the images of all the idempotent endomorphisms of
its objects is called the \emph{idempotent completion} of~$\sA$
(see, e.~g., \cite[Expos\'{e}~IV, Exercice~7.5(b)]{SGA4},
\cite[Section~1]{BaSch} or \cite[Section~4.4]{Fac2}).

\begin{exs} \label{idempotent-completion-top-agreeable}
 (1)~The idempotent completion of any agreeable additive category
is agreeable. \par
 (2)~Any structure of a right topological category on an additive
category $\sA$ can be extended in a unique way to a structure of
right topological category on the idempotent completion of~$\sA$.
 Indeed, given a topological abelian group $\A$ and its continuous
idempotent endomorphism $e\:\A\rarrow\A$, there exists a unique
topology on the abelian group $e\A$ for which both the inclusion
$e\A\rarrow\A$ and the projection $e\:\A\rarrow e\A$ are continuous.
 Hence both the existence and uniqueness follow from the continuity
axiom~(i). \par
 (3) In the context of~(2), if $\sA$ is a topologically agreeable
category, then such is the idempotent completion of~$\sA$.
\end{exs}

\begin{ex} \label{modules-topologically-agreeable-example}
 (1)~For any associative ring $A$, the category of left $A$\+modules
$\sA=A\modl$ is topologically agreeable.
 Indeed, $\sA$ is agreeable by Example~\ref{grothendieck-agreeable},
and the right topological structure on $\sA$ is defined by
the classical construction of the \emph{finite topology} on the group
of morphisms $\Hom_A(M,N)$ between two left $A$\+modules.
 Specifically, a base of neighborhoods of zero in $\Hom_A(M,N)$ is
provided by the annihilators of finitely generated $A$\+submodules
$E\subset M$ (see the references in~\cite[Example~2.13]{Pproperf},
and~\cite[Section~7.1]{PS} for a further discussion).

 One easily observes that a family of morphisms
$(f_x\:M\rarrow N_x)_{x\in X}$ in $A\modl$ corresponds to a morphism
$M\rarrow\bigoplus_{x\in X}N_x$ if and only if the family
is \emph{locally finite}, that is, for every finitely generated
submodule $E\subset M$ the set of all $x\in X$ for which
$f_x|_E\ne0$ is finite.
 When $N_x=N$ is one and the same module for all $x\in X$, this is
equivalent to convergence of the family of elements $(f_x)_{x\in X}$
to zero in the finite topology on the group $\Hom_A(M,N)$.

 (2)~More generally, the same construction as in~(1) provides
a topologically agreeable category structure on any locally finitely
generated (Grothendieck) abelian category (in the sense
of~\cite[Section~1.E]{AR}).
 This includes all locally finitely presentable abelian categories,
and in particular, all locally Noetherian and locally coherent
Grothendieck categories (cf.\ the discussion in
Section~\ref{topologically-coherent-secn} below).
\end{ex}

\begin{exs} \label{lwfg-closed-functors-examples}
 As it is essentially shown in the paper~\cite[Sections~9\<10]{PS},
further examples of topologically agreeable additive categories
include: \par
 (1)~all locally weakly finitely generated abelian
categories~\cite[Section~9.2]{PS}; and \par
 (2)~all the additive categories admitting a \emph{closed additive
functor} into a locally weakly finitely generated abelian category,
as discussed in~\cite[Section~9.3]{PS} (in particular, the categories
of comodules over coalgebras and corings~\cite[Proposition~10.4]{PS}
and semimodules over semialgebras~\cite[Proposition~10.8]{PS}).
\end{exs}

\begin{exs} \label{closed-functors-full-generality-examples}
 (1)~More generally, let $\sA$ be an additive category, $\sC$ be
an agreeable additive category, and $F\:\sA\rarrow\sC$
be a faithful functor preserving coproducts.
 Then the category $\sA$ is agreeable.
 
 (2)~Let $\sA$ be an additive category, $\sC$ be a right topological
additive category, and $F\:\sA\rarrow\sC$ be a faithful additive functor.
 Then for any two objects $M$ and $N\in\sC$ one can endow the group
$\Hom_\sA(M,N)$ with the induced topology of a subgroup in
the topological abelian group $\Hom_\sC(F(M),F(N))$.
 This construction makes $\sA$ a right topological additive category.

 (3)~Let $\sA$ be an additive category and $\sC$ be a topologically
agreeable additive category.
 An additive functor $F\:\sA\rarrow\sC$ is said to be \emph{closed} if
it is faithful, preserves coproducts, and for every pair of objects
$M$ and $N\in\sA$, the image of the injective map $F:\Hom_\sA(M,N)
\rarrow\Hom_\sC(F(M),F(N))$ is a closed subgroup in
$\Hom_\sC(F(M),F(N))$.

 Given a closed functor $F\:\sA\rarrow\sC$, the induced right
topological category structure on $\sA$ (as in~(2)) is clearly
complete and separated.
 Moreover, this topology makes $\sA$ a topologically agreeable
category.

 Indeed, let $(f_x\:M\rarrow N)_{x\in X}$ be a summable family of
morphisms in the agreeable category~$\sA$, and let
$f\:M\rarrow N^{(X)}$ be the related morphism into the coproduct.
 Then the family of morphisms $(F(f_x)\:F(M)\rarrow F(N))_{x\in X}$
is summable in $\sC$, since there is the morphism $F(f)\:F(M)
\rarrow F(N^{(X)})=F(N)^{(X)}$.
 Since $\sC$ is topologically agreeable by assumption, it follows
that the family of morphisms $(F(f_x))_{x\in X}$ converges to zero
in the topological group $\Hom_\sC(F(M),F(N))$.
 Therefore, the family of morphisms $(f_x)_{x\in X}$ converges to
zero in the induced topology of the subgroup
$\Hom_\sA(M,N)\subset\Hom_\sC(F(M),F(N))$.
\end{exs}

\begin{exs}
 A topologically agreeable category structure on a given agreeable
category does \emph{not} need to be unique.
 It is instructive to start with the following example of a bijective
continuous homomorphism of complete, separated topological rings with
right linear topologies $f\:\R'\rarrow\R''$ such that the induced map
$f[[X]]\:\R'[[X]]\rarrow\R''[[X]]$ is bijective for every set $X$, yet
the inverse homomorphism $f^{-1}\:\R''\rarrow\R'$ is not continuous.
 This will show that a complete, separated topology on an abelian group
is in no way determined by the related zero-convergent families of
elements and their sums.

 (1)~Let $R$ be the ring of (commutative) polynomials in an uncountable
set of variables~$x_i$ over a field~$k$, and let $S\subset R$ be
the multiplicative subset generated by the elements $x_i\in R$.
 Let $\R'$ denote the ring $R$ endowed with the discrete topology, and
let $\R''$ be the ring $R$ endowed with the \emph{$S$\+topology}
(in which the ideals $sR$, \,$s\in S$, form a base of neighborhoods
of zero).
 By~\cite[Proposition~1.16]{GT}, \,$\R''$ is a complete, separated
topological ring.
 One can easily see that \emph{no infinite family of nonzero elements
converges to zero} in~$\R''$.
 The identity map $f\:\R'\rarrow\R''$ is a continuous ring homomorphism,
and one has $\R'[[X]]=R[X]=\R''[[X]]$ for any set $X$; still the map
$f^{-1}\:\R''\rarrow\R'$ is not continuous.
 (Cf.~\cite[Remark~6.3]{Pcoun}.)

 (2)~Now let us present an example of an agreeable category with two
different topologically agreeable structures.
 For this purpose, one does not have to look any further than
the categories of modules $\sA=A\modl$ over associative rings~$A$.
 The constructions of
Examples~\ref{modules-topologically-agreeable-example}\,(1)
and~\ref{lwfg-closed-functors-examples}\,(1) provide two different
topologically agreeable structures on $A\modl$, generally speaking.

 A left $A$\+module $D$ is said to be \emph{weakly finitely
generated}~\cite[Section~9.2]{PS} if, for any family of left
$A$\+modules $(N_x)_{x\in X}$, the natural map of abelian groups
$\bigoplus_{x\in X}\Hom_A(D,N_x)\rarrow
\Hom_A(D,\>\bigoplus_{x\in X}N_x)$ is an isomorphism.
 Such modules $D$ are known as \emph{dually slender} or ``small'' in
the literature~\cite{EGT,Zem} (cf.~\cite[Remark~9.4]{PS}); and
an associative ring $A$ is said to be \emph{left steady} if all
such modules are finitely generated.
 In the \emph{weakly finite topology} of
Example~\ref{lwfg-closed-functors-examples}\,(1), for any left
$A$\+modules $M$ and $N$, annihilators of weakly finitely generated
submodules $D\subset M$ form a base of neighborhoods of zero in
the topological group $\Hom_A(M,N)$.

 Let $\sA''$ denote the topologically agreeable category of left
$A$\+modules with the finite topology on the groups of homomorphisms,
and let $\sA'$ stand for the topologically agreeable category of left
$A$\+modules with the weakly finite topology on the Hom groups.
 Then the identity functor $F\:\sA'\rarrow\sA''$ induces continuous
bijective maps $\Hom_{\sA'}(M,N)\rarrow\Hom_{\sA''}(M,N)$ for all
left $A$\+modules $M$ and $N$, but the inverse map
$\Hom_{\sA''}(M,N)\rarrow\Hom_{\sA'}(M,N)$ does not need to be
continuous.

 (3)~To give a specific example, let $A=k\{x,y\}$ be the free associative
algebra with two generators.
 Then any injective $A$\+module is weakly finitely
generated~\cite[Lemma~3.2]{Zem}.
 Let $M$ be an (infinitely generated) injective cogenerator of $A\modl$.
 Then the ring $\R'=\Hom_{\sA'}(M,M)^\rop$ is discrete, while the ring
$\R''=\Hom_{\sA''}(M,M)^\rop$ is not.
 Indeed, let $\U\subset\R''$ be an open right ideal.
 Then $\U$ contains the annihilator of some finitely generated
submodule $E\subset M$.
 This annihilator is nonzero, as $\Hom_A(M/E,M)\ne0$.
 Thus the zero ideal is not open in~$\R''$.
 But it is open in $\R'$, since $M$ is weakly finitely generated
(so one can take $D=M$).

 As in~(1), we have a continuous bijective map of complete, separated
topological rings with right linear topologies $f\:\R'\rarrow\R''$,
where $\R'$ is discrete but $\R''$ is not.
 Still, no infinite family of nonzero elements converges to zero
in~$\R''$.

 (4)~In fact, the functor $F\:\sA'\rarrow\sA''$ is an equivalence of
topologically agreeable categories \emph{if and only if} the ring $A$ is
left steady.
 Indeed, let $M$ be a weakly finitely generated left $A$\+module
that is not finitely generated.
 Put $N=\bigoplus_{E\subset M}M/E$, where $E$ ranges over all
the finitely generated $A$\+submodules of~$M$.
 Then the topological abelian group $\A'=\Hom_{\sA'}(M,N)$ is discrete,
since $M$ is weakly finitely generated.
 Let us show that the complete, separated topological abelian group
$\A''=\Hom_{\sA''}(M,N)$ is \emph{not} discrete.
 Let $\U\subset\A''$ be an open subgroup.
 Then $\U$ contains the annihilator of some finitely generated
submodule $E\subset M$.
 By construction, this annihilator is nonzero, as $\Hom_A(M/E,N)\ne0$.
 So the zero subgroup is not open in~$\A''$.

 Set $L=M\oplus N$, and consider the complete, separated right linear
topological rings $\R'=\Hom_{\sA'}(L,L)^\rop$ and $\R''=
\Hom_{\sA''}(L,L)^\rop$.
 Then the identity map $f\:\R'\rarrow\R''$ is a continuous ring
homomorphism such that the induced map $f[[X]]\:\R'[[X]]\rarrow
\R''[[X]]$ is bijective for every set~$X$.
 Yet it is clear from the previous paragraph and the argument in
Example~\ref{idempotent-completion-top-agreeable}\,(2) that the map
$f^{-1}\:\R''\rarrow\R'$ is not continuous.

 (5)~Slightly more generally, let $A$ be an associative ring and $M$
be a \emph{self-small} left $A$\+module, i.~e., a module for which
the natural map $\bigoplus_{x\in X}\Hom_A(M,M)\rarrow
\Hom_A(M,\>\bigoplus_{x\in X}M)$ is an isomorphism for any set~$X$.
 Then the finite topology on the ring $\R''=\Hom_{\sA''}(M,M)^\rop$ is
a complete, separated right linear topology in which no infinite
family of nonzero elements converges to zero.
 The weakly finite topology on the ring $\R'=\Hom_{\sA'}(M,M)^\rop$ has
the same properties.
\end{exs}

\begin{rem} \label{Q/R-endomorphisms-remark}
 Beyond the historical discussion in~\cite[Example~2.13]{Pproperf}
and~\cite[Section~7.1]{PS}, let us mention one class of examples of
topological rings of endomorphisms which appears in the literature.
 Let $R$ be a commutative domain and $Q$ be its ring of fractions.
 Consider the $R$\+module $K=Q/R$.
 Then the $R$\+algebra $\Hom_R(K,K)$ is
commutative~\cite[Proposition~7.1]{Mat} and isomorphic to
the completion of $R$ in
the $R$\+topology~\cite[Proposition~6.4]{Mat}.
 The completion (projective limit) topology on $\Hom_R(K,K)$ as
the completion of $R$ coincides with
the $R$\+topology~\cite[Theorems~6.8 and~6.10]{Mat} and with
the finite topology on $\Hom_R(K,K)$ as the endomorphism ring.

 The similar results hold for an arbitrary commutative ring $R$ and
its full ring of quotients $Q$ \,\cite[Section~2]{Mat2}, or even
more generally, for a commutative ring $R$ with a multiplicative
subset of regular elements $S\subset R$ and the $R$\+module
$K=(S^{-1}R)/R$.
 In the latter case, the commutative $R$\+algebra $\Hom_R(K,K)$ is
isomorphic to the completion of $R$ in the $S$\+topology;
the completion topology on $\Hom_R(K,K)$ coincides with
the $S$\+topology and with the finite topology of the endomorphism
ring~\cite[Proposition~3.2, and Theorems~2.3 and~2.5]{PMat}.
 A construction of a two-sided linear topology on the endomorphism
ring of a similar module $K$ for a noncommutative ring $R$ was
suggested in the paper~\cite[Proposition~3.5]{FN}.
\end{rem}

 In the rest of this section, we discuss the question of
\emph{existence} of a topologically agreeable structure and
consequences of such existence.

\begin{rem} \label{projective-contramodules-topologically-agreeable}
 The reader should be warned that the abelian category $\R\contra$
of left contramodules over a topological ring $\R$ is rarely agreeable, 
generally speaking (see, e.~g., the discussion
in~\cite[Section~1.5]{Prev}).
 However, the additive category of \emph{projective} left
$\R$\+contramodules $\R\contra_\proj$ is agreeable, and in fact has
a natural topologically agreeable category structure, which can be
explicitly constructed as follows.
 The construction of the matrix topology in
Section~\ref{matrix-topology-secn} (extended from the square
row-zero-convergent matrices with entries in $\R$ to the rectangular
ones in the obvious way) provides an agreeable topologization on
the full subcategory of \emph{free} left $\R$\+contramodules
in $\R\contra$.
 This topologization can be extended to the whole category
$\R\contra_\proj$ as explained in
Examples~\ref{idempotent-completion-top-agreeable}.
\end{rem}

 More generally, the following lemma holds.

\begin{lem} \label{agreeable-projectives-locally-presentable}
 Let\/ $\sB$ be a cocomplete abelian category with a projective
generator.
 Assume that the full subcategory of projective objects\/
$\sB_\proj\subset\sB$ is agreeable.
 Then the category\/ $\sB$ is locally presentable.
\end{lem}
 
\begin{proof}
 This is explained in the discussion in~\cite[Section~1.2]{PR}.
 In fact, let $P$ be a projective generator of~$\sB$.
 Then we claim that the category $\sB$ is locally $\kappa$\+presentable,
where $\kappa$~is the successor cardinal of the cardinality of the set
$\Hom_\sB(P,P)$.

 Indeed, the category $\sB$ is equivalent to the category of modules
over the monad $\boT_P\:X\longmapsto\Hom_\sB(P,P^{(X)})$ on the category
of sets; so it suffices to show that this monad is $\kappa$\+accessible
(see~\cite[Section~1.1]{PR} or~\cite[Section~1]{Pper}).
 For this purpose, one simply observes that a summable (in
the terminology of~\cite{PR}, ``admissible'') family of morphisms
$(f_x\:P\rarrow P)_{x\in X}$ cannot have any given nonzero morphism
$P\rarrow P$ repeated in it more than a finite number of times;
or otherwise the cancellation trick leads to a contradiction.
 Thus the cardinality of any summable family of nonzero morphisms
$(f_x\:P\rarrow P)$ is smaller than~$\kappa$.
\end{proof}

 Given an additive category $\sA$ with set-indexed coproducts and
an object $M\in\sA$, we denote by $\Add(M)\subset\sA$
the full subcategory consisting of all the direct summands of coproducts
of copies of $M$ in~$\sA$.
 The following theorem is the main result of what we call ``generalized
tilting theory'' (see~\cite{PS,PS2}).

\begin{thm} \label{generalized-tilting}
\textup{(a)} Let\/ $\sA$ be a idempotent-complete additive category
with set-indexed coproducts and $M\in\sA$ be an object.
 Then there exists a unique abelian category\/ $\sB$ with enough
projective objects such that the full subcategory\/ $\sB_\proj\subset\sB$
of projective objects in\/ $\sB$ is equivalent to the full subcategory\/
$\Add(M)\subset\sA$, that is
$$
 \sA\supset\Add(M)\cong\sB_\proj\subset\sB.
$$ \par
\textup{(b)} In the context of~\textup{(a)}, assume additionally that
the additive category\/ $\sA$ is agreeable.
 Then the abelian category\/ $\sB$ is locally presentable and, for
any family of projective objects $(P_x\in\sB)_{x\in X}$,
the natural morphism\/ $\coprod_{x\in X} P_x\rarrow\prod_{x\in X} P_x$
is a monomorphism in\/~$\sB$. \par
\textup{(c)} In the context of~\textup{(b)}, assume additionally that
the additive category\/ $\sA$ is topologically agreeable.
 Then the category\/ $\sB$ is equivalent to the category of left
contramodules over the topological ring\/ $\R=\Hom_\sA(M,M)^\rop$,
that is\/ $\sB=\R\contra$.
 The equivalence of additive categories\/ $\Add(M)\cong\R\contra_\proj$
takes the object $M\in\Add(M)$ to the free left\/ $\R$\+contramodule
with one generator\/ $\R\in\R\contra_\proj$.
\end{thm}

\begin{proof}
 Part~(a) is~\cite[Theorem~1.1(a)]{PS2}
(see also~\cite[Section~6.3]{PS}).
 In part~(b), we observe that the category $\Add(M)$ is agreeable
by Example~\ref{full-subcategory-top-agreeable}\,(1), and it follows
that the category $\sB_\proj$ is agreeable, too.
 Let $P\in\sB_\proj$ be the object corresponding to the object
$M\in\Add(M)$ under the equivalence $\Add(M)\cong\sB_\proj$;
then $P$ is a projective generator of~$\sB$.
 By Lemma~\ref{agreeable-projectives-locally-presentable},
the category $\sB$ is locally presentable.
 Finally, part~(c) essentially follows from
Lemma~\ref{agreeable-agreement-lemma} (cf.\ the proofs
of~\cite[Theorems~7.1, 9.9, and~9.11]{PS}).
\end{proof}

 If $\sA$ is abelian, the equivalence of categories $\Add(M)\cong\sB_\proj$ in
Theorem~\ref{generalized-tilting} can be extended to a pair of
adjoint functors between the whole categories $\sA$ and~$\sB$
\cite[Section~1]{PS2}.
 The fully faithful functor $\Add(M)\cong\sB_\proj\hookrightarrow\sB$
extends to a left exact functor $\Psi\:\sA\rarrow\sB$, while
the fully faithful functor $\sB_\proj\cong\Add(M)\hookrightarrow\sA$
extends to a right exact functor $\Phi\:\sB\rarrow\sA$, with
the functor $\Phi$ left adjoint to~$\Psi$.

 We will need the following more explicit description of
the functor~$\Psi$ \,\cite[proof of Proposition~6.2]{PS},
\cite[Section~1]{PS2}.
 The abelian category $\sB$ can be constructed as the category
of modules over the additive monad $\boT_M\:X\longmapsto
\Hom_\sA(M,M^{(X)})$ on the category of sets.
 Hence the category $\sB$ is endowed with a faithful, exact,
limit-preserving forgetful functor $\sB\rarrow\Ab$ to the category
of abelian groups.
 The composition of the right exact functor $\Psi\:\sA\rarrow\sB$ with
the exact forgetful functor $\sB\rarrow\Ab$ is computed as
the functor $\Hom_\sA(M,{-})\:\sA\rarrow\Ab$.

 In particular, in the context of Theorem~\ref{generalized-tilting}(c),
the monad $\boT_M$ is isomorphic to the monad $X\longmapsto\R[[X]]$
on the category of sets.
 The forgetful functor $\sB\cong\R\contra\rarrow\Ab$ assigns to
a left $\R$\+contramodule its underlying abelian group.

\medskip

 We say that an agreeable category is \emph{topologizable} if it
admits a topologically agreeable category structure.
 Notice that any topologizable category is (agreeable, hence)
additive with set-indexed coproducts by definition.

 The following theorem is one of the main results of this paper
(cf.\ the closely related Theorem~\ref{contra-split-iff-semisimple}).
 Its proof will be given below in Section~\ref{split-categories-secn}.

\begin{thm} \label{topologizable-spectral-is-semisimple}
 Any topologizable split abelian category is Ab5 and semisimple.
 So any topologizable spectral category is discrete spectral.
\end{thm}

\Section{Topological Rings as Endomorphism Rings}
\label{as-endomorphism-rings-secn}

 The aim of this section is to show that the categories of modules over
associative rings are, in a rather strong sense, representative among
all the topologically agreeable categories (so some results about
direct sum decompositions in topologically agreeable categories follow
from the same results for modules, as we will see below in
Section~\ref{perfect-decompositions-secn}).
 More specifically, in this section we prove that all complete,
separated topological rings with right linear topology can be realized
as the rings of endomorphisms of modules over associative rings
(endowed with the finite topology).

 Let $\sC$ be a small category.
 Then the category $\Funct(\sC,\Ab)$ of all functors from $\sC$ to
the category of abelian groups is a locally finitely presentable
Grothendieck abelian category.
 According to
Example~\ref{modules-topologically-agreeable-example}\,(2),
the category $\Funct(\sC,\Ab)$ is topologically agreeable; so for
any two functors $F$, $G\:\sC\rarrow\Ab$ there is a natural complete,
separated topology on the abelian group $\Hom_\Funct(F,G)$.
 Moreover, the ring $\Hom_\Funct(F,F)^\rop$ is a topological ring
with right linear topology.
 
 This topology can be explicitly described as follows.
 By the definition, a base of neighborhoods of zero in
$\Hom_\Funct(F,G)$ consists of the annihilators of finitely generated
subfunctors $E\subset F$.
 Simply put, this means that one has to choose a finite sequence of
objects $C_1$,~\dots, $C_m\in\sC$ and some element $e_j\in F(C_j)$
for every $j=1$,~\dots,~$m$; and consider the subgroup in
$\Hom_\Funct(F,G)$ consisting of all the natural transformations
$F\rarrow G$ annihilating the chosen elements~$e_j$.
 Subgroups of this form constitute a base of neighborhoods of zero in
$\Hom_\Funct(F,G)$.

 More generally, let $\sD$ be a (not necessarily small) category and
$\sC\subset\sD$ be a small subcategory.
 For every functor $F\:\sD\rarrow\Ab$ and object $D\in\sD$ consider
the induced homomorphism of abelian groups
$$
 F_{\sC/D}\:\bigoplus_{C\to D} F(C)\lrarrow F(D),
$$
where the direct sum in the left-hand side is taken over all pairs
(an object $C\in\sC$, a morphism $C\rarrow D$ in~$\sD$).
 We will say that the full subcategory $\sC\subset\sD$ is \emph{weakly
dense for~$F$} if the map $F_{\sC/D}$ is surjective for every object
$D\in\sD$.
 The full subcategory in the category of functors $\Funct(\sD,\Ab)$ 
consisting of all the functors $F\:\sD\rarrow\Ab$ for which the full
subcategory $\sC\subset\sD$ is weakly dense will be denoted by
$\Funct(\sD/\sC,\Ab)\subset\Funct(\sD,\Ab)$.

 Clearly, if the full subcategory $\sC\subset\sD$ is weakly dense for
a functor $F\:\sD\rarrow\Ab$, then any morphism $F\rarrow G$ of
functors $\sD\rarrow\Ab$ is determined by its restriction to
the full subcategory $\sC\subset\sD$.
 In particular, the restriction functor
$$
 \rho=\rho_{\sD/\sC}\:\Funct(\sD/\sC,\Ab)\rarrow\Funct(\sC,\Ab)
$$
assigning to a functor $F\:\sD\rarrow\Ab$ its restriction
$\rho(F)=F|_\sC$ to the full subcategory $\sC\subset\sD$ is faithful.
 It follows that morphisms $F\rarrow G$ between any two fixed
functors $F\in\Funct(\sD/\sC,\Ab)$ and $G\in\Funct(\sD,\Ab)$ form
a set rather than a proper class.
 Moreover, the following assertion holds.

\begin{lem} \label{restriction-closed}
 The restriction functor $\rho_{\sD/\sC}$ is a closed additive functor
in the sense of
Example~\textup{\ref{lwfg-closed-functors-examples}\,(2)}
or~\textup{\ref{closed-functors-full-generality-examples}\,(3)},
and~\cite[Section~9.3]{PS}.
\end{lem}

\begin{proof}
 Essentially, we have to check that, given two functors
$F\in\Funct(\sD/\sC,\Ab)$ and $G\in\Funct(\sD,\Ab)$, a morphism
of functors $g\:F|_\sC\rarrow G|_\sC$ can be extended to a morphism
of functors $f\:F\rarrow G$ provided that, for every finitely
generated subfunctor $E\subset F|_\sC$ there exists a morphism
of functors $h\:F\rarrow G$ such that the restriction of $h|_\sC$
to $E$ coincides with the restriction of~$g$ to~$E$.

 The latter condition means that, for every finite sequence of
objects $C_1$,~\dots, $C_m\in\sC$ and any elements $e_j\in F(C_j)$,
\ $j=1$,~\dots,~$m$, there exists a morphism of functors
$h\:F\rarrow G$ such that $h(e_j)=g(e_j)$ for all $1\le j\le m$.
 We have to show that there exists a morphism of functors
$f\:F\rarrow G$ such that $f|_\sC=g$; in other words, this means that
for every object $D\in\sD$ the map of abelian groups
$$
 \bigoplus_{C\to D} g_C\:\bigoplus_{C\to D} F(C)\lrarrow
 \bigoplus_{C\to D} G(C)
$$
descends to a map $f_D\:F(D)\rarrow G(D)$ forming a commutative square
with the maps $F_{\sC/D}$ and~$G_{\sC/D}$.
 Let $e\in\bigoplus_{C\to D} F(C)$ be an element annihilated by
the surjective map $F_{\sC/D}$; then the element~$e$ has a finite
number of nonzero components $e_j\in F(C_j)$, \ $C_j\in\sC$, \
$1\le j\le m$.
 Now existence of a morphism of functors $h\:F\rarrow G$ agreeing
with the morphism of functors $g\:F|_\sC\rarrow G|_\sC$ on
the elements $e_j\in F(C_j)$ guarantees that the element~$e$ is also
annihilated by the composition of maps $\bigoplus_{C\to D} F(C)
\rarrow\bigoplus_{C\to D} G(C)\rarrow G(D)$, implying existence
of the desired map~$f_D$.
\end{proof}

 In view of Lemma~\ref{restriction-closed}, for any two functors
$F\in\Funct(\sD/\sC,\Ab)$ and $G\in\Funct(\sD,\Ab)$, the group
$\Hom_\Funct(F,G)$ is a closed subgroup of the topological abelian
group $\Hom_\Funct(F|_\sC,\allowbreak G|_\sC)$.
 We endow the group $\Hom_\Funct(F,G)$ with the induced topology,
making it a complete, separated topological abelian group.
 In particular, the ring $\Hom_\Funct(F,F)^\rop$ becomes a complete,
separated topological ring with a right linear topology and
a closed subring in $\Hom_\Funct(F|_\sC,F|_\sC)^\rop$.

 Explicitly, a base of neighborhoods of zero in $\Hom_\Funct(F,G)$
is provided by the annihilators of finite sets of elements
$e_j\in F(C_j)$, \ $C_j\in\sC$.
 Now we observe that, due to the condition of surjectivity of
the maps $F_{\sC/D}$ imposed on the functor~$F$, the collection of
the annihilator subgroups of all finite sets of elements
$e_j\in F(D_j)$, \ $D_j\in\sD$, \ $1\le j\le m$, is another base of 
neighborhoods of zero for \emph{the same} topology on $\Hom_\Funct(F,G)$.
 Thus the complete, separated topology on the group $\Hom_\Funct(F,G)$
that we have constructed depends only on the (possibly large) category
$\sD$ and the fuctors $F$ and $G$, and does \emph{not} depend on
the choice of a weakly dense small subcategory $\sC\subset\sD$ for
the functor~$F$.

 In the rest of this section we apply the above considerations to
one specific large category $\sD$ with a small subcategory $\sC$ and
a functor $F\:\sD\rarrow\Ab$.
 Namely, let $R$ be a topological ring with a right linear topology,
and let $\sD=\discr R$ be the abelian category of discrete right
$R$\+modules.
 Furthermore, let $\sC\subset\sD$ be the full subcategory consisting
of all the cyclic discrete right $R$\+modules $R/I$, where $I\subset R$
ranges over the open right ideals in~$R$.
 Finally, let $F\:\discr R\rarrow\Ab$ be the forgetful functor,
assigning to a discrete right $R$\+module $\N$ its underlying
abelian group~$\N$.

\begin{prop} \label{endomorphisms-of-forgetful}
 The restriction map\/ $\Hom_\Funct(F,F)\rarrow
\Hom_\Funct(F|_\sC,F|_\sC)$ is bijective (hence a topological ring
isomorphism).
 The complete, separated topological ring\/ $\Hom_\Funct(F,F)^\rop$ is 
naturally isomorphic, as a topological ring, to the completion\/ $\R$ of
the topological ring $R$ (with the projective limit topology on\/~$\R$).
\end{prop}

\begin{proof}
 Since $\discr R=\discr\R$, we can replace the topological ring $R$
by its completion $\R$ from the outset and assume that we are dealing
with the category $\sD=\discr\R$, its full subcategory $\sC$ spanned
by the cyclic discrete right modules $\R/\I$, where $\I$ runs over
the open right ideals in $\R$, and the forgetful functor
$F\:\discr\R\rarrow\Ab$.
 Then the right action of $\R$ on the discrete right modules over it
provides a natural ring homomorphism $\R\rarrow\Hom_\Funct(F,F)^\rop$.
 Since we already know that $\Hom_\Funct(F,F)$ is a closed subring
in $\Hom_\Funct(F|_\sC,F|_\sC)$ with the induced topology, it suffices
to show that the composition $\R\rarrow\Hom_\Funct(F,F)^\rop\rarrow
\Hom_\Funct(F|_\sC,F|_\sC)^\rop$ is an isomorphism of topological rings.

 Since the topological ring $\R$ is separated, for any element $t\in\R$
there exists an open right ideal $\I\subset\R$ not containing~$t$.
 Then the element~$t$ acts nontrivially on the coset $1+\I\in F(\R/\I)$,
taking it to the coset $t+\I\ne0$.
 This proves injectivity of the map $\R\rarrow
\Hom_\Funct(F|_\sC,F|_\sC)^\rop$.
 To prove surjectivity, consider a natural transformation
$\tau\:F|_\sC\rarrow F|_\sC$.
 For every open right ideal $\I\in\R$, we can apply~$\tau$ to
the coset $1+\I\in F(\R/\I)$, obtaining an element $t_\I+\I=
\tau_{\R/\I}(1+\I)\in F(\R/\I)=\R/\I$.
 For every open right ideal $\J\subset\I$, there is a morphism
$p_{\I,\J}\:\R/\J\rarrow\R/\I$ in the category $\sC$ taking
the coset $r+\J$ to the coset $r+\I$ for every $r\in\R$.
 The commutativity equation $\tau_{\R/\I}F|_\sC(p_{\I,\J})=
F|_\sC(p_{\I,\J})\tau_{\R/\J}$ on the natural transformation~$\tau$
applied to the coset $1+\J\in F(\R/\J)$ implies that
$p_{\I,\J}(t_\J)=t_\I$.
 Thus the family of cosets $(t_\I+\I\in\R/\I)$ represents a well-defined
element~$t$ of the projective limit $\varprojlim_{\I\subset\R}\R/\I=\R$.

 Let us show that our natural transformation~$\tau$ is the image of
the element~$t$ under the map $\R\rarrow\Hom_\Funct(F|_\sC,F|_\sC)^\rop$.
 We have to check that, for every open right ideal $\I\subset\R$ and
every coset $s+\I\in\R/\I$, the image of $s+\I$ under~$\tau$ is given by
the rule $\tau_{\R/\I}(s+\I)=st+\I$.
 Choose an open right ideal $\J\subset\R$ such that $s\J\subset\I$.
 Then there is a morphism $s_{\I,\J}\:\R/\J\rarrow\R/\I$ in
the category $\sC$ taking the coset $r+\J$ to the coset $sr+\I$
for every $r\in\R$.
 Now the commutativity equation $\tau_{\R/\I}F|_\sC(s_{\I,\J})=
F|_\sC(s_{\I,\J})\tau_{\R/\J}$ on the natural transformation~$\tau$
applied to the coset $1+\J\in F(\R/\J)$ implies the desired equality
$\tau_{\R/\I}(s+\I)=st+\I$.

 It remains to show that the topologies on the rings $\R$ and
$\Hom_\Funct(F,F)$ agree.
 Indeed, for any finite sequence of discrete right $\R$\+modules
$\N_1$,~\dots, $\N_m$ and any chosen elements $e_j\in\N_j$,
the intersection of the annihilators of the elements~$e_j$ is
an open right ideal in~$\R$.
 Conversely, any open right ideal $\I\subset\R$ is the annihilator of
the single coset $1+\I\in\R/\I$.
\end{proof}

\begin{prop} \label{endomorphisms-of-functor=of-module}
 Let\/ $\sC$ be a small category and $F\:\sC\rarrow\Ab$ be a functor.
 Then there exists an associative ring $A$ and a left $A$\+module $M$
such that the topological ring\/ $\Hom_\Funct(F,F)^\rop$ is
isomorphic to the topological ring\/ $\Hom_A(M,M)^\rop$ (with
the finite topology on the latter; see
Example~\textup{\ref{modules-topologically-agreeable-example}\,(1)}).
\end{prop}

\begin{proof}
 Set $M=\bigoplus_{C\in\sC}F(C)$ to be the direct sum of the abelian
groups $F(C)$ over all the objects $C\in\sC$.
 Then every natural transformation $\tau\:F\rarrow F$ acts naturally by
an abelian group homomorphism $\tau_M\:M\rarrow M$ by the direct sum
of the maps $\tau_C\:F(C)\rarrow F(C)$.

 For every object $C\in\sC$, let $p_C\:M\rarrow M$ denote the projector
onto the direct summand $F(C)$ in~$M$.
 For every morphism $s\:C\rarrow D$ in $\sC$, let $s_{C,D}\:M\rarrow M$
denote the map whose restriction to $F(C)\subset M$ is equal to
the composition of the map $F(s)\:F(C)\rarrow F(D)$ with the inclusion
$F(D)\rarrow M$, while the restriction of~$s_{C,D}$ to $F(C')$ is zero
for all $C'\in\sC$, \,$C'\ne C$.
 So, in particular, $p_C=(\id_C)_{C,C}$.
 
 Let $A$ be the subring (with unit) in $\Hom_\boZ(M,M)$ generated by
the maps~$p_C$ and $s_{C,D}$, \ $C$, $D\in\sC$.
 Then the maps~$\tau_M$ are exactly all the abelian group endomorphisms
of $M$ commuting with all the maps~$p_C$ and~$s_{C,D}$.
 Indeed, any map $t\:M\rarrow M$ commuting with the projectors~$p_C$
for all $C\in\sC$ has the form $t=\bigoplus_{C\in\sC}t_C$ for some maps
$t_C\:F(C)\rarrow F(C)$.
 If the map~$t$ also commutes with the maps $s_{C,D}$ for all
morphisms $s\:C\rarrow D$ in $\sC$, then the collection of maps~$t_C$
is an endomorphism of the functor~$F$.
 Thus $\Hom_\Funct(F,F)\cong\Hom_A(M,M)$.

 Furthermore, the annihilators of finite subsets in $M$ form the same
collection of subgroups in $\Hom_A(M,M)$ as the annihilators of finite
sequences of elements $e_j\in F(C_j)$, \ $C_j\in\sC$, \ $j=1$,~\dots,~$m$.
 So the topologies on the rings $\Hom_\Funct(F,F)$ and $\Hom_A(M,M)$
also agree.
\end{proof}

\begin{cor} \label{topological-rings-are-endomorphism-rings}
 For every complete, separated topological ring\/ $\R$ with right
linear topology, there exists an associative ring $A$ and a left
$A$\+module $M$ such that\ $\R$ is isomorphic, as a topological ring,
to the ring\/ $\Hom_A(M,M)^\rop$ (endowed with the finite topology).
\end{cor}

\begin{proof}
 The assertion follows from
Propositions~\ref{endomorphisms-of-forgetful}
and~\ref{endomorphisms-of-functor=of-module}.
 For convenience, let us spell out the specific construction of
the ring $A$ and the module $M$ that we obtain.
 The underlying abelian group of $M$ is the direct sum
$\oplus_{\I\subset\R}\R/\I$ of the quotient groups of $\R$ by its
open right ideals.
 So $M$ is naturally a discrete right $\R$\+module (as it should be).
 The ring $A$ acting on $M$ on the left is simplest constructed as
the ring of endomorphisms of the right $\R$\+module~$M$.
 The above proof of Proposition~\ref{endomorphisms-of-functor=of-module}
provides a smaller ring $A$ which works just as well.
 Namely, it is the subring (with unit) in $\Hom_{\R^\rop}(M,M)$ or
$\Hom_\boZ(M,M)$ generated by the compositions
$$
 M\lrarrow\R/\J\overset{s_{\I,\J}}\lrarrow\R/\I\lrarrow M
$$
of the direct summand projections $M\rarrow\R/\J$, the direct summand
inclusions $\R/\I\rarrow M$, and the maps~$s_{\I,\J}$ mentioned in
the proof of Proposition~\ref{endomorphisms-of-forgetful}.
\end{proof}

\begin{rem}
 Let $B$ be a set of open right ideals forming a base of neighborhoods
of zero in a topological ring $R$ with right linear topology.
 Consider the category $\sD=\discr R$ and the full subcategory
$\sC\subset\sD$ consisting of all the cyclic discrete right
$R$\+modules $R/I$ with the ideal $I\in B$.
 Let $F\:\discr R\rarrow\Ab$ be the forgetful functor.
 The assertion of Proposition~\ref{endomorphisms-of-forgetful} remains
valid in this context, and the proof is essentially the same.
 In the particular case of a complete, separated topological ring $\R$
with a right linear topology, the construction of
Corollary~\ref{topological-rings-are-endomorphism-rings} gets
modified accordingly, providing a ``smaller'' ring $A$ and
left $A$\+module $M=\bigoplus_{\I\in B}\R/\I$ such that
$\R$ is isomorphic, as a topological ring, to $\Hom_A(M,M)^\rop$.

 The following question was raised at the end of~\cite[Section~2]{IMR}.
 Let $k$~be a field and $V$ be an infinite-dimensional $k$\+vector
space.
 Can one characterize complete, separated (associative and unital) topological $k$\+algebras $\R$ with right linear topology that can be
realized as closed subalgebras of $\Hom_k(V,V)^\rop$\,?

 The results of this section together with the first paragraph of
this remark allow to give a complete answer to this question.
 Let $\lambda=\dim_kV$ denote the cardinality of a basis of~$V$.
 Then a complete, separated topological $k$\+algebra $\R$ with
right linear topology can be realized as a closed subalgebra in
$\Hom_k(V,V)^\rop$ (with the finite topology on $\Hom_k(V,V)^\rop$
and the induced topology on~$\R$) if and only if the following two
conditions hold:
\begin{enumerate}
\renewcommand{\theenumi}{\roman{enumi}}
\item $\R$ has a base of neighborhoods of zero of the cardinality
not exceeding~$\lambda$; and
\item for every open right ideal $\I\subset\R$, one has
$\dim_k\R/\I\le\lambda$.
\end{enumerate}

 Indeed, if~(i) and~(ii) hold and $B$ is a base of neighborhoods
of zero in $\R$ consisting of at most~$\lambda$ open right ideals,
then one can consider the right $\R$\+module
$N=\bigoplus_{\I\in B}\R/\I$.
 By assumption, $\dim_kN\le\lambda$.
 In case the inequality is strict, one can replace $N$ by
the direct sum $M=N^{(\lambda)}$ of $\lambda$~copies of~$N$;
otherwise, put $M=N$.
 Denoting by $A$ the ring of endomorphisms of the right $\R$\+module
$M$, one has $\R=\End_A(M,M)^\rop$, essentially by the first
paragraph of this remark.
 By~\cite[Lemma~2.2\,(2)]{IMR}, \,$\R$ is a closed subring in
$\Hom_k(M,M)^\rop$.

 Conversely, let $\R\subset\End_k(V,V)^\rop$ be a closed subring with
the induced topology.
 Choose a basis $(v_i)_{i\in\lambda}$ of the vector space~$V$.
 Then the annihilators of finite subsets of~$\{v_i\}$ in $\R$
provide a base of neighborhoods of zero in $\R$ of the cardinality
at most~$\lambda$.
 This proves~(i).
 Furthermore, for any open right ideal $\I\subset\R$, there exists
a finite-dimensional subspace $W\subset V$ such that $\I$ contains
the annihilator of $W$ in~$\R$.
 Then the action of $\R$ in $V$ provides a natural injective map
$\R/\I\rarrow\Hom_k(W,V)$.
 Since $\dim_k\Hom_k(W,V)=\lambda$, this proves~(ii).
 
 Thus, in particular, any complete, separated topological $k$\+algebra
$\R$ with a right linear topology can be realized as a closed subalgebra
in the algebra $\Hom_k(V,V)^\rop$ of endomorphisms of a $k$\+vector
space~$V$.
 Moreover, one can realize $\R$ as a ``bicommutant'' or
``bicentralizer'', i.~e., the algebra of all $k$\+linear operators
commuting with a certain ring $A$ of such operators acting on~$V$.

 The above construction is certainly not new.
 A far more advanced result concerning realization of algebras over
commutative rings as endomorphism algebras of modules is known
as Corner's realization theorem~\cite[Theorem~20.1]{GT}.
\end{rem}

\Section{Matrix Topologies}  \label{matrix-topology-secn}

% Let $\sA$ be an idempotent-complete additive category with set-indexed
%coproducts and $M$, $N\in\sA$ be two objects such that the full
%subcategories $\Add(M)$ and $\Add(N)\subset\sA$ coincide.
% The aim of this section is to show that the monad
%$X\longmapsto\Hom_\sA(M,M^{(X)})$ on the category of sets can be
%described in terms of a complete, separated, right linear topology on
%the ring $\R=\Hom_\sA(M,M)^\rop$ if and only if the monad
%$X\longmapsto\Hom_\sA(N,N^{(X)})$ can be described in terms of
%a similar topology on the ring $\S=\Hom_\sA(N,N)^\rop$.
% Moreover, given a choice of the topology on the ring $\R$,
%a topology on the ring $\S$ can be chosen in such a way that
%the categories of discrete right $\R$\+modules and discrete
%right $\S$\+modules are equivalent. (???)

 Let $\R$ be a complete, separated topological ring with a right linear
topology and $Y$ be a set.
 The aim of this section is to construct a complete, separated right
linear topology on a certain ring $\S=\Mat_Y(\R)$ of $Y$\+sized matrices
with the entries in $\R$, in such a way that the category of left
$\S$\+contramodules be equivalent to the category of left
$\R$\+contramodules and the category of discrete right $\S$\+modules
equivalent to the category of discrete right $\R$\+modules.
 Furthermore, the free $\R$\+contramodule $\R[[Y]]$ with the set of 
generators $Y$ corresponds to the free $\S$\+contramodule with
one generator $\S$ under the equivalence of categories
$\R\contra\cong\S\contra$.

 Specifically, we denote by $\Mat_Y(\R)$ the set of all
\emph{row-zero-convergent} matrices with entries in $\R$, meaning
matrices $(a_{x,y}\in\R)_{x,y\in Y}$ such that for every $x\in Y$
the family of elements $(a_{x,y})_{y\in Y}$ converges to zero in
the topology of~$\R$.
 The abelian group $\Mat_Y(\R)$, with the obvious entrywise additive
structure, is in fact an associative ring with the unit element
$1=(\delta_{x,y})_{x,y\in Y}$ and the matrix multiplication
$$
 (ab)_{x,z}=\sum\nolimits_{y\in Y}a_{x,y}b_{y,z},
 \qquad a, \,b\in\Mat_Y(\R),
$$
defined using the multiplication in $\R$ and the infinite summation
of zero-converging families of elements, understood as the limit of
finite partial sums in the topology of~$\R$.
 It is important for this construction that the topology in $\R$ is
right linear, so whenever a family of elements $a_y\in\R$ converges
to zero, so does the family of elements $a_yb_y$, for \emph{any}
family of elements $b_y\in\R$.
 For a similar reason, the matrix~$ab$ is row-zero-convergent
whenever the matrices $a$ and~$b$ are.
 The multiplication in $\Mat_Y(\R)$ is associative, because for
any three matrices $a$, $b$, $c\in\Mat_Y(\R)$ and any two fixed
indices $x$ and $w\in Y$, the whole $(Y\times Y)$\+indexed family
of triple products $(a_{x,y}b_{y,z}c_{z,w})_{y,z\in Y}$ converges to
zero in the topology of~$\R$.

 Let us define a topology on $\Mat_Y(\R)$.
 For any finite subset $X\subset Y$ and any open right ideal
$\I\subset\R$, denote by $\K_{X,\I}\subset\Mat_Y(\R)$ the subgroup
consisting of all matrices~$a=(a_{x,y})_{x,y\in Y}$ such that
$a_{x,y}\in\I$ for all $x\in X\subset Y$ and $y\in Y$.

\begin{lem}
 The collection of all the subgroups\/ $\K_{X,\I}$, where $X$ ranges
over the finite subsets in $Y$ and\/ $\I$ ranges over the open right
ideals in\/ $\R$, is a base of neighborhoods of zero in a complete,
separated right linear topology on the ring\/ $\Mat_Y(\R)$. 
\end{lem}

\begin{proof}
 One easily observes that the collection of all subgroups
$\K_{X,\I}\subset\Mat_Y(\R)$ is a topology base.
 The quotient group $\Mat_Y(\R)/\K_{X,\I}$ is the group of all
row-finite rectangular $(X\times Y)$\+matrices with the entries
in~$\R/\I$.
 For a fixed finite subset $X\subset Y$, the projective limit of
the groups $\Mat_Y(\R)/\K_{X,\I}$ over all the open right ideals
$\I\subset\R$ is the group of all row-zero-convergent rectangular
$(X\times Y)$\+matrices (with a finite number of rows indexed
by the set~$X$) with entries in~$\R$.
 Passing to the projective limit over the finite subsets
$X\subset Y$, one obtains the whole group $\Mat_Y(\R)$.
 So our topology on this group is complete and separated.

 It is also easy to observe that $\K_{X,\I}$ is a right ideal in
the ring $\Mat_Y(\R)$.
 It remains to check that, for any matrix $a\in\Mat_Y(\R)$, any
finite subset $X\subset Y$ and any open right ideal $\I\subset\R$,
there exists a finite subset $W\subset Y$ and an open right
ideal $\J\subset\R$ such that $a\K_{W,\J}\subset\K_{X,\I}$
in $\Mat_Y(\R)$.
 Indeed, since $a$~is a row-zero-convergent matrix, there is
a finite subset $W\subset Y$ such that one has $a_{x,y}\in\I$
for all $x\in X$ and $y\in Y\setminus W$.
 Now $(a_{x,w})_{x\in X, w\in W}$ is just a finite matrix, and it
remains to choose an open right ideal $\J\subset\R$ in such
a way that $a_{x,w}\J\subset\I$ for all $x\in X$ and $w\in W$.
\end{proof}

\begin{prop} \label{matrix-discrete-modules}
 For any complete, separated topological ring\/ $\R$ with a right
linear topology and any nonempty set $Y$, the abelian category of
discrete right modules over the topological ring\/ $\Mat_Y(\R)$ is
equivalent to the abelian category of discrete right\/ $\R$\+modules.
\end{prop}

\begin{proof}
 The functor $\cV_Y\:\discr\R\rarrow\discr\Mat_Y(\R)$ assigns to
a discrete right $\R$\+module $\N$ the direct sum $\cV_Y(\N)=\N^{(Y)}$
of $Y$ copies of $\N$, viewed as the group of finitely supported
$Y$\+sized rows of elements of~$\N$.
 In other words, an element $m\in\cV_Y(\N)$ is a $Y$\+indexed family
of elements $(m_y\in\N)_{y\in Y}$ such that $m_y=0$ for all but
a finite subset of indices $y\in Y$.
 The ring $\Mat_Y(\R)$ acts on $\cV_Y(\N)$ by the usual formula for
the right action of matrices in rows,
$$
 (ma)_y=\sum\nolimits_{x\in Y}m_xa_{xy}, \qquad m\in \cV_Y(\N), \
 a\in\Mat_Y(\R).
$$
 One can easily check that the row~$ma$ is finitely supported, using
the assumptions that the row~$m$ is finitely supported, the matrix~$a$
is row-zero-convergent, and the right action of $\R$ on $\N$ is discrete.
 The action of $\cV_Y$ on morphisms between discrete right $\R$\+modules
is defined in the obvious way.
 The resulting functor $\cV_Y$ is clearly exact and faithful, so it
remains to check that it is surjective on morphisms and on
the isomorphism classes of objects.

 The closed subring of diagonal matrices in $\Mat_Y(\R)$ is isomorphic,
as a topological ring, to the product $\R^Y$ of $Y$ copies of
the ring $\R$ (endowed with the product topology).
 Let $\M$ be a discrete right $\Mat_Y(\R)$\+module; then $\M$ can be
also considered as a discrete right module over the subring
$\R^Y\subset\Mat_Y(\R)$.
 Then the description of discrete modules over a product of
topological rings~\cite[Lemma~8.1(a)]{Pproperf} shows that $\M$
decomposes naturally into a direct sum of discrete right $\R$\+modules
$\M=\bigoplus_{y\in Y}\N_y$, with the componentwise action of
the diagonal matrices from $\R^Y$ in $\bigoplus_{y\in Y}\N_y$.

 Let $e_{x,y}\in\Mat_Y(\R)$ denote the elementary matrix with
the entry~$1$ in the position~$(x,y)\in Y\times Y$ and the entry~$0$
in all other positions.
 Then the action of the elements $e_{x,y}$ and $e_{y,x}$ on the discrete
right $\Mat_Y(\R)$\+module $\M$ provides an isomorphism of discrete
right $\R$\+modules $\N_x\cong\N_y$.
 This makes all the discrete right $\R$\+modules $\N_y$, \,$y\in Y$
naturally isomorphic to each other; so we can set $\N=\N_y$.
 This provides an inverse functor $\discr\Mat_Y(\R)\rarrow
\discr\R$.

 Now we have a natural isomorphism of abelian groups
$\M\cong\cV_Y(\N)$ which agrees with the action of both the diagonal
matrices from $\R^Y$ and the elementary matrices~$e_{x,y}$ on
these two discrete right $\Mat_Y(\R)$\+modules.
 Since the subring generated by these two kinds of matrices in
$\Mat_Y(\R)$ contains all the finitely-supported matrices and is,
therefore, dense in the topology of $\Mat_Y(\R)$, it follows that
we have a natural isomorphism of discrete right $\Mat_Y(\R)$\+modules
$\M\cong\cV_Y(\N)$.
 This proves surjectivity of the functor $\cV_Y$ on objects; and its
surjectivity on morphisms is also clear from the above arguments.
\end{proof}

\begin{prop} \label{matrix-contramodules}
 For any complete, separated topological ring\/ $\R$ with a right
linear topology and any nonempty set $Y$, the abelian category of
left contramodules over the topological ring\/ $\Mat_Y(\R)$ is
equivalent to the abelian category of left\/ $\R$\+contramodules.
\end{prop}

\begin{proof}
 The functor $\V_Y\:\R\contra\rarrow\Mat_Y(\R)\contra$ assigns to
a left $\R$\+contra\-module $\C$ the product $\V_Y(\C)=\C^Y$ of $Y$
copies of $\C$, viewed as the group of all $Y$\+sized columns of
elements of~$\C$.
 So the elements $d\in\V_Y(\C)$ are described as $Y$\+indexed
families of elements $(d_y\in\C)_{y\in Y}$.
 The left contraaction of the ring $\Mat_Y(\R)$ on the set $\V_Y(\C)$
is defined by a ``contra'' (infinite summation) version of the usual
formula for the left action of matrices on columns,
$$
 \pi_{\V_Y(\C)}\left(\sum_{d\in\V_Y(\C)} a_dd\right)_{\!x}\;=\;
 \pi_\C\left(\sum_{d\in\V_Y(\C),\,y\in Y} a_{d,x,y}d_y
 \right).
$$
 Here $a_d=(a_{d,x,y}\in\R)_{x,y\in Y}$ is an element of the ring
$\Mat_Y(\R)$, defined for every $d\in\V_Y(\C)$; the family of elements
$(a_d)_{d\in\V_Y(\C)}$ converges to zero in the topology of $\Mat_Y(\R)$.
 The expression in parentheses in the left-hand side is an element
of the set $\Mat_Y(\R)[[\V_Y(\C)]]$ of infinite formal linear
combinations of elements of $\V_Y(\C)$ with zero-convergent families
of coefficients in $\Mat_Y(\R)$.
 The whole left-hand side is the $x$\+indexed component of
the element of $\V_Y(\C)$ that we want to obtain by applying
the left $\Mat_Y(\R)$\+contraaction map to a given element
of $\Mat_Y(\R)[[\V_Y(\C)]]$.

 The expression in parentheses in the right-hand side is an element
of the set $\R[[\C]]$.
 The sum in the right-hand side is understood as the limit of finite
partial sums in the projective limit topology of the group
$\R[[\C]]=\varprojlim_{\I\subset\R}(\R/\I)[\C]$ (where $\I$~ranges
over the open right ideals in~$\R$).
 To check that this sum converges in $\R[[\C]]$, it suffices to observe
that, for any fixed index~$x$, the double-indexed family of elements
$(a_{d,x,y}\in\R)_{d\in\V_Y(\C),\,y\in Y}$ converges to zero in~$\R$.
 The latter observation follows from the definition of the topology
on $\Mat_Y(\R)$.

 Checking the contraunitality of this left contraaction of $\Mat_Y(\R)$
in $\V_Y(\C)$ is easy; and the contraassociativity follows from
the contraassociativity of the left $\R$\+contra\-action in $\C$,
essentially, for the same reason as the matrix multiplication is
generally associative.
 To simplify the task of checking the details, one can use the notation
of~\cite[Section~1.2]{Pweak} and~\cite[Section~2.1]{Prev} for
the contraaction operation and the contraassociativity axiom.

 As in the previous proof, the action of $\V_Y$ on morphisms of
left $\R$\+contramodules is defined in the obvious way, and
the resulting functor $\V_Y\:\R\contra\rarrow\Mat_Y(\R)\contra$ is
clearly exact and faithful.
 So it remains to check that it is surjective on morphisms and
on the isomorphism classes of objects.
 
 Let $\D$ be a left $\Mat_Y(\R)$\+contramodule.
 Restricting the $\Mat_Y(\R)$\+contraaction in $\D$ to the closed
subring of diagonal matrices $\R^Y\subset\Mat_Y(\R)$ and using
the description of contramodules over a product of topological rings
given in~\cite[Lemma~8.1(b)]{Pproperf}, we obtain a functorial
decomposition of $\D$ into a direct product of left $\R$\+contramodules
$\D=\prod_{y\in Y}\C_y$, with the componentwise contraaction of
the diagonal matrices from $\R^Y$ in $\prod_{y\in Y}\C_y$.
 As in the previous proof, the action of the elementary matrices
$e_{x,y}\in\Mat_Y(\R)$ provides natural isomorphisms of left
$\R$\+contramodules $\C_x\cong\C_y$.
 So we can set $\C=\C_y$; this defines an exact inverse functor
$\Mat_Y(\R)\contra\rarrow\R\contra$.

 Now we have a natural isomorphism of abelian groups $\D=\V_Y(\C)$,
and it essentially remains to show that this is an isomorphism of
$\Mat_Y(\R)$\+contramodules.
 For this purpose, we will demonstrate that a left
$\Mat_Y(\R)$\+contramodule structure on any contramodule $\D$ is
can be expressed in terms of the contraaction of the diagonal
subring $\R^Y$ and the action of the elementary matrices~$e_{x,y}$.

 Indeed, let $a_d=(a_{d,x,y}\in\R)_{x,y\in Y}\in\Mat_Y(\R)$, \,$d\in\D$
be a $\D$\+indexed family of elements converging to zero in $\Mat_Y(\R)$.
 Put $\S=\Mat_Y(\R)$ for brevity.
 For every element $d\in\D$ and a pair of indices $x$, $y\in Y$,
consider the element $e_{x,y}d\in\S[[\D]]$.
 This is a finite formal linear combination of elements of the set $\D$
with exactly one nonzero coefficient in~$\S$.

 Furthermore, for every fixed index $x\in Y$, consider the infinite
formal linear combination of finite formal linear combinations
$\sum_{d,y}a_{d,x,y}(e_{x,y}d)$.
 Here the coefficients are the matrix entries $a_{d,x,y}\in\R$ viewed
as the scalar (hence diagonal) matrices
$a_{d,x,y}\in\R\subset\Mat_Y(\R)$.
 The $(\D\times Y)$\+indexed family of elements $(a_{d,x,y})_{d\in\D,\,
y\in Y}$ converges to zero in $\R$, because the $\D$\+indexed family
of matrices $(a_d)_{d\in\D}$ consists of row-zero-convergent matrices
and converges to zero in $\Mat_Y(\R)$.
 As the embedding of the subring of scalar matrices $\R\rarrow
\Mat_Y(\R)$ is continuous, the $(\D\times Y)$\+indexed family of
elements $(a_{d,x,y})_{d,y}$ converges to zero in $\S=\Mat_Y(\R)$
as well.
 So we have $\sum_{d,y}a_{d,x,y}(e_{x,y}d)\in\S[[\S[[\D]]]]$ for
every $x\in Y$.

 Finally, we consider the element
\begin{equation} \label{sigma-matrix-contraaction}
 \sigma=\sum_{x\in Y}e_{x,x}\left(\sum_{d\in\D,\,y\in Y}
 a_{d,x,y}(e_{x,y}d)\right)\in\S[[\S[[\S[[\D]]]]]].
\end{equation}
 This is an infinite formal linear combination of elements of
the set $\S[[\S[[\D]]]]$ with the coefficients $(e_{x,x}\in\S)_{x\in Y}$,
which form a $Y$\+indexed family of elements converging to zero
in the topology of~$\S=\Mat_Y(\R)$.
 In fact, the elements~$e_{x,x}$ belong to the closed subring of
diagonal matrices $\R^Y\subset\Mat_Y(\R)$.

 Now the (iterated) contraassociativity axiom tells that all
the compositions of ``opening of parentheses'' (monad multiplication)
and contraaction maps acting from $\S[[\S[[\S[[\D]]]]]]$ into $\D$
are equal to each other.
 In particular, for any set $Z$ there is the squared monad
multiplication map
$$
 \phi_Z^{(2)}=\phi_Z\circ\S[[\phi_Z]]=\phi_Z\circ\phi_{\S[[Z]]}
 \:\,\S[[\S[[\S[[Z]]]]]]\lrarrow\S[[Z]].
$$
 Setting $Z=\D$ and applying this map to the element~$\sigma$, we
obtain
\begin{equation} \label{phi-squared-matrix-contraaction}
 \phi^{(2)}_\D(\sigma)=
 \sum_{d\in\D}\sum_{x,y\in Y}(e_{x,x}a_{d,x,y}e_{x,y})d=
 \sum_{d\in\D}a_dd,
\end{equation}
since $a_d=\sum_{x,y\in Y}a_{d,x,y}e_{x,y}=\sum_{x,y\in Y}
e_{x,x}a_{d,x,y}e_{x,y}$ as the limit of finite partial sums
converging in the topology of~$\Mat_Y(\R)$.

 On the other hand, for any set $Z$ endowed with an (arbitrary)
map of sets $\pi_Z\:\S[[Z]]\rarrow Z$, there is the iterated map
$$
 \pi_Z^{(3)}=\pi_Z\circ\S[[\pi_Z]]\circ\S[[\S[[\pi_Z]]]]
 \:\,\S[[\S[[\S[[Z]]]]]]\lrarrow Z.
$$
 In the situation at hand with $Z=\D$, we see
from~\eqref{sigma-matrix-contraaction} that the value of
$\pi_\D^{(3)}(\sigma)\in\D$ is uniquely determined by the action of
the elements $e_{x,y}$ and the contraaction of the diagonal
subring $\R^Y\subset\Mat_Y(\R)$ in~$\D$.
 The contraassociativity equation
$$
 \pi_\D(\phi^{(2)}_\D(\sigma))=\pi_\D^{(3)}(\sigma)
$$
together with the equality~\eqref{phi-squared-matrix-contraaction}
tells that the whole left contraaction of $\Mat_Y(\R)$ in $\D$ is
determined by (and expressed explicitly by the above formulas in
terms of) these data, concluding the proof.
\end{proof}

\begin{lem} \label{matrix-free-contramodules-lemma}
 In the category equivalence\/ $\Mat_Y(\R)\contra\cong\R\contra$
of Proposition~\ref{matrix-contramodules},
the free left contramodule with one generator\/ $\Mat_Y(\R)$ over
the topological ring\/ $\Mat_Y(\R)$ corresponds to the free
left contramodule\/ $\R[[Y]]$ with the set of generators $Y$ over
the topological ring\/~$\R$.
\end{lem}

\begin{proof}
 Following the proof of Proposition~\ref{matrix-contramodules}
and~\cite[Lemma~8.1(b)]{Pproperf}, the left $\R$\+con\-tramodule $\C$
corresponding to a left $\Mat_Y(\R)$\+contramodule $\D$ can be
computed as the $\R$\+subcontramodule $\C=e_{x,x}\D$ in $\D$,
where $x$~is any chosen element of the set~$Y$.
 Hence, in the particular case of the free left
$\Mat_Y(\R)$\+contramodule $\D=\Mat_Y(\R)$, the left
$\R$\+contramodule $\C$ can be described as the set of all
zero-convergent $Y$\+sized rows (or rectangular
$(\{x\}\times Y)$\+matrices with one row) with entries in~$\R$,
which is isomorphic to $\R[[Y]]$ as an $\R$\+contramodule.
\end{proof}

\begin{rem}
 The above proof of Proposition~\ref{matrix-contramodules} has
the advantage of being direct and explicit, but it is quite involved.
 There is an alternative indirect argument based on the result of
Corollary~\ref{topological-rings-are-endomorphism-rings}.

 Let $A$ be an associative ring and $M$ be a left $A$\+module such
that the topological ring $\R$ is isomorphic to the topological ring
of endomorphism $\Hom_A(M,M)^\rop$ of the $A$\+module~$M$.
 Then the topological ring $\Mat_Y(\R)$ (with the above-defined
topology on it) is isomorphic to the topological ring of endomorphisms
$\Hom_A(M^{(Y)},M^{(Y)})^\rop$ of the direct sum $M^{(Y)}$ of $Y$
copies of the $A$\+module~$M$.
 By Theorem~\ref{generalized-tilting}(c), we have
$$
 \R\contra_\proj\cong\Add(M)=\Add(M^{(Y)})\cong\Mat_Y(\R)\contra_\proj.
$$
 An abelian category with enough projective objects is determined by
its full subcategory of projective objects; so the equivalence of
additive categories $\R\contra_\proj\cong\Mat_Y(\R)\contra_\proj$ extends
uniquely to an equivalence of abelian categories
$\R\contra\cong\Mat_Y(\R)\contra$.
 This is, of course, the same equivalence of categories
$\R\contra\cong\Mat_Y(\R)\contra$ as the one provided by
the constructions in the proof of
Proposition~\ref{matrix-contramodules}. {\hbadness=2000\par}

 Yet another proof of Proposition~\ref{matrix-contramodules} can be
found in~\cite[Theorem~7.9 and Example~7.10]{PS}.
 The direct approach worked out above in this section has also
another advantage, though: on par with the equivalence of
the contramodule categories, it allows to obtain an equivalence
of the categories of discrete modules in
Proposition~\ref{matrix-discrete-modules}.
\end{rem}

\Section{Topologically Semisimple Topological Rings}

 The concept of a topologically semisimple right linear topological
ring is based on Theorem~\ref{topologically-semisimple-ring}, which
we prove in this section.
 An important related result is
Theorem~\ref{contra-split-iff-semisimple}, whose proof we
postpone to Section~\ref{split-categories-secn}.

 Given an additive category $\sA$ with set-indexed coproducts, we denote
by $\Add_\infty(\sA,\Ab)$ the full subcategory in $\Funct(\sA,\Ab)$
consisting of all the functors $\sA\rarrow\Ab$ preserving all coproducts.
 If $M\in\sA$ is an object, we also denote by $\{M\}=\{M\}_\sA$
the full subcategory in $\sA$ spanned by the single object~$M$.
 When $\sA=\Add(M)$, the restriction functor $\Add_\infty(\sA,\Ab)
\rarrow\Funct(\{M\},\Ab)$ taking a functor $F\:\sA\rarrow\Ab$ to
the functor $F|_{\{M\}}\:\{M\}\rarrow\Ab$ is faithful.
 Hence morphisms between any two fixed objects in $\Add_\infty(\sA,\Ab)$
form a set. {\hbadness=1275\par}

 Similarly, given a cocomplete additive category $\sA$, we denote by
$\Rex_\infty(\sA,\Ab)$ the full subcategory in $\Funct(\sA,\Ab)$
consisting of all the functors $\sA\rarrow\Ab$ preserving all colimits
(or equivalently, all coproducts and cokernels).
 If an object $G\in\sA$ is a generator, then the restriction functor
$\Rex_\infty(\sA,\Ab)\rarrow\Funct(\{G\},\Ab)$ is faithful.
 So morphisms between any two fixed objects in $\Rex_\infty(\sA,\Ab)$
form a set.

\begin{lem} \label{discrete-modules-as-functors-on-contramodules}
 Let\/ $\R$ be a complete, separated topological ring with a right
linear topology.
 Then the functor
$$
 \Theta\:\discr\R\lrarrow\Rex_\infty(\R\contra,\Ab)
$$
induced by the pairing functor of contratensor product
$$
 \ocn_\R\:\discr\R\times\R\contra\lrarrow\Ab
$$
is fully faithful.
\end{lem}

\begin{proof}
 The free left $\R$\+contramodule $\R$ with one generator is a generator
of $\R\contra$, so morphisms between any two fixed objects in
$\Rex_\infty(\R\contra,\Ab)$ form a set.
 The functor~$\ocn_\R$ preserves colimits (in both its arguments), so
the functor $\Theta$ indeed takes values in $\Rex_\infty(\R\contra,\Ab)$.
 Furthermore, let us consider the composition of $\Theta$ with
the restriction functor $\rho\:\Rex_\infty(\R\contra,\Ab)\rarrow
\Funct(\{\R\},\Ab)$
 $$
  \discr\R\overset\Theta\lrarrow\Rex_\infty(\R\contra,\Ab)
  \overset\rho\lrarrow\Funct(\{\R\},\Ab).
 $$
 The category $\Funct(\{\R\},\Ab)\cong(\Hom_{\R\contra}(\R,\R))\modl\cong
\modr\R$ is equivalent to the category of right $\R$\+modules, and
the functor $\rho\circ\Theta$ is isomorphic to the fully faithful 
inclusion functor $\discr\R\rarrow\modr\R$ (due to the natural
isomorphism of abelian groups $\N\ocn_\R\R\cong\N$ for any discrete
right $\R$\+module~$\N$).
 According to the discussion preceding the lemma, the functor~$\rho$ is
faithful.
 Since $\rho$ is faithful and $\rho\circ\Theta$ is fully faithful, it
follows that $\Theta$ is fully faithful.
\end{proof}

\begin{thm} \label{topologically-semisimple-ring}
 Let\/ $\S$ be a complete, separated topological ring with a
right linear topology.
 Then the following conditions are equivalent:
\begin{enumerate}
\item the abelian category\/ $\S\contra$ is Ab5 and semisimple;
\item the abelian category\/ $\discr\S$ is split (or equivalently,
semisimple);
\item there exists an associative ring $A$ and a \emph{semisimple} left
$A$\+module $M$ such that\/ $\S$ is isomorphic, as a topological ring,
to the endomorphism ring of $M$ endowed with the finite topology,
$\S\cong\End_A(M)^\rop=\Hom_A(M,M)^\rop$;
\item there is a set $X$, an $X$\+indexed family of nonempty sets $Y_x$,
and an $X$\+indexed family of division rings $D_x$, \,$x\in X$, such
that the topological ring\/ $\S$ is isomorphic to the product of
the endomorphism rings of $Y_x$\+dimensional vector spaces over~$D_x$,
$$
 \S\,\cong\,\prod_{x\in X}\End_{D_x}\bigl(D_x^{(Y_x)}\bigr)^\rop.
$$
\end{enumerate}
\end{thm}

 Here\/ $\End_{D_x}\bigl(D_x^{(Y_x)}\bigr)^\rop =\Mat_{Y_x}(D_x)$ is,
generally speaking, the ring of row-finite infinite matrices of
size $Y_x$ with entries in~$D_x$.
 It is endowed with the finite topology of the endomorphism ring of
the $D_x$\+module $D_x^{(Y_x)}$ (see
Example~\ref{modules-topologically-agreeable-example}\,(1)), which
coincides with the topology of the ring of matrices with entries in
the discrete ring $D_x$ (as defined in
Section~\ref{matrix-topology-secn}).
 The ring $\S$ is isomorphic, as a topological ring, to a (generally
speaking) infinite product of such rings of infinite matrices, endowed
with the product topology.
 Topological rings satisfying the equivalent conditions of
Theorem~\ref{topologically-semisimple-ring} are called
\emph{topologically semisimple}.

\begin{rem} \label{top-semisimple-von-Neumann-regular}
 Associative rings of the form described above appear in the theory
of direct sum decompositions of modules~\cite{Har,Ang,AS}, where people
seem to usually say that such a ring $\S$ is von~Neumann\- regular
(which it is---but it is a very special kind of von~Neumann\-
regular ring).
 Certainly, $\S$ is not classically semisimple as an abstract ring, 
generally speaking; it is not Artinian, and the categories $\S\modl$
and $\modr\S$ of left and right modules over it are not semisimple.
 But as a topological ring, $\S$ is topologically semisimple in
the sense of the above theorem.
 (See Remark~\ref{semiregular-remark} below for further discussion.)
\end{rem}

\begin{rem} \label{simple-objects-made-explicit}
 It is instructive to consider the simple objects of the semisimple 
abelian categories $\discr\S$ and $\S\contra$.
 There is only one simple discrete right module over the topological ring
$\S_x=\End_{D_x}\bigl(D_x^{(Y_x)}\bigr)^\rop$, namely,
the $Y_x$\+dimensional vector space~$D_x^{(Y_x)}$.
 There is also only one simple left $\S_x$\+contramodule, namely,
the product $D_x^{Y_x}$ of $Y_x$ copies of~$D_x$.
 The discrete module and contramodule structures on these objects were
explicitly described in the proofs of
Propositions~\ref{matrix-discrete-modules}
and~\ref{matrix-contramodules}.
 For the ring $\S=\prod_{x\in X}\S_x$, both the simple
discrete right modules and the simple left contramodules are indexed
by the set~$X$ (see~\cite[Lemma~8.1]{Pproperf}).
\end{rem}

\begin{rem}
 The same class of topological rings (up to switching the roles of
the left and right sides) as in
Theorem~\ref{topologically-semisimple-ring} was characterized by
a list of many equivalent conditions in
the paper~\cite[Theorem~3.10]{IMR}, with the proof of the equivalence
based on a preceding result in the book~\cite[Theorem~29.7]{War}.
 In particular, our condition~(4) of
Theorem~\ref{topologically-semisimple-ring} is the same as
condition~(d) of~\cite[Theorem~3.10]{IMR}.
\end{rem}

\begin{proof}[Proof of Theorem~\ref{topologically-semisimple-ring}]
 By Remark~\ref{discrete-modules-split-iff-semisimple}, the abelian
category $\discr\S$ is split if and only if it is semisimple.
 We will prove the implications
$$
\text{(1)~$\Longrightarrow$ (2) $\Longrightarrow$ (4)
$\Longrightarrow$ (3) $\Longrightarrow$~(1).}
$$

 (1)\,$\Longrightarrow$\,(2) The argument is based on
Lemma~\ref{discrete-modules-as-functors-on-contramodules}.
 By Theorem~\ref{semisimple-category}\,(4), the category $\S\contra$
being Ab5 and semisimple means a category equivalence
$\S\contra\cong\mathop{\text{\Large $\times$}}_{x\in X} D_x\modl$
for some set of indices $X$ and a family of skew-fields
$(D_x)_{x\in X}$.
 The category $\S\contra$ is split abelian, so the two full
subcategories $\Rex_\infty(\S\contra,\Ab)$ and
$\Add_\infty(\S\contra,\Ab)$ in $\Funct(\S\contra,\Ab)$ coincide.

 Furthermore, for any coproduct-preserving functor
$N\:\mathop{\text{\Large $\times$}}_{x\in X} D_x\modl\rarrow\Ab$
the image of the one-dimensional left vector space $D_x$ over $D_x$
is naturally a left module over the ring $\Hom_{D_x\modl}(D_x,D_x)
=D_x^\rop$, i.~e., a right $D_x$\+vector space.
 The functor $N$ is uniquely determined by the collection of right
$D_x$\+vector spaces $(N(D_x)\in\modr D_x)_{x\in X}$, which can be
arbitrary.
 So the assignment $N\longmapsto (N(D_x))_{x\in X}$ establishes
a category equivalence
$$
 \Add_\infty\left(\mathop{\text{\huge $\times$}}_{x\in X} D_x\modl,\>
 \Ab\right)\,\cong\,\mathop{\text{\huge $\times$}}_{x\in X}
 \modr D_x.
$$

 By Lemma~\ref{discrete-modules-as-functors-on-contramodules}, it
follows that $\discr\S$ is a full subcategory in a semisimple
abelian category $\mathop{\text{\Large $\times$}}_{x\in X}\modr D_x$.
 It remains to observe that any abelian category which can be
embedded as a full subcategory into a split abelian category
is split.

 (2)\,$\Longrightarrow$\,(4) The argument is based on
Proposition~\ref{endomorphisms-of-forgetful}.
 By Theorem~\ref{semisimple-category}\,(4), we have $\discr\S\cong
\mathop{\text{\Large $\times$}}_{x\in X}\modr D_x$ for some set of
indices $X$ and a family of skew-fields $(D_x)_{x\in X}$.
 The forgetful functor $F\:\discr\S\rarrow\Ab$ (assigning to every
discrete right $\S$\+module $\N$ is underlying abelian group~$\N$)
can be thus interpreted as a functor
$$
 F\:\mathop{\text{\huge $\times$}}_{x\in X}\modr D_x\lrarrow\Ab.
$$
 We know that the functor $F$ is faithful and preserves
colimits/coproducts.

 As above, the image of the one-dimensional right vector space $D_x$
over $D_x$ under the functor $F$ is naturally a left module over
the ring $\Hom_{\modr D_x}(D_x,D_x)=D_x$.
 Denote this left $D_x$\+module by $V_x=F(D_x)$, and let $Y_x$ be
a set such that $V_x\cong D_x^{(Y_x)}$ in $D_x\modl$.
 Since the functor $F$ is faithful, the set $Y_x$ is nonempty.
 
 Now an endomorphism $t\:F\rarrow F$ of the functor $F$ is uniquely 
determined by the collection of left $D_x$\+module morphisms
$t_{D_x}\:V_x\rarrow V_x$, which can be arbitrary.
 In view of Proposition~\ref{endomorphisms-of-forgetful}, this
provides the desired ring isomorphism 
$\S\cong\prod_{x\in X}\End_{D_x}\bigl(D_x^{(Y_x)}\bigr)^\rop$.
 Finally, the topology on the ring $\S$ is also described by
Proposition~\ref{endomorphisms-of-forgetful}, which allows to
identify it with the product of the finite topologies on
the rings of row-finite matrices.

 (4)\,$\Longrightarrow$\,(3) Set
$A=\prod_{x\in X}D_x$ and $M=\bigoplus_{x\in X}D_x^{(Y_x)}$.

 (3)\,$\Longrightarrow$\,(1) The argument is based on the generalized
tilting theory.
 By Theorem~\ref{generalized-tilting}(a,c),
we have a category equivalence $\Add(M)\cong\S\contra_\proj$.
 Since $M$ is a semisimple $A$\+module, the category $\Add(M)$ is
abelian, Grothendieck, and semisimple.

 An abelian category $\sB$ with enough projective objects is uniquely
determined by its full subcategory of projective objects $\sB_\proj$.
 In particular, if $\sB_\proj$ happens to be split abelian, then all
objects of $\sB$ are projective.

 In the situation at hand, we conclude that the category $\S\contra=
\S\contra_\proj\cong\Add(M)$ is Ab5 and semisimple.

 Alternatively, the implication (3)\,$\Longrightarrow$\,(4) is easy
to prove, and (4)\,$\Longrightarrow$\,(1) holds by
Proposition~\ref{matrix-contramodules} and~\cite[Lemma~8.1(b)]{Pproperf}
(while (4)\,$\Longrightarrow$\,(2) follows directly from
Proposition~\ref{matrix-discrete-modules}
and~\cite[Lemma~8.1(a)]{Pproperf}).

 Notice also that it is clear from condition~(4) that the functor
$\Theta\:\discr\S\rarrow\Rex_\infty(\S\contra,\Ab)$ from
Lemma~\ref{discrete-modules-as-functors-on-contramodules} (which
was used in the proof of the implication (1)\,$\Longrightarrow$\,(2))
is actually an equivalence of categories for a topologically
semisimple topological ring~$\S$.
\end{proof}

 Our next theorem shows that, similarly to
Theorem~\ref{topologically-semisimple-ring}\,(2),
the semisimplicity condition in
Theorem~\ref{topologically-semisimple-ring}\,(1) can be relaxed
to the splitness condition.

\begin{thm} \label{contra-split-iff-semisimple}
 Let\/ $\S$ be a complete, separated topological ring with a
right linear topology.
 Then the abelian category\/ $\S\contra$ is split if and only if
it is Grothendieck and semisimple.
\end{thm}

\begin{rems}
 Notice that the abelian category of $\S$\+contramodules is rarely
Ab5 or Grothendieck, generally speaking.
 Theorem~\ref{contra-split-iff-semisimple} says that if it is split,
then it is both Grothendieck and semisimple.

 It would be interesting to know an example of a cocomplete split
abelian category with a generator which is not Grothendieck (i.~e.,
does not satisfy Ab5).
 We are not aware of any such examples.

 Theorem~\ref{contra-split-iff-semisimple} is essentially the same
result as the above Theorem~\ref{topologizable-spectral-is-semisimple}.
 Its proof is based on known results in the theory of direct sum
decompositions of modules~\cite{AS}.
 We present it below in Section~\ref{split-categories-secn}.
\end{rems}

\Section{Topologically Left T-Nilpotent Subsets}

 This section contains a technical lemma which will be
useful in Section~\ref{perfect-decompositions-secn}.

 Let $R$ be a separated topological ring with a right linear topology.
 A subset $E\subset R$ is said to be \emph{topologically nil} if for
every element $a\in E$ the sequence of elements $a^n$,
\,$n=1$, $2$,~\dots\ converges to zero in the topology of $R$
as $n\to\infty$.
 A subset $E\subset R$ is \emph{topologically left T\+nilpotent} if for
every sequence of elements $a_1$, $a_2$, $a_3$,~\dots~$\in E$
the sequence of elements $a_1$, $a_1a_2$, $a_1a_2a_3$,~\dots, 
$a_1a_2\dotsm a_n$,~\dots~$\in R$ converges to zero
as $n\to\infty$.
  In other words, this means that for every open right ideal
$I\subset R$ there exists $n\ge1$ such that $a_1a_2\dotsm a_n\in I$
(cf.~\cite[Section~6]{Pproperf}).

\begin{lem} \label{closed-subring-top-T-nilpotent}
 Let\/ $\R$ be a complete, separated topological ring with a right
linear topology and $E\subset\R$ be a topologically left T\+nilpotent
subset.
 Denote by\/ $\E$ the topological closure of the subring without unit
generated by\/ $E$ in\/~$\R$.
 Then\/ $\E$ is also a topologically left T\+nilpotent subset in\/~$\R$.
\end{lem}

\begin{proof}
 Let $E'$ denote the multiplicative subsemigroup (without unit)
generated by $E\cup-E$ in~$\R$.
 Clearly, if $E$ is topologically left $T$\+nilpotent, then so is~$E'$.
 
 Let $E''$ denote the additive subgroup generated by $E'$ in~$\R$.
 Our next aim is to show that $E''$ is topologically left T\+nilpotent
in~$\R$.

 Indeed, $(b_n)_{n\ge1}$ be a sequence of elements in~$E''$.
 Then $b_n=a_{n,1}+\dotsb+a_{n,m_n}$, where $a_{n,j}\in E'$ and
$m_n$~are some nonnegative integers.
 We need to show that the sequence of products $b_1b_2\dotsm b_n$
converges to zero in~$\R$.

 Consider the following rooted tree~$A$.
 The root vertex (that is, the only vertex of depth~$0$) has
$m_1$~children, marked by the elements $a_{1,1}$,~\dots,
$a_{1,m_1}\in\R$.

 These are the vertices of depth~$1$.
 Each of them has $m_2$~children.
 The children of the vertex of depth~$1$ marked by the element
$a_{1,j_1}$ are marked by the elements $a_{1,j_1}a_{2,1}$,~\dots,
$a_{1,j_1}a_{2,m_2}\in\R$.

 Generally, every vertex of depth~$n-1$ has $m_n$~children.
 If a vertex of depth~$n-1$ is marked by an element $r\in\R$,
then its children are marked by the elements $ra_{n,1}$,~\dots,
$ra_{n,m_n}\in\R$.
 So every vertex of depth~$n\ge1$ is marked by a product of the form
$a_{1,j_1}a_{2,j_2}\dotsm a_{n,j_n}\in\R$, where $1\le j_i\le m_i$.

 Now let $\I\subset\R$ be an open right ideal.
 Notice that if a vertex~$v$ in our tree~$A$ is marked by
an element~$a_v$ belonging to~$\I$, then all the vertices below~$v$
in $A$ are also marked by elements belonging to~$\I$.
 So we can consider the reduced tree $A_\I$ obtained by deleting
from $A$ all the vertices~$v$ for which $a_v\in\I$.

 Since the subset $E'\subset\R$ is topologically left T\+nilpotent,
every branch of the tree $A$ eventually encounters a vertex marked
by an element from~$\I$.
 Hence the reduced tree $A_\I$ has no infinite branches.
 It is also locally finite by construction (i.~e., every vertex has
only a finite number of children).
 By the K\H onig lemma, it follows that the whole reduced tree $A_\I$
is finite.

 Thus there exists an integer $n\ge1$ such that $A_\I$ has no
vertices of depth greater than $n-1$.
 Then the product $b_1b_2\dotsm b_n$ belongs to the ideal~$\I$.

 The subset $E''\subset\R$ is exactly the subring without unit
generated by $E$ in~$\R$.
 We have shown that $E''$ is topologically left T\+nilpotent.
 It remains to check that the topological closure $\E$ of $E''$
in $\R$ is.
 
 Let $(c_n\in\E)_{n\ge1}$ be a sequence of elements and $\I\subset\R$
be an open right ideal.
 Since $E''$ is dense in $\E$, there exists an element $b_1\in E''$
such that $c_1-b_1\in\I$.

 Furthermore, there exists an open right ideal $\J_1\subset\R$ such
that $b_1\J_1\subset\I$.
 Let $b_2\in E''$ be an element such that $c_2-b_2\in\J_1$.
 Then we have $b_1(c_2-b_2)\in\I$.

 Proceeding by induction, for every $i\ge 2$ we choose an open right
ideal $\J_{i-1}\subset\R$ such that $b_1\dotsm b_{i-1}\J_{i-1}\subset\I$,
and an element $b_i\in E''$ such that $c_i-b_i\in\J_{i-1}$.
 Then we have $b_1\dotsm b_{i-1}(c_i-b_i)\in\I$.
 Now
\begin{multline*}
 c_1\dotsm c_n - b_1\dotsm b_n = (c_1-b_1)c_2\dotsm c_n +
 b_1(c_2-b_2)c_3\dotsm c_n + \dotsb \\
 + b_1\dotsm b_{n-2}(c_{n-1}-b_{n-1})c_n +
 b_1\dotsm b_{n-1}(c_n-b_n)\,\in\,\I
\end{multline*}
for every $n\ge1$.
 It remains to choose $n\ge1$ such that $b_1b_2\dotsm b_n\in\I$,
and conclude that $c_1c_2\dotsm c_n\in\I$.
\end{proof}

\begin{rem} \label{almost-perfect-remark}
 Topologically T\+nilpotent ideals have been considered in
the literature in the following context.
 Let $R$ be commutative domain.
 Endow $R$ with the $R$\+topology; so nonzero ideals form a base
of neighborhoods of zero in~$R$.
 An ideal in $R$ was called ``topologically T\+nilpotent''
in~\cite{Sm} if it is topologically T\+nilpotent (in the sense of
our definition above) in the $R$\+topology.
 A commutative local domain $R$ was called ``topologically
T\+nilpotent'' (TTN) in~\cite{Sm} if its maximal ideal is
topologically T\+nilpotent.
 A commutative local domain is topologically T\+nilpotent in this
sense if and only if it is \emph{almost perfect} in the sense of
the papers~\cite{BS1,BS2}.

 More generally, let $R$ be a prime ring (i.~e., an associative ring
in which the product of any two nonzero two-sided ideals is nonzero).
 Then nonzero two-sided ideals form a base of a topology on $R$,
making $R$ a topological ring.
 For a prime local ring $R$, the Jacobson radical of $R$ in
topologically left T\+nilpotent in the described topology if and only
if the ring $R$ is \emph{left almost perfect} in the sense of
the paper~\cite{FP}.

 To give another generalization, let $R$ be a commutative local ring
and $S\subset R$ be a multiplicative subset.
 Endow $R$ with the $S$\+topology.
 Then the maximal ideal of $R$ is topologically T\+nil in the sense
of our definition above if and only if the ring $R$ is
\emph{$S$\+h-nil} in the sense of~\cite[Section~6]{BP1}.
\end{rem}

\Section{Lifting Orthogonal Idempotents}

 In this section we show that an (infinite, zero-convergent)
complete family of orthogonal idempotents can be lifted modulo
a topologically nil strongly closed two-sided ideal.
 In order to do so, we have to fill (what we think is) a gap in
the proof of~\cite[Lemma~8]{MM2}.
 The results of this section will be used in
Section~\ref{perfect-decompositions-secn}.

 First, we recall a lemma from~\cite{Pproperf} concerning the lifting of
a single idempotent.
 Given a subgroup $K$ in a topological abelian group $\A$, we
denote by $\overline{K}\subset\A$ the topological closure of
$K$ in~$\A$.
 We also recall the notation $H(R)$ for the Jacobson radical of
an associative ring~$R$ (when $\R$ is a topological ring, $H(\R)$
denotes the Jacobson radical of the abstract ring $\R$, which
ignores the topology).

\begin{lem} \label{single-idempotent-lifted}
 Let\/ $\R$ be a complete, separated topological ring with a right
linear topology, and let\/ $\HH\subset\R$ be a topologically nil
closed two-sided ideal.
 Let $f\in\R$ be an element such that $f^2-f\in\HH$.
 Then there exists an element $e\in f+\HH\subset\R$ such that
$e^2=e$ in\/ $\R$ and $e\in \overline{f\R f}\subset\R$.
\end{lem}

\begin{proof}
 A particular natural choice of the element~$e$ is provided by
the construction in the proof of~\cite[Lemma~10.3]{Pproperf}.
 The first two desired properties of the element~$e$ are
established in the argument in~\cite{Pproperf}, and the last one
is clear from the construction.
\end{proof}

 As pointed out in~\cite{CN}, once every individual idempotent in
an orthogonal family has been lifted, orthogonalizing the lifted
idempotents becomes a task solvable under weaker assumptions.
 The next proposition is a restatement of~\cite[Lemma~8]{MM2} with
the left and right sides switched (as the authors of~\cite{MM2}
work with left linear topologies on rings and we prefer the right
linear ones).

\begin{prop} \label{mohamed-mueller-orthog}
 Let\/ $\R$ be a complete, separated topological ring with a right
linear topology, $Z$ be a linearly ordered set of indices, and
$(e_z\in\R)_{z\in Z}$ be a family of idempotent elements such that
$e_we_z\in H(\R)$ for every pair of indices $z<w$ in $Z$,
the family of elements~$e_z$ converges to zero in the topology of\/ $\R$, 
and the element $u=\sum_{z\in Z}e_z$ is invertible in\/~$\R$.
 Then $(u^{-1}e_z)_{z\in Z}$ and $(e_zu^{-1})_{z\in Z}$ are two
families of orthogonal idempotents, converging to zero in\/ $\R$
with\/ $\sum_{z\in Z}u^{-1}e_z=1=\sum_{z\in Z}e_zu^{-1}$.
\end{prop}

 Note that if in the context of the proposition one has $u=1$,
then it follows that $e_we_z=0$ for all $w\ne z$.
 (Cf.\ the discussion in~\cite[Corollary~4]{MM2}.)

 The proof of Proposition~\ref{mohamed-mueller-orthog} is based on
two lemmas.
 The following one is just~\cite[Lemma~1]{MM2} with the left and right
sides switched.

\begin{lem}
 Let $R$ be an associative ring, $e=e^2\in R$ be an idempotent element,
and $a\in R$ be an element such that $eae\equiv e \bmod H(R)$.
 Then there exists an element $f=f^2\in aRe$ such that $Re=Rf$. \qed
\end{lem}

 The next lemma is our (expanded) version of~\cite[Lemma~2]{MM2}.

\begin{lem} \label{mm-lemma2}
 Let $R$ be an associative ring, $Z$ be a finite, linearly ordered set
of indices, and let $(e_z\in R)_{z\in Z}$ be idempotent elements
such that $e_we_z\in H(R)$ for all pairs of indices $z<w$ in~$Z$.
 Then
\begin{enumerate}
\item the sum of left ideals\/ $\sum_{z\in Z}Re_z$ is direct, and
a direct summand in the left $R$\+module $R$; \par
\item for any right $R$\+module $N$, the sum of subgroups\/
$\sum_{z\in Z} Ne_z$ in $N$ is direct.
\end{enumerate}
\end{lem}

\begin{proof}
 The assertion of~\cite[Lemma~2]{MM2} is essentially our part~(1);
the difference is that we have added part~(2).

 For $|Z|=2$, the argument in~\cite{MM2} uses the previous lemma
in order to find a pair of orthogonal idempotents $f_1$ and
$f_2\in R$ such that $Re_1=Rf_1$ and $Re_2=Rf_2$.
 It follows that $Ne_1=Nf_1$ and $Ne_2=Nf_2$.
 Then it is clear that $R=Rf_1\oplus Rf_2\oplus R(1-f_1-f_2)$,
and similarly $N=Nf_1\oplus Nf_2\oplus N(1-f_1-f_2)$.

 For $|Z|\ge3$, the argument in~\cite{MM2} proceeds by
induction, using the case $|Z|=2$ in order to make the induction step.
 This works for part~(2) exactly in the same way as for part~(1).

 The nature of the induction step is such that assuming
the original idempotents $e_z\in R$ to be only ``half-orthogonal
modulo~$H(R)$'' (for $z<w$) is more convenient than requiring
two-sided orthogonality modulo the Jacobson radical.
\end{proof}

\begin{proof}[Proof of Proposition~\ref{mohamed-mueller-orthog}]
 By continuity of multiplication in $\R$, both the families of
elements $(u^{-1}e_z)_{z\in Z}$ and $(e_zu^{-1})_{z\in Z}$ converge
to zero, and $\sum_{z\in Z}u^{-1}e_z=1=\sum_{z\in Z}e_zu^{-1}$.
 We have to prove that these are two families of orthogonal
idempotents.
 Here it suffices to check that $e_wu^{-1}e_z=\delta_{z,w}e_z$
for all $z$, $w\in Z$.

 For every fixed $w\in Z$, we have
\begin{equation} \label{mm-equation}
 e_w=e_w\left(\sum_{z\in Z}u^{-1}e_z\right)=
 e_wu^{-1}e_w+\sum_{z\in Z,\, z\ne w}e_wu^{-1}e_z,
\end{equation}
where the infinite sum is understood as the limit of finite partial
sums in the topology of~$\R$.
 Let $\I\subset\R$ be an open right ideal.
 Then there exists an open right ideal $\J\subset\R$ such that
$e_wu^{-1}\J\subset\I$.
 Since the family $(e_z)_{z\in Z}$ converges to zero in $\R$ by
assumption, for all but a finite subset of indices $z\in Z$ we
have $e_z\in\J$, hence $e_wu^{-1}e_z\in\I$.

 Considering the equation~\eqref{mm-equation} modulo~$\I$ (that is,
as an equation in $\R/\I$), the converging infinite sum in the right-hand
side reduces to a finite one.
 Let $Z'\subset Z$ denote a finite subset of indices such that
$w\in Z'$ and $e_wu^{-1}e_z\in\I$ for $z\in Z\setminus Z'$.
 Applying Lemma~\ref{mm-lemma2}\,(2) to the ring $R=\R$, the finite set
of idempotents $(e_z\in\R)_{z\in Z'}$, and the right $R$\+module
$N=\R/\I$, we obtain that the sum $\sum_{z\in Z'}Ne_z$ is direct.
 Hence it follows from the equation~\eqref{mm-equation} taken
modulo~$\I$ that
$$
 e_wu^{-1}e_w\equiv e_w \bmod \I
 \quad\text{and}\quad e_wu^{-1}e_z\equiv 0 \bmod \I
 \text{\ \ for $z\ne w$.}
$$
 Since this holds modulo every open right ideal $\I\subset\R$, and
the topological ring $\R$ is assumed to be separated, we can conclude
that $e_wu^{-1}e_w=e_w$ and $e_wu^{-1}e_z=0$ in $\R$ for all $z\ne w$,
as desired.
\end{proof}

 Combining Lemma~\ref{single-idempotent-lifted} with
Proposition~\ref{mohamed-mueller-orthog}, we obtain
the following theorem.

\begin{thm} \label{lifting-orthogonal-thm}
 Let\/ $\R$ be a complete, separated topological ring with a right
linear topology and\/ $\HH\subset\R$ be a topologically nil closed
two-sided ideal.
 Let $(f_z\in\R)_{z\in Z}$ be a family of elements such that
$f_z^2-f_z\in\HH$ for all $z\in Z$ and $f_wf_z\in\HH$ for all
$z\ne w$, \,$z,w\in Z$.
 Assume further that the family of elements $(f_z\in\R)_{z\in Z}$
converges to zero in the topology of\/ $\R$ and\/ $\sum_{z\in Z}f_z
\in 1+\HH$.
 Then there exist elements $e_z\in f_z+\HH\subset\R$ such that
$e_z^2=e_z$ for all $z\in Z$ and $e_we_z=0$ for all $z\ne w$,
\,$z,w\in Z$.
 Moreover, the family of elements $(e_z\in\R)_{z\in Z}$ converges to
zero in\/ $\R$ and\/ $\sum_{z\in Z}e_z=1$.
 In addition, one can choose the elements~$e_z$ in such a way that
$e_z\in\overline{\R f_z}$ for all $z\in Z$, or in such a way that
$e_z\in\overline{f_z\R}$ for all $z\in Z$, as one wishes.
\end{thm}

\begin{proof}
 According to Lemma~\ref{single-idempotent-lifted}, there exist 
idempotent elements $e'_z\in f_z+\HH$ for which one also has
$e'_z\in\overline{f_z\R f_z}$.
 Hence for every open right ideal $\I\subset\R$ one has $e'_z\in\I$
whenever $f_z\in\I$; so the family of elements~$e'_z$ converges
to zero in $\R$ since the family of elements~$f_z$ does.
 Furthermore, one has $\sum_{z\in Z}e'_z\in\sum_{z\in Z}f_z+\HH=1+\HH$.
 Besides, $e'_we'_z\in f_wf_z+\HH=\HH$ for all $z\ne w$.

 By~\cite[Lemma~7.6(a)]{Pproperf}, any topologically nil left or right
ideal in $\R$ is contained in $H(\R)$; hence $\HH\subset H(\R)$.
 We also observe that the element $u=\sum_{z\in Z}e'_z$ is invertible
in $\R$, since $u\in 1+\HH$.
 Thus Proposition~\ref{mohamed-mueller-orthog} is applicable to
the family of elements $e'_z\in\R$; and one can set, at one's
choice, either $e_z=u^{-1}e'_z$ for all $z\in Z$, or
$e_z=e'_zu^{-1}$ for all $z\in Z$.
 Finally, in both cases $e_z\in e_z'+\HH=f_z+\HH$, since $u\in 1+\HH$.
\end{proof}

\begin{cor} \label{lifting-orthogonal-cor}
 Let\/ $\R$ be a complete, separated topological ring with a right
linear topology, let\/ $\HH\subset\R$ be a topologically nil, strongly
closed two-sided ideal, and let\/ $\S=\R/\HH$ be the quotient ring.
 Let $(\bar e_z\in\S)_{z\in Z}$ be a set of orthogonal idempotents
in\/~$\S$.
 Assume further that the family of elements~$\bar e_z$ converges to
zero in\/ $\S$ and\/ $\sum_{z\in Z}\bar e_z=1$ in\/~$\S$.
 Then there exists a set of orthogonal idempotents $(e_z\in\R)_{z\in Z}$
such that $\bar e_z=e_z+\HH$ for every $z\in Z$.
 Moreover, the family of elements~$e_z$ converges to zero in\/ $\R$
and\/ $\sum_{z\in Z}e_z=1$ in\/~$\R$.
\end{cor}

\begin{proof}
 As the ideal $\HH\subset\R$ is strongly closed by assumption,
the zero-convergent family of elements $\bar e_z\in\S$ can be lifted
to \emph{some} zero-convergent family of elements $f_z\in\R$.
 This is enough to satisfy the assumptions of
Theorem~\ref{lifting-orthogonal-thm}.
\end{proof}

\Section{Split Direct Limits}  \label{split-direct-limits-secn}

 The aim of this section and the next one is to discuss
the contramodule-theoretic implications and categorical generalizations
of the characterization of modules with perfect decompositions
in~\cite[Theorem~1.4]{AS}.
 In particular, in this section we extend the result 
of~\cite[Theorem~1.4\,(2)\,$\Leftrightarrow$\,(3)]{AS}
to the context of additive categories.

 One difference with the approach in~\cite{AS} is that, given
an object $M$ in an additive or abelian category $\sB$, we want to
formulate the condition of split direct limits in $\Add(M)\subset\sB$
as \emph{intrinsic} to the category $\sA=\Add(M)$, so that it makes
sense independently of the ambient category~$\sB$.

 Let $X$ be a directed poset and $\sA$ be an additive category.
 We will say that $\sA$ has \emph{$X$\+sized coproducts} if coproducts
of families of objects of cardinality not exceeding the cardinality
of $X$ exist in~$\sA$.
 An additive category $\sA$ is said to have \emph{$X$\+shaped direct
limits} if the direct limit of any $X$\+shaped diagram $X\rarrow\sA$
exists in it.

 Recall that whenever $X$\+shaped direct limits exist, they always
preserve epimorphisms and cokernels.
 Assuming in addition that $X$\+sized coproducts exist, for any
$X$\+shaped diagram $(g_{yx}\:N_x\to N_y)_{x<y\in X}$ in $\sA$ there
is a natural right exact sequence
\begin{equation} \label{presentation-direct-limit}
 \coprod\nolimits_{x<y}N_x\lrarrow\coprod\nolimits_{x\in X}N_x
 \lrarrow\varinjlim\nolimits_{x\in X}N_x\lrarrow 0.
\end{equation}
 In other words, the righmost nontrivial morphism in this sequence is
the cokernel of the leftmost one.
 In particular, it is an epimorphism.
 (In fact, all these properties hold for colimits of $X$\+shaped
diagrams for an arbitrary small category~$X$.)

\begin{lem} \label{split-X-direct-limits-conditions}
 Let $X$ be a directed poset and\/ $\sA$ be an idempotent-complete
additive category with $X$\+shaped direct limits and $X$\+sized
coproducts.
 Consider the following properties: \par
\begin{enumerate}
\item the direct limit of any $X$\+shaped diagram of split monomorphisms
is a split monomorphism in\/~$\sA$;
\item the direct limit of any $X$\+shaped diagram of split epimorphisms
is a split epimorphism in\/~$\sA$; \par
\item for any $X$\+shaped diagram $(N_x\to N_y)_{x<y\in X}$ in\/ $\sA$,
the natural epimorphism\/ $\coprod_{x\in X} N_x\rarrow
\varinjlim_{x\in X} N_x$ is split.
\end{enumerate}
 Then the implications \textup{(1)~$\Longrightarrow$~(2)
$\Longleftrightarrow$~(3)} hold.
\end{lem}

\begin{proof}
 (1)\,$\Longrightarrow$\,(2)
 Let $(M_x\to M_y)_{x<y}\rarrow (N_x\to N_y)_{x<y}$ be
an $X$\+indexed diagram of split epimorphisms $M_x\rarrow N_x$ in~$\sA$.
 Let $L_x$ be the kernel of the split epimorphism $M_x\rarrow N_x$; then
$(L_x\to L_y)_{x<y}\rarrow (M_x\to M_y)_{x<y}$ is an $X$\+indexed
diagram of split monomorphisms.
 By assumption, $f\:\varinjlim_{x\in X} L_x\rarrow
\varinjlim_{x\in X} M_x$ is a split monomorphism in~$\sA$.
 Since direct limits preserve cokernels, the morphism
$\varinjlim_{x\in X} M_x\rarrow\varinjlim_{x\in X} N_x$ is the cokernel
of~$f$, hence a split epimorphism.

 (2)\,$\Longrightarrow$\,(3)
 Given an $X$\+shaped diagram $(g_{yx}\:N_x\to N_y)_{x<y\in X}$,
we consider the following $X$\+shaped diagram $(M_x\to M_y)_{x<y}$
in the category~$\sA$.
 For every index $x\in X$, the object $M_x$ is the coproduct of $N_z$
over all the indices $z\le x$ in~$X$.
 For every pair of indices $x<y$, the morphism $M_x\rarrow M_y$
is the subcoproduct inclusion of the coproduct indexed by
$\{z\mid z\le x\}$ into the coproduct indexed by $\{z\mid z\le y\}$.
 Then one has $\varinjlim_{x\in X} M_x=\coprod_{x\in X}N_x$.
 Furthermore, for every $x\in X$ there is a split epimorphism
$M_x\rarrow N_x$ in $\sA$ with the components $g_{x,z}\:N_z\rarrow N_x$,
\ $z\le x$.
 Taken together, these morphisms form a natural morphism of diagrams
$(M_x\to M_y)_{x<y}\rarrow (N_x\to N_y)_{x<y}$, and the induced
morphism of direct limits $\varinjlim_{x\in X} M_x\rarrow
\varinjlim_{x\in X} N_x$ is the epimorphism $\coprod_{x\in X} N_x\rarrow
\varinjlim_{x\in X} N_x$ we are interested in.
 Thus the latter epimorphism is split whenever (2)~holds.

 (3)\,$\Longrightarrow$\,(2)
 Let $(M_x\to M_y)_{x<y}\rarrow (N_x\to N_y)_{x<y}$ be
an $X$\+indexed diagram of split epimorphisms.
 Then the induced morphism $\coprod_x M_x\rarrow\coprod_x N_x$ is
a split epimorphism, too.
 By assumption, so is the natural morphism $\coprod_x N_x\rarrow
\varinjlim_x N_x$.
 Now it follows from commutativity of the square diagram
$\coprod_x M_x\rarrow\coprod_x N_x\rarrow\varinjlim_x N_x$, \
$\coprod_x M_x\rarrow\varinjlim_x M_x\rarrow\varinjlim_x N_x$ that
the morphism $\varinjlim_x M_x\rarrow\varinjlim_x N_x$ is also
a split epimorphism (cf.~\cite[proof of
Theorem~1.4\,(2)\,$\Rightarrow$\,(3)]{AS}).
\end{proof}

 We will say that an additive category $\sA$ with $X$\+sized coproducts
and $X$\+shaped direct limits has \emph{split $X$\+direct limits} if
the condition of Lemma~\ref{split-X-direct-limits-conditions}\,(1)
is satisfied, that is, the direct limit of $X$\+shaped diagrams in $\sA$
preserves split monomorphisms.

\begin{lem} \label{subcategory-split-direct-limits-lemma}
 Let\/ $\sB$ be an idempotent-complete additive category with $X$\+sized
coproducts and\/ $\sA\subset\sB$ be a full subcategory closed under
direct summands and $X$\+sized coproducts.
 In this setting \par
\textup{(a)} if\/ $\sB$ has split $X$\+direct limits, then so
does\/~$\sA$; \par
\textup{(b)} if\/ $\sA$ has split $X$\+direct limits, then\/ $\sA$
is closed under $X$\+shaped direct limits in\/~$\sB$.
\end{lem}

\begin{proof}
 Part~(a): let $N=(N_x\to N_y)_{x<y}$ be an $X$\+shaped diagram in~$\sA$.
 By assumption, the direct limit of the diagram $N$ exists in~$\sB$ and
is isomorphic to a direct summand of the coproduct $\coprod_{x\in X}N_x$.
 Since $\sA$ is closed under direct summands and $X$\+sized coproducts
in $\sB$, the direct limit of the diagram $N$ computed in $\sB$ belongs
to~$\sA$.
 It follows that the direct limit of the diagram $N$ exists in $\sA$ and
coincides with the direct limit of the same diagram in~$\sB$.

 Part~(b): once again, let $N=(N_x\to N_y)_{x<y}$ be an $X$\+shaped
diagram in~$\sA$.
 Consider the diagram $M=(M_x\to M_y)_{x<y}$ constructed in the proof
of Lemma~\ref{split-X-direct-limits-conditions}\,(2)\,$\Rightarrow$\,(3)
and denote by $K=(K_x\to K_y)_{x<y}$ the kernel of the natural termwise
split epimorphism of diagrams $M\rarrow N$. It follows from the construction
that $K_z\cong\coprod_{x<z}N_x$ for each $z\in X$.
 Applying the same construction to the diagram $K$, we obtain a termwise
split epimorphism of diagrams $L\rarrow K$, where
$$
 L_y=\coprod\nolimits_{z\le y}K_z\cong
     \coprod\nolimits_{z\le y}\coprod\nolimits_{x<z}N_x=
		 \coprod\nolimits_{x<z\le y}N_x.
$$
 Now in both the categories $\sA$ and $\sB$, the diagram $N$ is
the cokernel of the composition of morphisms of diagrams $L\rarrow K
\rarrow M$.
 Furthermore, in both the categories $\sA$ and $\sB$ we have
$\varinjlim_x M_x=\coprod_x N_x$ and
$\varinjlim_y L_y=\coprod_y K_y\cong\coprod_{x<y} N_x$.
 The morphism of diagrams $L\rarrow M$ induces a morphism of their direct
limits
$$
 f\:\coprod\nolimits_{x<y} N_x\cong\varinjlim_x L_x\rarrow
 \varinjlim_x M_x\cong\coprod\nolimits_x N_x,
$$
which coincides with the leftmost morphism
in the sequence~\eqref{presentation-direct-limit}.

 Now, for any of our two categories $\sC=\sA$ or $\sC=\sB$, the direct
limit of the diagram $N$ in $\sC$ exists if and only if the cokernel of 
the morphism~$f$ exists in $\sC$, and if they do exist then they
coincide, $\varinjlim^\sC_xN_x = \coker^\sC(f)$.
 In particular, $X$\+shaped direct limits exist in $\sA$ by assumption,
so we have an object $\coker^\sA(f)=\varinjlim^\sA_x N_x$.
 Moreover, since we are assuming that $\sA$ has split $X$\+direct
limits, in view of
Lemma~\ref{split-X-direct-limits-conditions}\,(1)\,$\Rightarrow$\,(2)
the morphism~$f$ is the composition of a split epimorphism
$\varinjlim_x L_x\rarrow \varinjlim_x K_x$ and a split monomorphism
$\varinjlim_x K_x\rarrow\varinjlim_x M_x$ in~$\sA$.
 Therefore, the cokernel of~$f$ also exists in $\sB$ and coincides with
the cokernel of~$f$ in~$\sA$, that is $\coker^\sB(f)=\coker^\sA(f)$.
 Finally, we have $\varinjlim_x^\sB N_x=\coker^\sB(f)$.
\end{proof}

 As in the preceding proof, we will use the notation
$\varinjlim_{x\in X}^\sA$ for the direct limit in a category $\sA$
(when the category is not clear from the context).
 Now we can formulate our categorical version 
of~\cite[Theorem~1.4\,(2)\,$\Leftrightarrow$\,(3)]{AS}
in a way which resembles the original result.

\begin{cor} \label{ab5-split-direct-limits-cor}
 Let $X$ be a directed poset, $\sB$ be an abelian category with
$X$\+sized coproducts and exact functors of $X$\+shaped direct limits,
and\/ $\sA\subset\sB$ be a full subcategory closed under direct
summands and $X$\+sized coproducts.
 Then the following conditions are equivalent:
\begin{enumerate}
\item the additive category\/ $\sA$ has $X$\+split direct limits;
\item the direct limit in\/ $\sB$ of any $X$\+shaped diagram of split
monomorphisms in\/ $\sA$ is a split monomorphism;
\item the direct limit in\/ $\sB$  of any $X$\+shaped diagram of split
epimorphisms in\/ $\sA$ is a split epimorphism; \par
\item for any $X$\+shaped diagram $(N_x\to N_y)_{x<y\in X}$ in\/ $\sA$,
the natural epimorphism\/ $\coprod_{x\in X} N_x\rarrow
\varinjlim_{x\in X}^\sB N_x$ in the category\/ $\sB$ is split.
\end{enumerate}
 If any one of these conditions holds, then\/ $\sA$ is closed under
$X$\+shaped direct limits in\/~$\sB$ (so the $X$\+shaped direct limits
in\/ $\sA$ and\/ $\sB$ agree).
\end{cor}

\begin{proof}
 (1)\,$\Longrightarrow$\,(2)
 Assume that $\sA$ has $X$\+split direct limits.
 Then, by Lemma~\ref{subcategory-split-direct-limits-lemma}(b),
$\sA$ is closed under $X$\+shaped direct limits in~$\sB$.
 Hence (2)~follows by the definition of $\sA$ having $X$\+split
direct limits.

 (2)\,$\Longleftrightarrow$\,(3) holds since the direct limit in $\sB$
of any $X$\+shaped diagram of split short exact sequences in $\sA$ is
a short exact sequence.
 Only here in the implication (3)\,$\Longrightarrow$\,(2) we are
using the assumption that $X$\+shaped direct limits in $\sB$ are exact.

 (3)\,$\Longleftrightarrow$\,(4) is provable in the same way as
Lemma~\ref{split-X-direct-limits-conditions}\,(2)\,$\Leftrightarrow$\,(3)
(with the direct limits of $X$\+shaped diagrams in $\sA$ computed
in $\sA$ replaced by the direct limits of $X$\+shaped diagrams in $\sA$
computed in~$\sB$).

 (2)$\,+\,$(4)\,$\Longrightarrow$\,(1) In the context of~(4),
since $\sA$ is closed under $X$\+sized coproducts in $\sB$ by
assumption, the coproduct $\coprod_{x\in X}N_x$ is the same in $\sA$
and in~$\sB$.
 Being a direct summand of this coproduct, the object
$\varinjlim_{x\in X}^\sB N_x$ consequently also belongs to~$\sA$.
 Hence $\sA$ is closed under $X$\+shaped direct limits in~$\sB$.
 Now (2)~tells that $\sA$ has $X$\+split direct limits.
\end{proof}

 Given an additive category $\sA$ with set-indexed coproducts, we will
say that $\sA$ has \emph{split direct limits} if it has $X$\+split
direct limits for every directed poset~$X$.
 It is clear from~\cite[Lemma~1.6]{AR} that a category $\sA$ has split
direct limits whenever it has $X$\+split direct limits for every
\emph{linearly ordered} (or even well-ordered) index set~$X$.

\begin{ex} \label{locally-Noetherian-example}
 Let $\sA$ be a locally Noetherian Grothendieck abelian category and
$\sA_\inj\subset\sA$ be the full subcategory of injective objects
in~$\sA$.
 Then the category $\sA_\inj$ has split direct limits.
 Indeed, the full subcategory $\sA_\inj$ is closed under direct limits
in $\sA$, hence all (coproducts and) direct limits exist in~$\sA_\inj$.
 Furthermore, the direct limit of any directed diagram of split
monomorphisms in $\sA_\inj$ is a monomorphism in $\sA$ between two
objects of~$\sA_\inj$.
 Clearly, any such monomorpism is split.

 In view of Lemma~\ref{subcategory-split-direct-limits-lemma}(a), it
follows that, for any injective object $J\in\sA$, the full subcategory
$\Add(J)\subset\sA$ has split direct limits.
\end{ex}

 In the rest of this section, we discuss split direct limits in
the full subcategory $\Add(M)$ of direct summands of coproducts of
copies of a fixed object $M$ in a topologically agreeable abelian
category~$\sA$.
 The special case in which $\sA=A\modl$ is the category of modules
over an associative ring plays an important role.
 The aim is to interpret the condition of split direct limits in
$\Add(M)$ as a ``naturally sounding'' property of the related
contramodule category $\Hom_\sA(M,M)^\rop\contra$.

\begin{lem} \label{phi-colimits}
 Let\/ $\sA$ be a cocomplete abelian category,
$M\in\sA$ be an object, and\/ $\sB$ be
the abelian category with enough projective objects for which\/
$\sB_\proj\cong\Add(M)\subset\sA$, as in
Theorem~\textup{\ref{generalized-tilting}(a)}.
 Let $X$ be a small category.
 Assume that the full subcategory\/ $\sB_\proj$ is closed under
$X$\+shaped colimits in\/~$\sB$.
 Then the full subcategory\/ $\Add(M)$ is closed under $X$\+shaped
colimits in\/ $\sA$, and for any $X$\+shaped diagram
 $(N_x\to N_y)_{x\to y\in X}$
in\/ $\Add(M)$ the natural epimorphism\/
$\coprod_{x\in X} N_x\rarrow \varinjlim_{x\in X}^\sA N_x$ in
the category\/ $\sA$ is split.
\end{lem}

\begin{proof}
 According to the discussion near the end of
Section~\ref{topologically-agreeable-secn},
the equivalence of categories $\sB_\proj\cong\Add(M)$ can be extended
to a pair of adjoint functors $\Psi\:\sA\rarrow\sB$ and $\Phi\:\sB
\rarrow\sA$.
 Given an $X$\+shaped diagram $N\in\Add(M)^X$, consider the diagram
$P=\Psi(N)\in(\sB_\proj)^X$.
 Then we have $N=\Phi(P)$.
 Since the functor $\Phi$ is a left adjoint, it preserves colimits,
and follows that
$$
 \varinjlim\nolimits_{x\in X}N_x=
 \Phi\left(\varinjlim\nolimits_{x\in X} P_x\right)
 \quad\text{and}\quad
 \coprod\nolimits_{x\in X}N_x=
 \Phi\left(\coprod\nolimits_{x\in X}P_x\right).
$$
 Now we have $\varinjlim_{x\in X}P_x\in\sB_\proj$ by assumption,
hence $\varinjlim_{x\in X}N_x\in\Add(M)$.
 Moreover, the natural morphism $f\:\coprod_{x\in X}P_x\rarrow
\varinjlim_{x\in X}P_x$ is an epimorphism between projective
objects in $\sB$, hence a split epimorphism.
 Thus the morphism $\Phi(f)\:\coprod_{x\in X}N_x\rarrow
\varinjlim_{x\in X}N_x$ is a split epimorphism in~$\sA$
(and in~$\Add(M)$).
\end{proof}

\begin{prop} \label{split-direct-limits-of-projective-contramodules}
 Let\/ $\R$ be a complete, separated topological ring with a right
linear topology, and let $X$ be a directed poset.
 Then the category of projective left\/ $\R$\+contramodules\/
$\R\contra_\proj$ has $X$\+split direct limits if and only if the full
subcategory\/ $\R\contra_\proj$ is closed under $X$\+shaped
direct limits in\/ $\R\contra$.
\end{prop}

\begin{proof}
 ``Only if'' holds by
Lemma~\ref{subcategory-split-direct-limits-lemma}(b).
 To prove ``if'', we use
Corollary~\ref{topological-rings-are-endomorphism-rings}.
 Let $A$ be an associative ring and $M$ a left $A$\+module such that
$\R$ is isomorphic to $\Hom_A(M,M)^\rop$ as a topological ring.
 By Theorem~\ref{generalized-tilting}(a,c), we have
$\Add(M)\cong\R\contra_\proj$.
 By Lemma~\ref{phi-colimits}, it follows that for any $X$\+shaped
diagram $N$ in $\Add(M)$ the natural epimorphism of left $A$\+modules
$\coprod_{x\in X} N_x\rarrow\varinjlim_{x\in X}^{A\modl} N_x$ is split.
 Applying
Corollary~\ref{ab5-split-direct-limits-cor}\,(4)\,$\Rightarrow$\,(1)
to the full subcategory $\Add(M)$ in the Ab5-category $A\modl$, we
conclude that the category $\Add(M)$ has split direct limits.
 Hence so does the category $\R\contra_\proj\cong\Add(M)$.
\end{proof}

\begin{rem}
 The above proof is quite indirect.
 The problem is that direct limits are not exact in the abelian category
$\R\contra$, so Corollary~\ref{ab5-split-direct-limits-cor} is not
applicable to $\sA=\R\contra_\proj$ and $\sB=\R\contra$.
 It would be interesting to know whether the assertion of
Proposition~\ref{split-direct-limits-of-projective-contramodules} holds
for abelian categories of more general nature than the categories of
contramodules over topological rings---such as, e.~g., cocomplete
abelian categories $\sB$ with a projective generator in which
the additive category of projective objects $\sB_\proj$ is agreeable.
(Cf.\ Example~\ref{ring-summation-structure-counterex} below.)
\end{rem}

\begin{ex}
 The following simple counterexample purports to show why
Proposition~\ref{split-direct-limits-of-projective-contramodules}
is nontrivial.
 Let $k$ be a field and $\sA=\sB=k\vect^\sop$ be the opposite
category to the category of $k$\+vector spaces.
 This category is abelian, all its objects are projective, and
all short exact sequences split.
 So $\sA$ satisfies the conditions of
Lemma~\ref{split-X-direct-limits-conditions}\,(2\<3) and
$\sA\subset\sB$ satisfies the conditions of
Corollary~\ref{ab5-split-direct-limits-cor}\,(3\<4).
 But direct limits are not exact in $\sA=\sB$, so the conditions
in Lemma~\ref{split-X-direct-limits-conditions}\,(1) and
Corollary~\ref{ab5-split-direct-limits-cor}\,(1\<2) are \emph{not}
satisfied and $\sA$ does not have split direct limits.
\end{ex}

\begin{cor} \label{split-direct-limits-contramodule-interpretation}
 Let\/ $\sA$ be an idempotent-complete topologically agreeable additive
category, $M\in\sA$ be an object, and\/ $\R=\Hom_\sA(M,M)^\rop$ be
the ring of endomorphisms of $M$, endowed with its complete, separated,
right linear topology.
 Let $X$ be a directed poset.
 Then the full subcategory\/ $\Add(M)\subset\sA$ has $X$\+split
direct limits if and only if the full subcategory\/
$\R\contra_\proj$ is closed under $X$\+shaped direct limits in\/
$\R\contra$.
\end{cor}

\begin{proof}
 By Theorem~\ref{generalized-tilting}(a,c), we have an equivalence
of additive categories $\Add(M)\cong\R\contra_\proj$.
 So the assertion follows from
Proposition~\ref{split-direct-limits-of-projective-contramodules}.
\hbadness=1575
\end{proof}

\Section{Objects with Perfect Decompositions}
\label{perfect-decompositions-secn}

 Let $\sA$ be an agreeable category.
 A family of nonzero objects $(M_z\in\sA)_{z\in Z}$ is said to be
\emph{locally T\+nilpotent} if for any sequence of indices $z_1$,
$z_2$, $z_3$,~\dots~$\in Z$ and any sequence of nonisomorphisms
$f_i\in\Hom_\sA(M_{z_i},M_{z_{i+1}})$ the family of morphisms
$f_nf_{n-1}\dotsm f_1\:M_{z_1}\rarrow M_{z_{n+1}}$, \
$n=1$, $2$, $3$,~\dots, is summable (i.~e., corresponds to
a morphism $f\:M_{z_1}\rarrow\coprod_{n=1}^\infty M_{z_{n+1}}$;
see Section~\ref{topologically-agreeable-secn}).

 Notice that the ring of endomorphisms of every object in
a locally T\+nilpotent family is necessarily local, as for any
noninvertible endomorphism $h\in\Hom_\sA(M_z,M_z)$ a morphism
$(1-h)^{-1}=\sum_{n=0}^\infty h^n\:M_z\rarrow M_z$ exists.
 So nonautomorphisms form a two-sided ideal in $\Hom_\sA(M_z,M_z)$.
 It follows that the object $M_z$ is indecomposable.

 An object $M\in\sA$ is said to have a \emph{perfect decomposition}
if there exists a locally T\+nilpotent family of objects
$(M_z)_{z\in Z}$ in $\sA$ such that $M\cong\coprod_{z\in Z} M_z$.

 In view of the discussion in the first half of
Section~\ref{split-direct-limits-secn}, the result
of~\cite[Theorem~1.4]{AS} can be formulated in our language as follows.

\begin{thm}[{\cite[Theorem~1.4]{AS}}] \label{angeleri-saorin}
 Let $A$ be an associative ring and $M$ be a left $A$\+module.
 Then $M$ has a perfect decomposition if and only if the full
subcategory\/ $\Add(M)\subset A\modl$ has split direct limits.
\end{thm}

\begin{proof}
 A proof of this theorem can be found in~\cite{AS}, based on
preceding results from the books~\cite{Har,MM} and particularly
from the paper~\cite{GG}.
 An alternative proof of the (easy) implication ``only if'' based
on contramodule methods is suggested below in
Remark~\ref{contramodule-proof-of-a-s-remark}.
\end{proof}

 In this section we prove the following categorical extension
of Theorem~\ref{angeleri-saorin}.

\begin{thm} \label{top-agreeable-angeleri-saorin}
 Let\/ $\sA$ be an idempotent-complete topologically agreeable
additive category and $M\in\sA$ be an object.
 Then $M$ has a perfect decomposition if and only if the full
subcategory\/ $\Add(M)\subset\sA$ has split direct limits.
\end{thm}

 Let us explain the plan of our proof of
Theorem~\ref{top-agreeable-angeleri-saorin}.
 In the second half of Section~\ref{split-direct-limits-secn} we
have already interpreted the condition of split direct limits in
$\Add(M)$ in terms of the topological ring $\R=\Hom_\sA(M,M)^\rop$.
 In this section we similarly interpret the condition of existence
of a perfect decomposition $M\cong\coprod_{z\in Z}M_z$ of an object~$M$.
 Then we translate Theorem~\ref{angeleri-saorin} into the topological
ring language, and finally go back to the greater generality of
Theorem~\ref{top-agreeable-angeleri-saorin}.

 Let $\R$ be a complete, separated topological ring with a right linear
topology.
 The notion of the \emph{topological Jacobson radical} $\HH(\R)$ of
a topological ring $\R$ was discussed in
the papers~\cite[Section~3.B]{IMR} and~\cite[Section~7]{Pproperf}.
 We recall that, by the definition, $\HH(\R)$ is the intersection of
all open maximal right ideals in~$\R$.
 The topological Jacobson radical $\HH(\R)$ is a closed two-sided
ideal in~$\R$ \,\cite[Lemma~7.1]{Pproperf}.

\begin{lem} \label{H-is-topological-Jacobson}
 Suppose that\/ $\HH\subset\R$ is a topologically nil closed two-sided
ideal in\/ $\R$ such that the quotient ring\/ $\R/\HH$ with its quotient
topology is (complete and) topologically semisimple.
 Then the ideal\/ $\HH\subset\R$ coincides with the topological
Jacobson radical of the ring\/ $\R$ and with the Jacobson radical of
the abstract ring $\R$ (with the topology ignored), that is\/
$\HH=\HH(\R)=H(\R)$.
\end{lem}

\begin{proof}
 This is a straightforward generalization of~\cite[Lemma~9.1]{Pproperf}.
 One observes that every nonzero element of the quotient ring
$\S=\R/\HH$ acts nontrivially on a simple discrete right $\S$\+module
(see Remark~\ref{simple-objects-made-explicit}), and then argues as
in \emph{loc.\ cit.}
\end{proof}

 The following definition, containing a list of structural conditions on
a topological ring $\R$, plays the key role.
 We say that a complete, separated topological ring $\R$ with right
linear topology is \emph{topologically left perfect} if the next
three conditions hold:
\begin{enumerate}
\item the topological Jacobson radical\/ $\HH(\R)$ of the topological
ring\/ $\R$ is topologically left T\+nilpotent;
\item $\HH(\R)$ is a strongly closed subgroup/ideal in\/~$\R$; and
\item the quotient ring\/ $\S=\R/\HH(\R)$ with its quotient topology is
topologically semisimple.
\end{enumerate}

\begin{thm} \label{perfect-decomposition-endomorphism-ring}
 Let\/ $\sA$ be an idempotent-complete topologically agreeable
additive category, $M\in\sA$ be an object, and\/ $\R=\Hom_\sA(M,M)^\rop$
be the ring of endomorphisms of $M$, endowed with its complete,
separated, right linear topology.
 Then the object $M\in\sA$ has a perfect decomposition if and only if
the topological ring\/ $\R$ is topologically left perfect.
\end{thm}

\begin{rem} \label{semiregular-remark}
 A relevant concept in the abstract (nontopological) ring theory is
that of a semiregular ring.
 An associative ring $R$ is said to be \emph{semiregular} if its
quotient ring by the Jacobson radical $S=R/H(R)$ is von~Neumann regular 
and idempotents can be lifted modulo~$H(R)$.

 Any topologically left perfect topological ring is semiregular as
an abstract ring.
 Indeed, one has $H(\R)=\HH(\R)$ by
Lemma~\ref{H-is-topological-Jacobson}; the topologically semisimple
quotient ring $\S=\R/\HH(\R)$ is von~Neumann regular, as it was already
mentioned in Remark~\ref{top-semisimple-von-Neumann-regular}
(cf.\ Lemma~\ref{split-endomorphisms-vNregular} below);
and idempotents can be lifted modulo a topologically nil closed
two-sided ideal by Lemma~\ref{single-idempotent-lifted}
(see also Corollary~\ref{lifting-orthogonal-cor}).

 Semiregular rings appear in connection with modules with perfect
decompositions and in related contexts
in~\cite[Theorem~7.3.15\,(6)]{Har},
\cite[Propositions~4.1 and~4.2\,(1)]{Ang},
and~\cite[Theorem~1.1\,(4)]{AS}.
 The above theorem provides a more precise description of the class
of rings appearing as the endomorphism rings of modules/objects with
perfect decompositions in terms of their topological ring structures.
\end{rem}

 Before proceeding to prove
Theorem~\ref{perfect-decomposition-endomorphism-ring}, let us
formulate and prove three lemmas about right topological/topologically
agreeable additive categories in the spirit of
Lemmas~\ref{right-topological-coproduct-lemma}\<%
\ref{agreeable-agreement-lemma}.

\begin{lem} \label{right-topological-subgroup-lemma}
 Let\/ $\sA$ be a right topological additive category, $K$, $L$,
and $M\in\sA$ be three objects, and $(f_x\:K\to L)_{x\in X}$
be a family of morphisms converging to zero in the group\/
$\Hom_\sA(K,L)$.
 Let\/ $\U\subset\Hom_\sA(K,M)$ be an open subgroup.
 Then one has\/ $\Hom_\sA(L,M)f_x\subset\U$ for all but a finite
subset of indices $x\in X$.
\end{lem}

\begin{proof}
 Set $N=L\oplus M\in\sA$.
 Let $j\:L\rarrow N$ be the coproduct injection and $q\:N\rarrow M$
be the product projection.
 By axiom~(i), the family of morphisms $(jf_x\:K\to N)_{x\in X}$
converges to zero in the group $\Hom_\sA(K,N)$.
 By axiom~(ii), it follows that for any open subgroup $\V\subset
\Hom_\sA(K,N)$ we have $\Hom_\sA(N,N)jf_x\subset\V$ for all but
a finite subset of indices $x\in X$.
 By axiom~(i), for every open subgroup $\U\subset\Hom_\sA(K,M)$
there exists an open subgroup $\V\subset\Hom_\sA(K,N)$ such that
$q\V\subset\U$.
 Thus we have $q\Hom_\sA(N,N)jf_x\subset\U$ for all but a finite
subset of indices $x\in X$; and it remains to observe that
$q\Hom_\sA(N,N)j=\Hom_\sA(L,M)$.
\end{proof}

\begin{lem} \label{hom-from-coproduct-topology}
 Let\/ $\sA$ be a topologically agreeable category,
$(L_x\in\sA)_{x\in X}$ be a family of objects in\/ $\sA$,
and $M\in\sA$ be an object.
 Then the natural isomorphism\/ $\Hom_\sA(\coprod_{x\in X}L_x,M)
\cong\prod_{x\in X}\Hom_\sA(L_x,M)$ is an isomorphism of topological
abelian groups (where the right-hand side is endowed with
the product topology).
\end{lem}

\begin{proof}
 By axiom~(i), in any right topological additive category $\sA$
where the coproduct $\coprod_{x\in X}L_x$ exists, the natural map
$\Hom_\sA(\coprod_{x\in X}L_x,M)\rarrow\Hom_\sA(L_y,M)$ is
continuous for every $y\in X$.
 Hence so is the map to the product $\Hom_\sA(\coprod_{x\in X}L_x,M)
\rarrow\prod_{x\in X}\Hom_\sA(L_x,M)$.
 It remains to check continuity in the opposite direction, and
that is where we will need the assumption that $\sA$ is
topologically agreeable.

 Set $L=\coprod_{x\in X}L_x$.
 For every $x\in X$, let $\iota_x\in\Hom_\sA(L_x,L)$ be the coproduct
injection and $\pi_x\in\Hom_\sA(L,L_x)$ be the natural projection.
 Then the family of morphisms $(\pi_x\:L\to L_x)_{x\in X}$ is
summable in the agreeable category~$\sA$ (as these are the components
of the identity morphism $\id_L\:L\rarrow L$).
 Hence the family of morphisms $(\iota_x\pi_x\:L\to L)_{x\in X}$ is
summable, too.

 Since $\sA$ is topologically agreeable, it follows that the family of
projectors $(\iota_x\pi_x)_{x\in X}$ converges to zero in the topology
of the group $\Hom_\sA(L,L)$.
 By Lemma~\ref{right-topological-subgroup-lemma} (applied to
the objects $K=L$ and~$M$), every open subgroup $\U$ in the topological
group $\Hom_\sA(L,M)$ contains the subgroup $\Hom_\sA(L,M)\iota_x\pi_x=
\Hom_\sA(L_x,M)\pi_x$ for all but a finite subset of indices $x\in X$.
 In other words, one can say that the whole family of subgroups
$\Hom_\sA(L_x,M)\pi_x$ converges to zero in $\Hom_\sA(L,M)$.
 It follows that the map $\prod_{x\in X}\Hom_\sA(L_x,M)\rarrow
\Hom_\sA(L,M)$ assigning to a family of morphisms
$(f_x\:L_x\to M)_{x\in X}$ the morphism $\sum_{x\in X}f_x\pi_x$
is continuous.
\end{proof}

\begin{lem} \label{hom-to-copower-topology}
 Let\/ $\sA$ be a topologically agreeable category, $M$ and $N\in\sA$
be two objects, and $X$ be a set.
 Then the topological group\/ $\Hom_\sA(M,N^{(X)})$ is isomorphic to
the topological group\/ $\Hom_\sA(M,N)[[X]]$ (where for any
complete, separated topological abelian group\/ $\A$, the group\/
$\A[[X]]$ is endowed with the projective limit topology of\/
$\A[[X]]=\varprojlim_{\U\subset\A}(\A/\U)[X]$, with\/ $\U$ ranging
over the open subgroups of\/~$\A$).
\end{lem}

\begin{proof}
 For any right topological additive category $\sA$, the abelian
group homomorphism $\Hom_\sA(M,N)[[X]]\rarrow\Hom_\sA(M,N^{(X)})$
assigns to every zero-convergent family of morphisms
$(f_x\:M\to N)_{x\in X}$ the morphism $\sum_{x\in X}\iota_xf_x\:
M\rarrow N^{(X)}$ (where $\iota_x\:N\rarrow N^{(X)}$ are
the coproduct injections).
 The property of this map to be an isomorphism of abelian groups is
a restatement of the definition of a topologically agreeable
category~$\sA$.
 The lemma claims that this is a topological isomorphism; so we have
to prove continuity in both directions.

 Our map decomposes as $\Hom_\sA(M,N)[[X]]\rarrow
\Hom_\sA(M,N^{(X)})[[X]]\rarrow\Hom_\sA(M,N^{(X)})$.
 Here the first map assigns to a zero-convergent family of morphisms
$(f_x\:M\to N)_{x\in X}$ the family of morphisms
$(\iota_xf_x\:M\to N^{(X)})_{x\in X}$, which is zero-convergent by
Lemma~\ref{right-topological-coproduct-lemma}.
 One can see from the proof of
Lemma~\ref{right-topological-coproduct-lemma} that this map is
continuous.
 The second map assigns to a zero-convergent family of elements
$(g_x\in\A)_{x\in X}$ of the topological group
$\A=\Hom_\sA(M,N^{(X)})$ their sum $\sum_{x\in X}g_x$ in~$\A$.
 This map is continuous for every complete, separated topological
abelian group~$\A$.
 Hence the composition is continuous, too. {\hbadness=1675\par}

 To check that the inverse map $\Hom_\sA(M,N^{(X)})\rarrow
\Hom_\sA(M,N)[[X]]$ is continuous, suppose we are given an open
subgroup $\U\subset\Hom_\sA(M,N)$.
 Choose a fixed element $x_0\in X$, and consider the subgroup
$\V_{\U,x_0}\subset\Hom_\sA(M,N^{(X)})$ consisting of all
the morphisms $f\:M\rarrow N^{(X)}$ for which the composition
$\pi_{x_0}f$ (where $\pi_x\:N^{(X)}\rarrow N$ is the natural
projection) belongs to~$\U$.
 By the continuity axiom~(i), $\V_{\U,x_0}$ is an open subgroup in
$\Hom_\sA(M,N^{(X)})$.
 By axiom~(ii), there exists an open
$\Hom_\sA(N^{(X)},N^{(X)})$\+submodule
$\V\subset\Hom_\sA(N,N^{(X)})$ such that $\V\subset\V_{\U,x_0}$.

 Now for every morphism $f\in\V$ we have $\sigma_{x_0,x}f\in\V$
(where $\sigma_{x,y}\:N^{(X)}\rarrow N^{(X)}$ is the automorphism
permuting the coordinates $x$ and~$y$).
 Hence $\pi_xf=\pi_{x_0}\sigma_{x_0,x}f\in\U$.
 We have shown that the full preimage of the open subgroup
$\U[[X]]\subset\Hom_\sA(M,N)[[X]]$ under the map $\Hom_\sA(M,N^{(X)})
\rarrow\Hom_\sA(M,N)[[X]]$ contains an open subgroup
$\V\subset\Hom_\sA(M,N^{(X)})$; so this map is continuous.
\end{proof}

\begin{cor} \label{endomorphisms-of-copower}
 Let\/ $\sA$ be a topologically agreeable category, $M\in\sA$ be
an object, and $Y$ be a set.
 Let\/ $\R=\Hom_\sA(M,M)^\rop$ be the topological ring of endomorphisms
of~$M$.
 Then the topological ring\/ $\Hom_\sA(M^{(Y)},M^{(Y)})^\rop$ is
naturally isomorphic to the ring of row-zero-convergent matrices\/
$\Mat_Y(\R)$ with the topology defined in
Section~\ref{matrix-topology-secn}.
\end{cor}

\begin{proof}
 The ring isomorphism can be easily established; and the description
of the topology is provided
Lemmas~\ref{hom-from-coproduct-topology}\<\ref{hom-to-copower-topology}.
\end{proof}

\begin{proof}[Proof of the implication ``only if'' in
Theorem~\ref{perfect-decomposition-endomorphism-ring}]
 More generally, for any family of objects $(M_z)_{z\in Z}$
in an agreeable category $\sA$, the ring of endomorphisms
$R=\Hom_\sA(M,\allowbreak M)^\rop$ of their coproduct
$M=\coprod_{z\in Z} M_z$ can be described as the ring of row-summable 
matrices of morphisms $r=(r_{w,z}\:M_w\to M_z)_{w,z\in Z}$.
 Here the words ``row-summable'' mean that for every fixed index
$w\in Z$ the family of morphisms $(r_{w,z}\:M_w\to M_z)_{z\in Z}$
must be summable in the agreeable category~$\sA$.

 Assuming that the family of objects $(M_z)_{z\in Z}$ is locally
T\+nilpotent, we will denote by $X$ the set of isomorphism classes
of the objects~$M_z$.
 So we have a natural surjective map $Z\rarrow X$ assigning to every
module in the family its isomorphism class.
 For every element $x\in X$, let $Y_x\subset Z$ denote the full
preimage of the element~$x$ under this map; so the set $Z$ is
the disjoint union of nonempty sets~$Y_x$.

 Assume further that the category $\sA$ is topologically agreeable.
 Denote by $\HH\subset\R$ the subset of all \emph{matrices of
nonisomorphisms} in the ring $\R$; that is, an element
$h=(h_{w,z})_{w,z\in Z}\in\R$ belongs to $\HH$ if and only if,
for every pair of indices $w$, $z\in Z$, the morphism~$h_{w,z}$
is not an isomorphism.
 Our next aim is to show that $\HH$ is a closed two-sided ideal in~$\R$.

 Denote by $\HH_{w,z}\subset\R_{w,z}$ the subset of nonisomorphisms
in the topological group of morphisms $\R_{w,z}=\Hom_\sA(M_w,M_z)$.
 Following the discussion in the beginning of this section,
$\HH_{z,z}$ is a two-sided ideal in~$\R_{z,z}$.
 It follows that, for all $w$ and $z\in Z$, the subset $\HH_{w,z}
\subset\R_{w,z}$ is an open subgroup.
 Indeed, if the elements $w$ and $z\in Z$ do \emph{not} belong to
the same equivalence class $Y_x$, \,$x\in X$, then we have
$\HH_{w,z}=\R_{w,z}$.
 If there is $x\in X$ such that both the elements $w$
and~$z$ belong to $Y_x$, then choosing an isomorphism $M_w\cong M_z$
identifies $\R_{w,z}$ with $\R_{z,z}$ and $\HH_{w,z}$ with $\HH_{z,z}$.
 Any proper open right ideal in the topological ring $\R_{z,z}$
consists of nonisomorphisms.
 As the topology on $\R_{z,z}$ is separated, it follows that
the ideal of nonisomorphisms $\HH_{z,z}\subset\R_{z,z}$ is open.
 Hence $\HH_{w,z}\subset\R_{w,z}$ is an open subgroup in both cases.

 The projection map $\R\rarrow\R_{w,z}$, assigning to a matrix
$(a_{w,z})_{w,z\in Z}\in\R$ its matrix entry at the position
$(w,z)\in Z\times Z$, is a continuous group homomorphism (by
the continuity axiom~(i) of a right topological additive category).
 The subset $\HH\subset\R$ is the intersection of the full preimages
of the open sugroups $\HH_{w,z}\subset\R_{w,z}$ over all the pair of
indices $(w,z)$.
 It follows that $\HH$ is a closed subgroup in~$\R$.

 The multiplication in the ring $\R$ is computable as follows.
 Let $a=(a_{v,w})_{v,w\in Z}$ and $b=(b_{w,z})_{w,z\in Z}$ be two
matrices representing some elements of~$\R$.
 Then, for every fixed $v$ and $z\in Z$, the family of morphisms
$(a_{v,w}\:M_v\to M_w)_{w\in Z}$ is summable, and it follows that
the family of morphisms $(b_{w,z}a_{v,w}\:M_v\to M_z)_{w\in Z}$
is summable, too.
 In view of Lemma~\ref{agreeable-agreement-lemma}, the product of
two elements $c=ab\in\R$ is represented by the matrix
$(c_{v,z})_{v,z\in Z}$ with the entries
$$
 c_{v,z}=\sum\nolimits_{w\in Z}(b_{w,z}a_{v,w})_{v,z},
$$
where the infinite summation sign means the limit of finite partial
sums in the topology of $\R_{v,z}$.
 Furthermore, the composition of any two nonisomorphisms between
objects $M_z$, \,$z\in Z$ is not an isomorphism, since these objects
are indecomposable; while the composition of an isomorphism and
a nonisomorphism (in any order) is obviously a nonisomorphism.
 Therefore, $\HH\subset\R$ is a two-sided ideal.

 Let us show that $\HH$ is topologically left T\+nilpotent.
 Let $E\subset\R$ be the set of all matrices with a single nonzero
entry which is not an isomorphism.
 Using Lemmas~\ref{hom-from-coproduct-topology}
and~\ref{hom-to-copower-topology}, one can see that the subgroup
generated by $E$ is dense in~$\HH$.
 The condition of local T\+nilpotency of the family of objects
$(M_z)_{z\in Z}$ can be equivalently restated by saying that
the subset of elements $E\subset\R$ is topologically left T\+nilpotent.
 Applying Lemma~\ref{closed-subring-top-T-nilpotent}, we conclude
that the whole ideal $\HH\subset\R$ is topologically left
$T$\+nilpotent.
 
 Our next aim is to describe the quotient ring $\S=\R/\HH$.
 The elements of $\S$ are represented by matrices with entries in
the quotient groups of the groups $\R_{w,z}=\Hom_\sA(M_w,M_z)$ by
the subgroups of nonisomorphisms $\HH_{w,z}\subset\R_{w,z}$.
 These are block matrices supported in the subset
$\coprod_{x\in X}Y_x\times Y_x\subset Z\times Z$.
 For every $x\in X$, the related block $\S_x$ can be described as
a similar quotient ring of the topological ring of endomorphisms
of the object $\coprod_{y\in Y_x}M_y$.

 According to the above discussion, the ideal $\HH_{y,y}\subset\R_{y,y}$
is open and the topology on the quotient ring $D_y=\R_{y,y}/\HH_{y,y}$ 
is discrete.
 All the rings $\R_{y,y}$, \,$y\in Y_x$ are isomorphic to each other,
since the objects $M_y$ are; hence so are all the rings~$D_y$
for $y\in Y_x$.
 Choosing isomorphisms between the objects $M_y$ in a compatible way,
we can put $\R_x=\R_{y,y}$ and $D_x=D_y$.
 So $D_x$, \,$x\in X$, are some discrete skew-fields.
 
 By Corollary~\ref{endomorphisms-of-copower}, the topological ring
of endomorphisms of the object $\coprod_{y\in Y_x}M_y$ is isomorphic
to the topological ring of row-zero-convergent matrices
$\Mat_{Y_x}(\R_x)$.
 It follows that the topological ring $\S_x$ is isomorphic to
the topological ring of row-finite matrices $\Mat_{Y_x}(D_x)$.
 By Lemma~\ref{hom-from-coproduct-topology}, the topology on
the ring $\S=\prod_{x\in X}\S_x$ is the product topology.
 It remains to recall Theorem~\ref{topologically-semisimple-ring}\,(4)
and conclude that the topological ring $\S$ is topologically
semisimple.

 We also observe that the topological ring $\S$ is complete, as it is
clear from the explicit description of its topology that we have
obtained.
 We still have to show that the ideal $\HH$ is strongly closed in~$\R$.
 
 For this purpose, choose for every $x\in X$ and all $y$, $z\in Y_x$
some (automatically continuous) section $s_{y,z}\:D_{y,z}\rarrow
\R_{y,z}$ of the natural surjective map $p_{y,z}\:\R_{y,z}\rarrow
\R_{y,z}/\HH_{y,z}$ onto the discrete group $D_{y,z}=\R_{y,z}/\HH_{y,z}$,
satisfying only the condition that $s_{y,z}(0)=0$.
 Define a section $s\:\S\rarrow\R$ of the continuous surjective ring
homomorphism $p\:\S\rarrow\R$ by the rule that every zero matrix
entry is lifted to a zero matrix entry, while the maps~$s_{y,z}$ are
applied to nonzero matrix entries.
 Then the map~$s$ (does not respect either the matrix addition or
the matrix multiplication, of course; but it) is continuous.
 Applying the map~$s$, one can lift any zero-converging family of
elements in $\S$ to a zero-converging family of elements in~$\R$.

 Finally, we have $\HH(\R)=\HH$ by Lemma~\ref{H-is-topological-Jacobson}.
 This proves the implication ``only~if''.
\end{proof}

 In order to prove the inverse implication, we will need another lemma.

\begin{lem} \label{orthogonal-idempotents-decompose}
 Let\/ $\sA$ be an idempotent-complete agreeable category and\/
$M\in\sA$ be an object.
 Let $(e_z\in\Hom_\sA(M,M))_{z\in Z}$ be a summable family of
orthogonal idempotents such that\/ $\sum_{e\in Z}e_z=\id_M$.
 Then there exists a unique direct sum decomposition
$M\cong\coprod_{z\in Z}M_z$ of the object $M$ such that the projector
$\iota_z\pi_z\:M\rarrow M_z\rarrow M$ on the direct summand $M_z$
in $M$ is equal to~$e_z$ for every $z\in Z$.
\end{lem}

\begin{proof}
 For the uniqueness, one just observes that a direct summand $L$ of
an object $M$ in additive category is determined by the projector
$M\rarrow L\rarrow M$ onto it.
 Conversely, since $\sA$ is assumed to be idempotent-complete,
every idempotent~$e_z$, \,$z\in Z$, determines a direct summand $M_z$
in $M$ such that $e_z$~is the composition of the projection
$\pi_z\:M\rarrow M_z$ and the injection $\iota_z\:M_z\rarrow M$.

 It remains to construct an isomorphism $M\cong\coprod_{z\in Z}M_z$.
 The collection of morphisms $\iota_z\:M_z\rarrow M$ corresponds
to a uniquely defined morphism $f\:\coprod_{z\in Z}M_z\rarrow M$.
 Constructing the desired morphism in the opposite direction
$g\:M\rarrow\coprod_{z\in Z}M_z$ is equivalent to showing that
the collection of morphisms $\pi_z\:M\rarrow M_z$ is summable.

 Now the family of idempotents $(e_z\:M\to M)_{z\in Z}$ is summable
by assumption, so we have a morphism $h\:M\rarrow M^{(Z)}$ whose
components are the idempotents~$e_z$.
 Let $\pi\:M^{(Z)}\rarrow\coprod_{z\in Z}M_z$ be the coproduct of
the morphisms~$\pi_z$, that is $\pi=\coprod_{z\in Z}\pi_z$.
 Then the desired morphism~$g$ can be obtained as $g=\pi\circ h$.
 This is based on the observation that $\pi_ze_z=\pi_z$ for
every $z\in Z$.

 The composition $gf\:\coprod_{z\in Z}M_z\rarrow M\rarrow
\coprod_{z\in Z}M_z$ is the identity morphism, since $\pi_z\iota_w
=\id_{M_z}$ when $z=w$ and~$0$ otherwise (the latter property being
a consequence of the assumption of orthogonality, $e_ze_w=0$).
 Finally, the assertion that the composition $fg\:M\rarrow
\coprod_{z\in Z}M_z\rarrow M$ is the identity endomorphism of $M$
is, by the definition, a restatement of the equation
$\sum_{z\in Z}e_z=\id_M$.
\end{proof}

\begin{proof}[Proof of the implication ``if'' in
Theorem~\ref{perfect-decomposition-endomorphism-ring}]
 By Theorem~\ref{topologically-semisimple-ring}\,(4), we have
an isomorphism of topological rings $\S\cong\prod_{x\in X}\S_x$,
where $\S_x=\Mat_{Y_x}(D_x)$ are the rings of row-finite matrices
over some discrete skew-fields~$D_x$.

 For every $y\in Y_x$, let $\bar e_y\in\S_x$ be the idempotent element
represented by the matrix whose only nonzero entry is
the element $1\in D_x$ sitting in the position
$(y,y)\in Y_x\times Y_x$.
 Clearly, $(\bar e_y)_{y\in Y_x}$ is a family of orthogonal idempotents
in the ring $\S_x$, converging to zero in its matrix topology
with $\sum_{y\in Y_x}\bar e_y=1$.

 Let $Z=\coprod_{x\in X}Y_x$ be the disjoint union of the nonempty
sets~$Y_x$.
 For every $y\in Y_x$, we will consider~$\bar e_y$ as an element of
the ring~$\S$ (which we can do, as $\S_x$ is naturally a subring
in~$\S$, though with a different unit).
 Then $(\bar e_z)_{z\in Z}$ is a family of orthogonal idempotents in $\S$,
converging to zero in the product topology of $\S$ with
$\sum_{z\in Z}\bar e_z=1$.

 Set $\HH=\HH(\R)$.
 By Corollary~\ref{lifting-orthogonal-cor}, there exists a set
of orthogonal idempotents $e_z\in\R$, \,$z\in Z$, such that
$\bar e_z=e_z+\HH$, the family of elements $(e_z)_{z\in Z}$ converges
to zero in the topology of $\R$, and $\sum_{z\in Z}e_z=1$ in~$\R$.
 By Lemma~\ref{orthogonal-idempotents-decompose} (together with
Lemma~\ref{agreeable-agreement-lemma}), there exists a direct sum
decomposition $M=\coprod_{z\in Z}M_z$ of the object $M$ such that
$e_z=\iota_z\pi_z$ is the composition of the coproduct injection
and the product projection $M\rarrow M_z\rarrow M$.

 The ring $\R$ can be now viewed as the ring of (row-summable) matrices
$r=(r_{w,z})_{w,z\in Z}$ with entries in the topological groups
$\R_{w,z}=\Hom_\sA(M_w,M_z)$.
 The topology of $\R_{w,z}$ can be recovered from that of $\R$
as the topology on a direct summand with a continuous idempotent
projector (see Example~\ref{idempotent-completion-top-agreeable}\,(2)).

 The continuous ring homomorphism $p\:\R\rarrow\S$ assigns to
a matrix $r=(r_{w,z})$ the block matrix $p(r)=
(p_{w,z}(r_{w,z}))_{w,z\in Z}$.
 Here, when there exists $x\in X$ such that both the indices $w$
and~$z$ belong to the subset $Y_x\subset Z$, we have a continuous
surjective abelian group homomorphism (with a discrete codomain)
$p_{w,z}\:\R_{w,z}\rarrow D_x$.
 Otherwise, $p_{w,z}=0$.
 The ideal $\HH\subset\R$ consists of all the matrices
$h=(h_{w,z})_{w,z\in Z}\in\R$ such that $h_{w,z}\in\HH_{w,z}$
for all $w$, $z\in Z$, where $\HH_{w,z}\subset\R_{w,z}$ is
the kernel of the map~$p_{w,z}$.

 In particular, for every $z\in Z$ the ring $\R_{z,z}=e_z\R e_z$
is topologically isomorphic to the topological endomorphism ring
$\Hom_\sA(M_z,M_z)$.
 The map $p_{z,z}\:\R_{z,z}\rarrow D_{z,z}=D_x$ is a continuous
surjective ring homomorphism with the kernel $\HH_{z,z}\subset\HH$.
 Since the ideal $\HH\subset\R$ is topologically left T\+nilpotent
by assumption, so is the open ideal $\HH_{z,z}\subset\R_{z,z}$.
 It follows that $\R_{z,z}$ is a local ring and $\HH_{z,z}$ is its
maximal ideal.

 Let $f\:M_w\rarrow M_z$ be a morphism that does not belong to
the subgroup $\HH_{w,z}\subset\R_{w,z}$.
 Our next aim is to show that $f$~is an isomorphism.
 Indeed, there exists $x\in X$ such that both $w$ and $z$ belong
to $Y_x$, for otherwise $\HH_{w,z}=\R_{w,z}$.
 Consider the element $d=p_{w,z}(f)\in D_x$.
 Since $d\ne0$, there exists an inverse element $d^{-1}\in D_x$.

 The map $p_{z,w}\:\R_{z,w}\rarrow D_x$ is surjective, so we can
choose a preimage $g\:M_z\rarrow M_w$ of the element~$d^{-1}$.
 Consider the compositions $fg\in\R_{z,z}$ and $gf\in\R_{w,w}$.
 We have $p_{z,z}(fg)=p_{w,z}(f)p_{z,w}(g)=1\in D_x$ and
$p_{w,w}(gf)=p_{z,w}(g)p_{w,z}(f)=1\in D_x$.
 Since the rings $R_{z,z}$ and $R_{w,w}$ are local with the residue
skew-fields $D_x$, it follows that the elements $fg$ and~$gf$
are invertible in these rings.
 Hence our morphism~$f$ is  an isomorphism.

 Finally, let $E\subset\R$ denote the subset of all elements
represented by matrices with a single nonzero entry which is not
an isomorphism.
 We have shown that $E\subset\HH$.
 Since the ideal $\HH$ is topologically left $T$\+nilpotent in $\R$ by
assumption, it follows that so is the set~$E$.
 The latter observation is equivalent to saying that the family of
objects $(M_z\in\sA)_{z\in Z}$ is locally T\+nilpotent.
\end{proof}

 Having proved Theorem~\ref{perfect-decomposition-endomorphism-ring},
we can now deduce the results promised in the beginning of this section.

\begin{cor} \label{weak-iii-implies-iv}
 Let\/ $\R$ be a complete, separated topological ring with a right
linear topology.
 Then the full subcategory of projective\/ $\R$\+contramodules
$\R\contra_\proj$ is closed under direct limits in $\R\contra$ if and
only if the topological ring\/ $\R$ is topologically left perfect.
\end{cor}

\begin{proof}
 By Corollary~\ref{topological-rings-are-endomorphism-rings}, there
exists an associative ring $A$ and a left $A$\+module $M$ such that
the topological ring $\R$ is isomorphic to the topological ring
of endomorphisms of the $A$\+module $M$ endowed with the finite
topology, that is $\R\cong\Hom_A(M,M)^\rop$.

 Assume that the full subcategory $\R\contra_\proj$ is closed under
direct limits in $\R\contra$.
 By Corollary~\ref{split-direct-limits-contramodule-interpretation},
it means that the full subcategory $\Add(M)\subset A\modl$ 
has split direct limits.
 According to Theorem~\ref{angeleri-saorin}, it follows that
the left $A$\+module $M$ has a perfect decomposition.
 Using Theorem~\ref{perfect-decomposition-endomorphism-ring}, we can
conclude that the topological ring $\R$ is topologically left perfect.

 Conversely, assume that $\R$ is topologically left perfect.
 By Theorem~\ref{perfect-decomposition-endomorphism-ring}, it means
that the left $A$\+module $M$ has a perfect decomposition.
 According to Theorem~\ref{angeleri-saorin}, it follows that
the additive category $\Add(M)$ has split direct limits.
 From Corollary~\ref{split-direct-limits-contramodule-interpretation}
we conclude that the class of all projective left $\R$\+contramodules
is closed under direct limits in $\R\contra$.
\end{proof}

\begin{proof}[Proof of Theorem~\ref{top-agreeable-angeleri-saorin}]
 Let\/ $\R=\Hom_\sA(M,M)^\rop$ be the topological ring of endomorphisms
of an object $M$ in a topologically agreeable category~$\sA$.

 Assume that the object $M\in\sA$ has a perfect decomposition.
 By Theorem~\ref{perfect-decomposition-endomorphism-ring}, it follows
that the topological ring $\R$ is topologically left perfect.
 According to Corollary~\ref{weak-iii-implies-iv}, this means that
the full subcategory $\R\contra_\proj$ is closed under direct limits in
$\R\contra$.
 From Corollary~\ref{split-direct-limits-contramodule-interpretation}
we can conclude that the full subcategory $\Add(M)\subset\sA$
has split direct limits.

 Conversely, assume that the category $\Add(M)$ has split direct limits.
 By Corollary~\ref{split-direct-limits-contramodule-interpretation},
it means that the class of all projective left $\R$\+contramodules is
closed under direct limits in $\R\contra$.
 According to Corollary~\ref{weak-iii-implies-iv}, it follows that
the topological ring~$\R$ is topologically left perfect.
 By Theorem~\ref{perfect-decomposition-endomorphism-ring}, we can
conclude that the object $M\in\sA$ has a perfect decomposition.
\end{proof}

\Section{Split Contramodule Categories are Semisimple}
\label{split-categories-secn}

 In this section we prove
Theorems~\ref{topologizable-spectral-is-semisimple}
and~\ref{contra-split-iff-semisimple}.
 We also give a negative answer (present a counterexample) to
a question posed in~\cite[Section~1.2]{PR}.

 The following lemma is a straightforward generalization
of~\cite[Proposition~V.6.1]{St}.

\begin{lem} \label{split-endomorphisms-vNregular}
 Let\/ $\sA$ be a split abelian category and $M\in\sA$ be an object.
 Then the endomorphism ring $R=\Hom_\sA(M,M)^\rop$ is von~Neumann
regular. \qed
\end{lem}

\begin{cor} \label{no-top-T-nilpotent-ideals-cor}
 Let\/ $\R$ be a complete, separated topological ring with a right
linear topology.
 Assume that the abelian category\/ $\R\contra$ is split.
 Then any closed topologically left T\+nilpotent two-sided ideal in\/
$\R$ is zero.
\end{cor}

\begin{proof}[First proof]
 We will even prove that any topologically nil two-sided ideal
$J\subset\R$ is zero.
 Indeed, by~\cite[Lemma~7.6(a)]{Pproperf}, $J$ is contained in
the Jacobson radical $H(\R)$ of the ring~$\R$.
 On the other hand, the ring $\R$ (viewed as an abstract ring with
the topology ignored) is the opposite ring to the ring of
endomorphisms $\Hom^\R(\R,\R)$ of the free left $\R$\+contramodule $\R$
with one generator.
 By Lemma~\ref{split-endomorphisms-vNregular}, the ring $\R$ is
von~Neumann regular; hence $H(\R)=0$ \,\cite[Corollary~1.2(c)]{Goo}.
\end{proof}

\begin{proof}[Second proof]
 Let $\J\subset\R$ be a closed topologically left T\+nilpotent
two-sided ideal.
 For any left $\R$\+contramodule $\D$, consider the subcontramodule
$\C=\J\tim\D\subset\D$ (see~\cite[Section~2.10]{Pproperf}).
 Since the category $\R\contra$ is split by assumption, $\C$ is
a direct summand in~$\D$.
 Hence $\J\tim\C=\C$.
 By the contramodule Nakayama lemma~\cite[Lemma~6.2]{Pproperf},
it follows that $\C=0$.
 In particular, taking $\D=\R$ to be the free left $\R$\+contramodule
with one generator, one obtains $0=\C=\J\tim\R=\J\subset\R$.
\end{proof}

\begin{proof}[Proof of Theorem~\ref{contra-split-iff-semisimple}]
 Let $\R$ be a complete, separated topological ring with a right linear
topology for which the abelian category $\R\contra$ is split.
 Then all left $\R$\+contramodules are projective, so the class of all
projective left $\R$\+contramodules is closed under direct limits in
$\R\contra$.
 By Corollary~\ref{weak-iii-implies-iv}, it follows that the topological
ring $\R$ is topologically left perfect.

 So the topological Jacobson radical $\HH(\R)$ is a closed topologically
left T\+nilpotent two-sided ideal in~$\R$.
 According to Corollary~\ref{no-top-T-nilpotent-ideals-cor}, we can
conclude that $\HH(\R)=0$.
 Thus $\R=\S$ is a topologically semisimple topological ring, and
the abelian category $\R\contra$ is Ab5 and semisimple by
Theorem~\ref{topologically-semisimple-ring}\,(1).
\end{proof}

\begin{proof}[Proof of
Theorem~\ref{topologizable-spectral-is-semisimple}]
 Let $\sA$ be a topologically agreeable split abelian category.
 Given an object $G\in\sA$, consider the full subcategory
$\Add(G)\subset\sA$.
 Since $\sA$ is split abelian, the category $\Add(G)$ is split
abelian, too.
 Furthermore, the full subcategory $\Add(G)\subset\sA$ is closed
under kernels and colimits.
 Since $\sA$ is topologically agreeable,
Theorem~\ref{generalized-tilting}(a,c) tells that
$\Add(G)\cong\R\contra_\proj$, where $\R=\Hom_\sA(G,G)^\rop$
is the topological ring of endomorphisms of the object~$G$.
 So the additive category of projective left $\R$\+contramodules
is split abelian.
 It follows that all left $\R$\+contramodules are projective,
that is $\R\contra=\R\contra_\proj$ (cf.\ the discussion in the proof of
Theorem~\ref{topologically-semisimple-ring}\,(3)\,$\Rightarrow$\,(1)).
 Hence the abelian category $\R\contra$ is split.
 By Theorem~\ref{contra-split-iff-semisimple}, the category
$\R\contra$ is Grothendieck and semisimple.
 Thus so is the category $\Add(G)\cong\R\contra_\proj=\R\contra$.
 As this holds for every object $G\in\sA$, it follows easily that
the category $\sA$ is Ab5 and semisimple.
\end{proof}

\begin{rem}
 Conversely, it is easy to deduce
Theorem~\ref{contra-split-iff-semisimple} back from
Theorem~\ref{topologizable-spectral-is-semisimple}.
 Indeed, for any complete, separated topological ring $\R$ with
right linear topology, the additive category $\R\contra_\proj$ is
topologizable by
Remark~\ref{projective-contramodules-topologically-agreeable}
(or by Corollary~\ref{topological-rings-are-endomorphism-rings}
and Theorem~\ref{generalized-tilting}(a,c)).
 If $\R\contra=\R\contra_\proj$ is split abelian and all topologizable
split abelian categories are Ab5 and semisimple, then
$\R\contra$ is Ab5 and semisimple, so a Grothendieck category.
\end{rem}

\begin{ex} \label{ring-summation-structure-counterex}
 It was observed in~\cite[Section~1.2]{PR} that the category
$\R\contra_\proj$ is always agreeable, and the question was asked
whether the converse holds, in the following sense.
 Let $\sB$ be a locally presentable abelian category with a projective
generator such that the full subcategory of projective objects
$\sB_\proj\subset\sB$ is agreeable (cf.\
Lemma~\ref{agreeable-projectives-locally-presentable}).
 Does there exist a complete, separated topological ring $\R$ with
a right linear topology such that $\sB$ is equivalent to the category
of left $\R$\+contramodules?

 Now we can show that the answer is negative.
 Let $\sB$ be a spectral category.
 By the definition, $\sB$ is a Grothendieck abelian category;
in particular, it is locally presentable and has a generator~$G$.
 Furthermore, all objects of $\sB$ are projective, so $G\in\sB_\proj
=\sB$.
 Assume that there exists a topological ring $\R$ such that
$\sB=\R\contra$.
 By Theorem~\ref{contra-split-iff-semisimple}, it would then follow
that $\R\contra$ is a semisimple Grothendieck category, i.~e., it is
discrete spectral.
 Thus any nonzero continuous spectral category $\sB$ (such as, e.~g.,
the one described in Example~\ref{complete-boolean}) is
a counterexample.

 The above argument is quite involved and indirect, as it is based on
the results of the theory of direct sum decompositions of modules
(the ``if'' assertion of Theorem~\ref{angeleri-saorin}).
 A somewhat simpler and more direct alternative argument, producing
a more narrow class of counterexamples, is discussed below in
Remark~\ref{ring-summation-easy-counterex}.

 On the other hand, if $\sB$ is a cocomplete abelian category with
a projective generator $P$ such that the category $\sB_\proj$ is
\emph{topologically} agreeable, and $\R=\Hom_{\sB_\proj}(P,P)^\rop$ is
the topological ring of endomorphisms of the object $P$, then
the category $\sB$ is equivalent to $\R\contra$.
 This equivalence can be constructed by applying the result of
Theorem~\ref{generalized-tilting}(c) to the additive category
$\sA=\sB_\proj$ and the object $M=P$.
 Indeed, one has $\Add(P)=\sB_\proj\subset\sB$, so it remains to
take into account the uniqueness assertion in
Theorem~\ref{generalized-tilting}(a).
\end{ex}

\Section{Countable Topologies and Countably Generated Modules}
\label{countably-generated-secn}

 Let $R$ be an associative ring.
 A right $R$\+module $N$ is said to be \emph{coperfect} if every
descending chain of cyclic $R$\+submodules in $N$ terminates,
 or equivalently if  every descending chain of finitely
generated $R$\+submodules in $N$ terminates~\cite[Theorem~2]{Bj}.
 Clearly, any submodule and any quotient module of a coperfect
module is coperfect.
 The class of coperfect right $R$\+modules is also closed under
direct limits.

 A right $R$\+module $N$ is said to be \emph{$\Sigma$\+coperfect}
if the right $R$\+module $N^{(\omega)}=\bigoplus_{n=0}^\infty N$
is coperfect, or equivalently, the right $R$\+module $N^n$ is
coperfect for every $n\ge1$.
 Clearly, if an $R$\+module $N$ is $\Sigma$\+coperfect, then
the $R$\+module $N^{(X)}$ is $\Sigma$\+coperfect for any set~$X$.

 Let $A$ be an associative ring, $M$ be a left $A$\+module and
$R=\Hom_A(M,M)^\rop$ be its ring of endomorphisms.
 An $A$\+module $M$ is said to be \emph{endocoperfect} if $M$ is
a coperfect right $R$\+module.
 We will say that an $A$\+module $M$ is \emph{endo-$\Sigma$-coperfect}
if $M$ is a $\Sigma$\+coperfect module over its endomorphism ring,
that is a $\Sigma$\+coperfect $R$\+module.

 The following theorem was proved in the paper~\cite{AS}.

\begin{thm}[{\cite[Corollary~2.3]{AS}}] \label{angeleri-saorin-endo}
 Let $A$ be an associative ring and $M$ be a left $A$\+module.
 Assume that, for every sequence of left $A$\+module endomorphisms
$M\rarrow M\rarrow M\rarrow\dotsb$, the induced morphism of left
$A$\+modules\/ $\bigoplus_{n=0}^\infty M\rarrow\varinjlim_{n\ge0}M$
is split.
 Then the left $A$\+module $M$ is endo-$\Sigma$-coperfect.
 In particular, any left $A$\+module $M$ with a perfect decomposition
is endo-$\Sigma$-coperfect. \qed
\end{thm}

 The aim of this section is to prove the following theorem providing
a partial converse assertion to Theorem~\ref{angeleri-saorin-endo}.

\begin{thm} \label{countably-generated-endocoperfect-theorem}
 Let $A$ be an associative ring and $M$ be a countably generated
endo-$\Sigma$-coperfect left $A$\+module.
 Then the left $A$\+module $M$ has a perfect decomposition.
\end{thm}

 We start with the following simple lemma interpreting
endo-$\Sigma$-coperfectness as a property of the topological ring
of endomorphisms.

\begin{lem}[{cf.~\cite[Proposition~2.2%
\,(1)\,$\Leftrightarrow$\,(2)]{AS}}]
\label{endocoperfect-discrete-coperfect}
 Let $A$ be an associative ring, $M$ be a left $A$\+module, and\/
$\R=\Hom_A(M,M)^\rop$ be the topological ring of endomorphisms of $M$
(with the finite topology).
 Then the right\/ $\R$\+module $M$ is\/ $\Sigma$\+coperfect if and
only if all discrete right\/ $\R$\+modules are coperfect.
\end{lem}

\begin{proof}
 ``If'': It is clear from the definition of the finite topology in
Example~\ref{modules-topologically-agreeable-example}\,(1) that $M$ is
a discrete right $\R$\+module.
 Hence so is $M^{(\omega)}$.

 ``Only if'': assuming that $M^{(\omega)}$ is a coperfect right
$\R$\+module, we have to show that the cyclic right $\R$\+module
$\R/\I$ is coperfect for every open right ideal $\I\subset\R$.
 Indeed, by the definition of the finite topology, there exists
a finite set of elements $m_1$,~\dots, $m_n\in M$ such that $\I$
contains the intersection of the annihilators of the elements~$m_j$
in~$\R$.
 Let $m=(m_1,\dotsc,m_n,0,0,\dotsc)\in M^n\subset M^{(\omega)}$ be
the related element.
 Then the cyclic right $\R$\+module $\R/\I$ is a quotient module of
the cyclic right $\R$\+module $m\R\subset M^{(\omega)}$.
 It remains to recall that the class of all coperfect right
$\R$\+modules is closed under the passages to submodules and quotients.
\end{proof}

 Let $\R$ be a complete, separated topological ring with a right linear
topology.
 We say that $\R$ is \emph{topologically right coperfect} if all
discrete right $\R$\+modules are coperfect.
 We will see below in Section~\ref{topologically-perfect-secn} that
all topologically left perfect rings (in the sense of
Section~\ref{perfect-decompositions-secn}) are topologically
right coperfect.

 The proof of
Theorem~\ref{countably-generated-endocoperfect-theorem} is based on
the following theorem, which is the key technical result of
this section. 

\begin{thm} \label{countable-discrete-coperfect-implies-top-perfect}
 Let\/ $\R$ be a complete, separated topological ring with
a \emph{countable} base of neighborhoods of zero consisting of open
right ideals.
 Assume that the topological ring\/ $\R$ is topologically right
coperfect.
 Then\/ $\R$ is topologically left perfect.
\end{thm}

 The proof of
Theorem~\ref{countable-discrete-coperfect-implies-top-perfect} follows
below in the form of a sequence of lemmas.
 Given an associative ring $R$ and a right $R$\+module $M$,
the \emph{radical of $M$}, denoted by $\rad(M)\subset M$, is defined
as the intersection of all the maximal (proper) $R$\+submodules in~$M$.

\begin{lem} \label{rad-quot-rad}
 For any $R$\+module $M$, one has\/ $\rad(M/\rad(M))=0$.
\end{lem}

\begin{proof}
 All the maximal submodules of $M$ contain $\rad(M)$, so maximal
submodules of $M$ correspond bijectively to maximal submodules
of $M/\rad(M)$.
\end{proof} 

\begin{lem} \label{simple-submodules-are-summands}
 Let $M$ be an $R$\+module with\/ $\rad(M)=0$.
 Then any simple submodule of $M$ is a direct summand in~$M$.
\end{lem}

\begin{proof}
 Suppose $S\subset M$ is a simple submodule.
 Since $\rad(M)=0$, there exists a maximal $R$\+submodule
$N\subset M$ with $S\not\subset N$.
 Then we have $S\cap N=0$ since $S$ is simple, and $S+N=M$
since $N$ is maximal; hence $M=S\oplus N$.
\end{proof}

\begin{lem} \label{nonzero-socle}
 Any nonzero coperfect $R$\+module has a simple submodule.
\end{lem} 

\begin{proof}
 Any nonzero $R$\+module is either simple, or it contains
a nonzero proper cyclic submodule.
\end{proof}

\begin{lem} \label{semisimple-lemma}
 Let $M$ be a coperfect $R$\+module with\/ $\rad(M)=0$.
 Then $M$ is a direct sum of simple $R$\+modules.
\end{lem}

\begin{proof}
 It suffices to show that all cyclic $R$\+submodules of $M$ are
semisimple (since a sum of semisimple modules is always semisimple).
 Let $N\subset M$ be a cyclic submodule.
 By Lemma~\ref{nonzero-socle}, if $N\ne 0$, then $N$ has a simple
$R$\+submodule~$S_1$.
 By Lemma~\ref{simple-submodules-are-summands}, any simple submodule
of $N$ is a direct summand in $M$, hence also in~$N$.
 So we have $N=S_1\oplus N_1$ for some $R$\+submodule $N_1\subset N$.

 The $\R$\+module $N_1$ is also cyclic, as a direct summand of a cyclic
$\R$\+module~$N$.
 If $N_1\ne0$, then $N_1$ has a simple $R$\+submodule~$S_2$, and
the same argument shows that $N_1=S_2\oplus N_2$.
 Proceeding in this way, we get a descending chain of cyclic submodules
$N\varsupsetneq N_1\varsupsetneq N_2 \varsupsetneq\dotsb$ in
the $R$\+module~$M$.
 By the assumption of coperfectness of $M$, this chain must
terminate; so there exists $k\ge1$ such that $N_k=0$.
 Hence $N=\bigoplus_{i=1}^k S_k$, and we are done.
\end{proof}
 
 The following assertion (as well as some other properties of coperfect
modules) can be found in the book~\cite[Section~31.8]{Wis}.
 
\begin{lem} \label{rad-quot-semisimple}
 Let $M$ be a coperfect $R$\+module.
 Then the $R$\+module $M/\rad(M)$ is semisimple.
\end{lem}

\begin{proof}
 Follows from Lemmas~\ref{rad-quot-rad} and~\ref{semisimple-lemma}.
\end{proof} 

 Given an $R$\+module $M$, we denote by $\tp(M)$ the quotient
$R$\+module $M/\rad(M)$.

\begin{lem} \label{rad-top-exactness}
 Assume that a topological ring\/ $\R$ is topologically right coperfect,
and let\/ $0\rarrow K\rarrow L\rarrow M\rarrow0$ be a short exact
sequence of discrete right\/ $\R$\+modules.  Then \par
\textup{(a)} the short sequence\/ $\tp(K)\rarrow\tp(L)\rarrow\tp(M)
\rarrow 0$ is right exact; \par
\textup{(b)} the map\/ $\rad(L)\rarrow\rad(M)$ is surjective.
\end{lem}

\begin{proof}
 First of all, one needs to notice that both $\rad$ and $\tp$
are functors $\discr\R\rarrow\discr\R$, and there is a short exact
sequence of functors $0\rarrow\rad\rarrow\Id_{\discr\R}\rarrow
\tp\rarrow0$, where $\Id_{\discr\R}$ denotes the identity functor.
 These observations do not depend on the assumption of topological
coperfectness of $\R$ yet.

 Part~(a): denote by $(\discr\R)^\sss$ the full subcategory of
semisimple discrete right $\R$\+modules in $\discr\R$.
 Clearly, $(\discr\R)^\sss$ is an abelian category and the inclusion
functor $(\discr\R)^\sss\rarrow\discr\R$ is exact.

 By Lemma~\ref{rad-quot-semisimple}, one has $\tp(M)\in
(\discr\R)^\sss$ for any $M\in\discr\R$.
 On the other hand, for any module $N\in(\discr\R)^\sss$, any
$\R$\+module morphism $f\:M\rarrow N$ annihilates the submodule
$\rad(M)\subset M$, since $\rad(N)=0$.
 So $f$~factorizes (uniquely) as $M\rarrow\tp(M)\rarrow N$.
 In other words, this means that $\tp\:\discr\R\rarrow(\discr\R)^\sss$
is a left adjoint functor to the inclusion $(\discr\R)^\sss\rarrow
\discr\R$.
 It follows that $\tp$ is right exact as a functor $\discr\R\rarrow
(\discr\R)^\sss$, and consequently also as a functor $\discr\R\rarrow
\discr\R$.

 In part~(b), the short exact sequence of functors $0\rarrow\rad
\rarrow\Id_{\discr\R}\rarrow\tp\rarrow0$ applied to the short sequence
$0\rarrow K\rarrow L\rarrow M\rarrow0$ produces a short exact
sequence of $3$\+term complexes.
 A simple diagram chase (or an application of the snake lemma) shows
that (a) implies~(b).
\end{proof}

\begin{lem} \label{jacobson-radical-lemma}
 Assume that a topological ring\/ $\R$ is topologically right coperfect
and has a countable base of neighborhoods of zero.
 Then, for any open right ideal\/ $\I\subset\R$, one has\/
$\tp(\R/\I)=\R/(\I+\HH)$, where\/ $\HH\subset\R$ is the topological
Jacobson radical.
\end{lem}

\begin{proof}
 Let us show that $\rad(\R/\I)=(\I+\HH)/\I$.
 Consider the projective system of discrete right $\R$\+modules $\R/\I$,
where $\I$ ranges over the open right ideals in $\R$, and the projective
subsystem formed by the submodules $\rad(\R/\I)\subset\R/\I$.
 By Lemma~\ref{rad-top-exactness}(b), for any open right ideals
$\J\subset\I\subset\R$, the map $\rad(\R/\J)\rarrow\rad(\R/\I)$ is
surjective.
 Since the poset of all open right ideals $\I\subset\R$ has a countable
cofinal subposet, it follows that the projection map
$$
 \varprojlim\nolimits_{\J\subset\R}\rad(\R/\J)
 \lrarrow\rad(\R/\I)
$$
is surjective.

 The projective limit $\varprojlim_{\J\subset\R}\rad(\R/\J)$ is
a subgroup in $\varprojlim_{\J\subset\R}\R/\J=\R$.
 Let us compute this subgroup.
 We have
$$
 \rad(\R/\J)=\frac{\bigcap_{\MM\supset\J}\MM}{\J},
$$
where the intersection is taken over all the (necessarily open)
maximal right ideals $\MM\subset\R$ containing~$\J$.
 Consequently,
$$
 \varprojlim\nolimits_{\J\subset\R}\rad(\R/\J)\,=\,
 \bigcap\nolimits_{\J\subset\R}\bigcap\nolimits_{\MM\supset\J}\MM\,=\,
 \bigcap\nolimits_\MM\MM\,=\,\HH.
$$

 We have shown that the natural surjective map $\R\rarrow\R/\I$
restricts to a surjective map $\HH\rarrow\rad(\R/\I)$.
 Therefore, $\rad(\R/\I)=(\I+\HH)/\I$, as desired.
\end{proof}

 Now at last we can prove the promised theorems.

\begin{proof}[Proof of
Theorem~\ref{countable-discrete-coperfect-implies-top-perfect}]
 By~\cite[Lemma~2.3]{Pproperf}, the ideal $\HH$ is strongly closed
in~$\R$.
 In particular, the quotient ring $\S=\R/\HH$ is complete in its
quotient topology.

 By~\cite[Corollary~7.7]{Pproperf}, the topological Jacobson radical
$\HH$ of the ring $\R$ is topologically left T\+nilpotent.
 It remains to prove that the topological ring $\S$ is topologically
semisimple.

 For any topological ring $\R$ with a closed two-sided ideal
$\HH\subset\R$, the category of discrete modules over the topological
quotient ring $\R/\HH$ is equivalent to the full subcategory in
$\discr\R$ consisting of all the discrete right $\R$\+modules
annihilated by~$\HH$.
 In the situation at hand, in view of
Theorem~\ref{topologically-semisimple-ring}\,(2), it remains to show
that every discrete right $\R$\+module $N$ annihilated by $\HH$
is semisimple. 

 Since a sum of semisimple modules is semisimple, it suffices
to consider the case when $N$ is a cyclic discrete right $\R$\+module
annihilated by~$\HH$.
 So we have $N\cong\R/\I$, where $\I\subset\R$ is an open right ideal
containing~$\HH$.
 Now by Lemma~\ref{jacobson-radical-lemma} we have $\rad(\R/\I)=0$,
and it remains to invoke Lemma~\ref{semisimple-lemma}.
\end{proof}

\begin{proof}[Proof of
Theorem~\ref{countably-generated-endocoperfect-theorem}]
 Given a set of generators $\{m_y\in M\mid y\in Y\}$ of a left
$A$\+module $M$, the annihilators of finite subsets of $\{m_y\}$ form
a base of neighborhoods of zero in the topological ring
$\R=\Hom_A(M,M)^\rop$.
 Hence $\R$ has a countable base of neighborhoods of zero whenever
$M$ is countably generated.

 Now let $M$ be a countably generated endo-$\Sigma$-coperfect left
$A$\+module.
 Then all discrete right $\R$\+modules are coperfect
by Lemma~\ref{endocoperfect-discrete-coperfect}.
 According to
Theorem~\ref{countable-discrete-coperfect-implies-top-perfect}, it
follows that the topological ring $\R$ is topologically left perfect.
 Applying Theorem~\ref{perfect-decomposition-endomorphism-ring} to
the object $M\in\sA=A\modl$, we conclude that the left $A$\+module $M$
has a perfect decomposition.
\end{proof}

\Section{Topological Coherence and Coperfectness}
\label{topologically-coherent-secn}

 In this section we discuss some results of the paper~\cite{Ro},
which imply that a topologically right coperfect topological ring
$\R$ is topologically left perfect under the additional assumption of
\emph{topological right coherence} of~$\R$.

 We refer to the book~\cite[Sections~1.A and~2.A]{AR} for
the definitions of \emph{finitely accessible} and \emph{locally finitely
presentable} categories.
 In this section, we are interested in additive categories.
 Let us point out that the terminology in the literature is not
consistent: what we, following~\cite{AR}, call finitely accessible
categories, are called ``locally finitely presented categories''
in the papers~\cite{CB,Kra}.

 Let $\sA$ be a finitely accessible additive category and $\sA_\fp
\subset\sA$ be the full subcategory of finitely presentable objects.
 Then $\sA_\fp$ is an essentially small idempotent-complete additive
category.
 Conversely, any small idempotent-complete additive category $\sD$
can be realized as the category $\sA_\fp$ for a finitely accessible
additive category~$\sA$.
 The category $\sA$ can be recovered as the category of ind-objects
in~$\sD$ \,\cite[Theorem~2.26]{AR}.
 Alternatively, $\sA$ is the category of flat functors
$\sD^\sop\rarrow\Ab$ (where, given a small preadditive category $\sD$,
a contravariant additive functor $F\:\sD^\sop\rarrow\Ab$
is called \emph{flat} if the tensor product $G\longmapsto F\ot_\sD G$
is an exact functor on the abelian category of covariant additive
functors $G\:\sD\rarrow\Ab$) \,\cite[Theorem~1.4]{CB},
\cite[Proposition~5.1]{Kra}.

 A finitely accessible additive category $\sA$ is locally finitely
presentable if and only if $\sA$ has cokernels.
 In this case, the full subcategory $\sA_\fp$ is closed under
cokernels in~$\sA$.
 Conversely, if the category $\sA_\fp$ has cokernels, then so
does~$\sA$ \,\cite[Section~2.2]{CB}, \cite[Lemma~5.7]{Kra}.
 In this case, an additive functor $(\sA_\fp)^\sop\rarrow\Ab$ is flat
if and only if it is left exact, i.~e., takes cokernels in $\sA_\fp$ to
kernels in~$\Ab$.
 So the category $\sA$ can be recovered as the category of left exact
functors $(\sA_\fp)^\sop\rarrow\Ab$.

 The definition of a \emph{locally finitely generated} category
can be found in~\cite[Section~1.E]{AR}.
 Any locally finitely generated abelian category (hence, in particular,
any locally finitely presentable abelian category) is Grothendieck.
 Moreover, any locally finitely presentable abelian category is locally
finitely generated and an object of it is
finitely generated if and only if it is a quotient object of a finitely
presentable object~\cite[Proposition~1.69]{AR}.

 Let $\sC$ be a locally finitely generated abelian category.
 An object $C\in\sC$ is said to be \emph{coherent} if $C$ is finitely
generated and, for every finitely generated object $D\in\sC$,
the kernel of any morphism $D\rarrow C$ in $\sC$ is also finitely
generated.
 In any locally finitely generated abelian category, the class of all
coherent objects is closed under kernels, cokernels, and extensions;
so coherent objects form an abelian subcategory.
 A locally finitely generated abelian category $\sC$ is called
\emph{locally coherent} if has a set of coherent
generators~\cite[Section~2]{Ro}.

 An abelian category $\sC$ is locally coherent if and only if it is
locally finitely presentable and the full subcategory $\sC_\fp$ is
closed under kernels in $\sC$.
 Equivalently, a locally finitely presentable abelian category $\sC$
locally coherent if and only if any finitely generated subobject of
a finitely presentable object in $\sC$ is finitely presentable, and
if and only if the category $\sC_\fp$, viewed as an abstract category,
is abelian (in this case, the inclusion functor $\sC_\fp\rarrow\sC$
is exact).
 In a locally coherent abelian category, an object is coherent if and
only if it is finitely presentable.
 Conversely, for any small abelian category $\sD$ there exists a unique
locally coherent abelian category $\sC$ such that
the category $\sD$ is equivalent to~$\sC_\fp$.
 The category $\sC$ can be recovered as the category of ind-objects
in~$\sD$,  or which is the same, the category of flat (equivalently,
left exact) functors $\sD^\sop\rarrow\Ab$.

 A locally finitely generated abelian category $\sC$ is said to be
\emph{locally coperfect}~\cite[Section~3]{Ro} if it has a set of
(finitely generated) generators $(C_\alpha)$ with the property that
any descending chain of finitely generated subobjects in any one of
the objects $C_\alpha$ terminates.
 Equivalently, this means that any descending chain of finitely
generated subobjects in any object of $\sC$ terminates.
 A locally coherent abelian category $\sC$ is locally coperfect if and
only if the abelian category $\sC_\fp$ is Artinian, that is, any
descending chain of (sub)objects in the category $\sC_\fp$ terminates.

 Let $\R$ be a complete, separated topological ring with right linear
topology.
 Then the abelian category of discrete right $\R$\+modules $\discr\R$
is locally finitely generated.
 Moreover, an object $N\in\discr\R$ is finitely generated in the sense
of the definition in~\cite[Section~1.E]{AR} if and only if it is
a finitely generated right $\R$\+module (so our terminology
is consistent).

 The category $\discr\R$ is locally coperfect if and only if all
discrete right $\R$\+modules are coperfect (in the sense of
the discussion in Section~\ref{countably-generated-secn}).
 We recall that in this case the topological ring $\R$ is said to be
\emph{topologically right coperfect}.

 Following~\cite[Definition~4.3]{Ro}, we say that the topological ring
$\R$ is \emph{topologically right coherent} if the category $\discr\R$
is locally coherent.

\begin{lem}[{\cite[Remark~2 in Section~4]{Ro}}]
 The topological ring\/ $\R$ is topologically right coherent if and
only if there exists a set of open right ideals $B$ forming a base
of neighborhoods of zero in\/ $\R$ such that, for every pair of
open right ideals\/ $\I$ and\/ $\J\in B$ and any integer $n\ge1$,
the kernel of any right\/ $\R$\+module morphism $(\R/\I)^n\rarrow\R/\J$
is a finitely generated right\/ $\R$\+module.
\end{lem}

\begin{proof}
 Assume that the category $\discr\R$ is locally coherent.
 Then let $B$ denote the set of all open right ideals $\J\subset\R$
such that the right $\R$\+module $\R/\J$ is a coherent object
of $\discr\R$.
 Let us show that $B$ is a base of neighborhoods of zero in~$\R$.
 Indeed, let $\I\subset\R$ be an open right ideal, and let $M$
be a coherent object of $\discr\R$ admitting an epimorphism
$M\rarrow\R/\I$.
 Let $m\in M$ be any preimage of the element $1+\I\in\R/\I$, and
let $\J$ be the annihilator of~$m$ in~$\R$.
 Then $\R/\J$ is a finitely generated subobject of a coherent object
$M$, hence the object $\R/\J\in\discr\R$ is coherent and $\J\in B$.
 By construction, we have $\J\subset\I$.
 
 To show that $B$ is closed under finite intersections, one observes
that, for any open right ideals $\J'$, $\J''\in B$, the object
$\R/(\J'\cap\J'')\in\discr\R$ is a finitely generated subobject of
the coherent object $\R/\J'\oplus\R/\J''$.
 Finally, for any open right ideal $\I\subset\R$ and any $\J\in B$,
the kernel of a morphism from the finitely generated object
$(\R/\I)^n$ to the coherent object $\R/\J$ is finitely generated.

 Conversely, let $B$ be a base of neighborhoods of zero in $\R$
satisfying the condition of the lemma.
 Then the object $\R/\J\in\discr\R$ is coherent for all $\J\in B$.
 Indeed, let $M$ be a finitely generated discrete right $\R$\+module
and $M\rarrow\R/\J$ be a morphism.
 Then there exists an open right ideal $\I\in B$ and an integer
$n\ge1$ such that there is a surjective $\R$\+module morphism
$(\R/\I)^n\rarrow M$.
 The kernel of the composition $(\R/\I)^n\rarrow M\rarrow\R/\J$ is
finitely generated by assumption, hence the kernel of the morphism
$M\rarrow\R/\J$ is finitely generated, too.
 Now the discrete right $\R$\+modules $\R/\J$, \ $\J\in B$ form
a set of coherent generators of the locally finitely generated
abelian category $\discr\R$.
\end{proof}

\begin{prop}[{\cite[Theorem~6]{Ro}}] \label{can-be-realized}
 Any locally coperfect locally coherent abelian category\/ $\sC$ can be
realized as the category of discrete right modules over a (topologically
right coperfect and right coherent) topological ring,
$\sC\cong\discr\R$.
\end{prop}

\begin{proof}[Proof]
 For any locally coherent abelian category $\sC$, one considers
the category $\sA$ of all left exact, direct limit-preserving
covariant functors $\sC\rarrow\Ab$.
 Then $\sA$ is also a locally coherent abelian category.
 In the terminology of~\cite{Ro}, \,$\sA$ is called
the \emph{conjugate} locally coherent abelian category to~$\sC$.
 Then $\sC$ is also conjugate to $\sA$.
 In fact, $\sC$ is the category of left exact functors
$(\sC_\fp)^\sop\rarrow\Ab$ and $\sA$ is the category of left exact
functors $\sC_\fp\rarrow\Ab$; so the small abelian categories $\sA_\fp$
and $(\sC_\fp)^\sop$ are naturally equivalent.

 Hence the category $\sC_\fp$ is Artinian if and only if the category
$\sA_\fp$ is Noetherian.
 It follows that the category $\sC$ is locally coperfect if and only if
the category $\sA$ is locally Noetherian.
 Assuming that this is the case, the paper~\cite{Ro} suggests to choose
a \emph{big} injective object $J\in\sA$, which means
an object such that the full subcategory of injective objects
$\sA_\inj\subset\sA$ coincides with $\Add(J)$.
 Equivalently, an injective object $J\in\sA$ is big if and only if
it contains a representative of every isomorphism class of indecomposable
injectives in~$\sA$.
 Any big injective object is an injective cogenerator, but the converse
is not true, in general.

 Now let $F\:\sC\rarrow\Ab$ be the functor corresponding to the chosen
big injective object $J\in\sA$.
 Then there is a complete, separated topological ring $\R$ with
a right linear topology and a category equivalence $\sC\cong\discr\R$
transforming the functor $F\:\sC\rarrow\Ab$ into the forgetful functor
$\discr\R\rarrow\Ab$.
 This is one of the results of~\cite[Theorem~6]{Ro}.

 Essentially, $\R=\Hom_\sA(J,J)^\rop$ is the opposite ring to the ring
of endomorphisms of the object~$J$, endowed with the finite topology,
as in Example~\ref{modules-topologically-agreeable-example}\,(2).
 This means that annihilators of finitely generated (\,$=$~Noetherian)
subobjects $E\subset J$ form a base of neighborhoods of zero in~$\R$.
 Notice that the functor $F|_{\sC_\fp}\:\sC_\fp\rarrow\Ab$ can be
computed as the functor $\Hom_\sA({-},J)|_{\sA_\fp}\:(\sA_\fp)^\sop
\rarrow\Ab$.
 Now, for any object $N\in\sA$, the abelian group $\Hom_\sA(N,J)$
has a natural structure of right $\R$\+module, and this right
$\R$\+module is discrete for $N\in\sA_\fp$.
 This defines an exact functor $\sC_\fp=(\sA_\fp)^\sop\rarrow\discr\R$,
which can be uniquely extended to a direct limit-preserving exact
functor $G\:\sC\rarrow\discr\R$.

 It is explained in~\cite[first part of proof of Theorem~4 in
Section~4]{Ro} that the functor $G$ takes the objects of $\sC_\fp$
to coherent objects of $\discr\R$.
 In particular, since $G(E^\sop)=\Hom_\sA(E,J)\cong\R/\I$ for any
Noetherian subobject $E\subset J$ and its annihilator ideal
$\I\subset\R$, the discrete right $\R$\+module $\R/\I$ is coherent.
 So the category $\discr\R$ has a set of coherent generators and
the topological ring $\R$ is topologically right coherent.
 Furthermore, by~\cite[Lemma~4.1]{Ro}, the functor $G|_{\sC_\fp}\:
\sC_\fp\rarrow\discr\R$ is fully faithful.

 Since coherent objects in a locally coherent category
(here the category $\discr\R$)
are finitely presentable, it follows that the functor $G\:\sC\rarrow
\discr\R$ is fully faithful, too.
 Hence the essential image of $G$ is a full subcategory closed under
kernels, cokernels, and coproducts in $\discr\R$.
 Since the essential image of $G$ contains a set of generators of
$\discr\R$, as we have seen, we can conclude that $G$ is an equivalence
of categories.
\end{proof}

\begin{thm} \label{top-coherent-top-coperfect}
 Let\/ $\R$ be a topologically right coherent, topologically right
coperfect, complete and separated topological ring with a right
linear topology.
 Then the topological ring $\R$ is topologically left perfect.
\end{thm}

\begin{proof}
 The category of discrete right $\R$\+modules $\sC=\discr\R$ is
a locally coperfect, locally coherent abelian category.
 Left-exact, direct limit preserving functors $\sC\rarrow\Ab$ form
a locally Noetherian abelian category $\sA$ conjugate to $\sC$
(see the discussion in the previous proof).
 In particular, the forgetful functor $F\:\discr\R\rarrow\Ab$
corresponds to an object $J\in\sA$.
 The functor $F|_{\sC_\fp}$ can be computed as the functor
$\Hom_\sA({-},J)|_{\sA_\fp}$.
 Since the functor $F|_{\sC_\fp}$ is exact, it follows that the functor
$\Hom_\sA({-},J)$ is exact in restriction to $\sA_\fp$; thus
the object $J\in\sA$ is injective.
 Following~\cite[Corollary~5.1]{Ro}, \,$J$ is actually
a big injective object in $\sA$ (but we do not need to use this fact).

 The topological ring $\R$ can be uniquely recovered as the opposite
ring to the ring of endomorphisms of the functor $F$, endowed with
the finite topology, as explained in
Section~\ref{as-endomorphism-rings-secn} above.
 This is the same thing as the ring of endomorphisms of the object
$J\in\sA$, endowed with the finite topology, as in the proof of
Proposition~\ref{can-be-realized}.

 According to Example~\ref{locally-Noetherian-example}, the category
$\Add(J)=\sA_\inj\subset\sA$ has split direct limits.
 By Corollary~\ref{split-direct-limits-contramodule-interpretation},
it follows that the full subcategory $\R\contra_\proj\subset\R\contra$
is closed under direct limits.
 It remains to apply Corollary~\ref{weak-iii-implies-iv} in order to
conclude that $\R$ is topologically left perfect.

 Alternatively, one can prove directly from the definition that,
for any injective object $J$ in a locally Noetherian category $\sA$,
the decomposition of $J$ into a direct sum of indecomposable injectives
is a perfect decomposition.
 By Theorem~\ref{perfect-decomposition-endomorphism-ring}, it follows
that the topological ring $\R=\Hom_\sA(J,J)^\rop$ is topologically
left perfect.
\end{proof}

\Section{Topologically Perfect Topological Rings}
\label{topologically-perfect-secn}

 Recall the definition given in
Section~\ref{perfect-decompositions-secn}: a complete, separated
topological ring $\R$ with right linear topology is \emph{topologically
left perfect} if its topological Jacobson radical $\HH=\HH(\R)$ is
topologically left T\+nilpotent and strongly closed in $\R$,
and the quotient ring $\S=\R/\HH$ in its quotient topology is
topologically semisimple.
 In this section we discuss equivalent conditions characterizing
topologically perfect topological rings.

 There is a number of such conditions which we consider.
 Some of them are indeed equivalent to topological perfectness, as
we prove.
 The equivalence of the rest of the conditions is a conjecture.
 This conjecture is equivalent to a positive answer to
Question~\ref{as-main-question}, as we explain.
 In addition to results obtained above in this paper, we make use of
some results from the papers~\cite{Pproperf} and~\cite{BPS}.

 Given a full subcategory $\sC$ in a category $\sA$, a morphism
$c\:C\rarrow A$ in $\sA$ is said to be a \emph{$\sC$\+precover} (of
the object~$A$) if $C\in\sC$ and for every morphism $c'\:C'\rarrow A$
in $\sA$ with $C'\in\sC$ there exists a morphism $f:C'\rarrow C$
such that $c'=cf$.
 A $\sC$\+precover $c\:C\rarrow A$ is called a \emph{$\sC$\+cover}
(of~$A$) if the equation $cf=c$ for a morphism $f\:C\rarrow C$
implies that $f$~is an isomorphism.
 We refer to the papers~\cite[Sections~4 and~10]{Pproperf}
and~\cite[Sections~3 and~7]{BPS} for a discussion of projective covers
in contramodule categories.

 Let $\R$ be a complete, separated topological ring with a right linear
topology.
 A left $\R$\+contramodule $\F$ is said to be \emph{flat} if the functor
of contratensor product \mbox{${-}\ocn_\R\F\:$}$\discr\R\rarrow\Ab$ is exact.
 All projective left $\R$\+contramodules are flat, and the class of
all flat left $\R$\+contramodules is closed under direct limits in
$\R\contra$; so all the direct limits of projective left
$\R$\+contramodules are flat~\cite[Section~3]{Pproperf}.

 It is \emph{not} known whether the converse is true (i.~e., whether
an analogue of the Govorov--Lazard description of flat modules holds
for contramodules).%
\footnote{A counterexample is available now: \cite[Example~10.2]{PPT}.}
 Nevertheless, we can prove the following theorem, which is the main
result of this section.

\begin{thm} \label{i=ii=iii=iv}
 Let\/ $\R$ be a complete, separated topological ring with a right
linear topology.
 Then the following conditions are equivalent:
\begin{itemize}
\item[(i)] all flat left\/ $\R$\+contramodules have projective covers;
\item[(i$'$)] all the direct limits of projective
left\/ $\R$\+contramodules have projective covers in\/ $\R\contra$;
\item[(ii)] all left\/ $\R$\+contramodules have projective covers;
\item[(iii)] all flat left\/ $\R$\+contramodules are projective;
\item[(iii$'$)] the class of all projective left\/ $\R$\+contramodules\/
$\R\contra_\proj$ is closed under direct limits in\/ $\R\contra$;
\item[(iv)] the topological ring\/ $\R$ is topologically left perfect.
\end{itemize} 
\end{thm}

\begin{proof}
 (ii)\,$\Longrightarrow$\,(i), (iii)\,$\Longrightarrow$\,(i),
and (iii$'$)\,$\Longrightarrow$\,(i$'$) Obvious.

 (i)\,$\Longrightarrow$\,(i$'$) and (iii)\,$\Longrightarrow$\,(iii$'$)
These implications hold because all the direct limits of projective
contramodules are flat (see the above discussion).
 
 (iv)\,$\Longrightarrow$\,(iii)
 The argument uses the construction of the reduction functor
$\R\contra\allowbreak\rarrow\S\contra$ taking a left
$\R$\+contramodule $\C$ to the left $\S$\+contramodule $\C/\HH\tim\C$,
where $\HH=\HH(\R)$ (see~\cite[Sections~2.10 and~2.12]{Pproperf}).
 By~\cite[Theorem~9.3]{Pproperf}, a flat left $\R$\+contramodule $\F$ is
projective if and only if the left $\S$\+contramodule $\F/\HH\tim\F$
is projective.
 By Theorem~\ref{topologically-semisimple-ring}\,(1), all
left $\S$\+contramodules are projective.

 (iii$'$)\,$\Longleftrightarrow$\,(iv) This is
Corollary~\ref{weak-iii-implies-iv}.

 (i$'$)\,$\Longrightarrow$\,(iii$'$) This
is~\cite[Corollary~7.5]{BPS}.

 (iii) or (iii$'$) $\Longrightarrow$ (ii) This
is~\cite[Theorem~10.1 or Corollary~10.2]{Pproperf}.

 (iv)\,$\Longrightarrow$\,(ii) Provable along the lines
of~\cite[second proof of Theorem~10.4]{Pproperf}.
\end{proof}

\begin{rem} \label{contramodule-proof-of-a-s-remark}
 The above proofs of the implications (iv)~$\Longrightarrow$
(iii)~$\Longrightarrow$~(iii$'$) in Theorem~\ref{i=ii=iii=iv} allow
to obtain a contramodule-based proof of the implication ``only if'' in
Theorem~\ref{angeleri-saorin}.
 Indeed, let $A$ be an associative ring and $M$ be a left $A$\+module
with a perfect decomposition.
 Consider the ring of endomorphisms $\R=\Hom_A(M,M)^\rop$,
and endow it with the finite topology (see
Example~\ref{modules-topologically-agreeable-example}\,(1)).

 By Theorem~\ref{perfect-decomposition-endomorphism-ring},
the topological ring $\R$ is topologically left perfect.
 So we conclude from the above argument based
on~\cite[Theorem~9.3]{Pproperf} that the class of all projective left
$\R$\+contramodules $\R\contra_\proj$ is closed under direct limits
in $\R\contra$.
 By Corollary~\ref{split-direct-limits-contramodule-interpretation},
it follows that the full subcategory $\Add(M)\subset A\modl$ has
split direct limits.
\end{rem}

 Let $\R$ be a complete, separated topological ring with a right linear
topology.
 A \emph{Bass flat left\/ $\R$\+contramodule}~\cite[Section~5]{Pproperf}
is, by definition, the direct limit in the category of left
$\R$\+contramodules $\R\contra$ of a sequence of free left
$\R$\+contramodules $\R$ with one generator and left
$\R$\+contramodule morphisms between them,
$$
 \B=\varinjlim\,(\R\rarrow\R\rarrow\R\rarrow\dotsb).
$$

 Now we can formulate our conjecture.

\begin{conj} \label{topologically-perfect-ring-conjecture}
 Let\/ $\R$ be a complete, separated topological associative ring with
a right linear topology.
 Then the following conditions are equivalent to each other and to
the conditions listed in Theorem~\textup{\ref{i=ii=iii=iv}}:
\begin{itemize}
\item[(i$^\flat$)] all Bass flat left\/ $\R$\+contramodules have
projective covers;
\item[(iii$^\flat$)] all Bass flat left\/ $\R$\+contramodules
are projective;
\item[(v)] all discrete right\/ $\R$\+modules are coperfect.
\end{itemize} 
\end{conj}

 A partial generalization of
Conjecture~\ref{topologically-perfect-ring-conjecture} to
locally presentable abelian categories with a projective generator
is formulated in~\cite[Main Conjecture~4.6]{BP2}.

 The following additional property will be useful in the discussion
below.
 We do not expect it to be equivalent to the other conditions in
Conjecture~\ref{topologically-perfect-ring-conjecture} in general,
but sometimes it is, as we will see:
\begin{itemize}
\item[(vi)] all the discrete quotient rings of\/ $\R$ (i.~e.,
the quotient rings of\/ $\R$ by its open two-sided ideals) are
left perfect.
\end{itemize}
 
 The next theorem lists those implications in the above conjecture
that we can prove unconditionally.

\begin{thm} \label{unconditional}
 The following implications between the properties in
Theorem~\textup{\ref{i=ii=iii=iv}},
in Conjecture~\textup{\ref{topologically-perfect-ring-conjecture}},
and the additional property~\textup{(vi)} hold true:
$$
\begin{tikzcd}
\text{\textup{(i)}} \arrow[r, Leftrightarrow]
\arrow[d, Leftrightarrow]
&\text{\textup{(ii)}} \arrow[r, Leftrightarrow]
&\text{\textup{(iii)}} \arrow[r, Leftrightarrow]
\arrow[d, Leftrightarrow]
&\text{\textup{(iv)}} \\
\text{\textup{(i$'$)}} \arrow[rr, Leftrightarrow]
\arrow[d, Rightarrow]
&&\text{\textup{(iii$'$)}} \arrow[d, Rightarrow] \\
\text{\textup{(i$^\flat$)}} \arrow[rr, Leftrightarrow]
&&\text{\textup{(iii$^\flat$)}} \arrow[r, Rightarrow]
&\text{\textup{(v)}} \arrow[r, Rightarrow] &\text{\textup{(vi)}}
\end{tikzcd}
$$
\end{thm}
 
\begin{proof}
 The equivalence of all the conditions in the first two lines is
the result of Theorem~\ref{i=ii=iii=iv}.
 The equivalence (i$^\flat$)\,$\Longleftrightarrow$\,(iii$^\flat$)
is provided by~\cite[Corollary~3.10]{BPS}.
 
 The rest is essentially explained in~\cite[proofs of Theorems~11.1
and~13.4]{Pproperf} (though the generality level of the exposition
in~\cite{Pproperf} is somewhat more restricted than in the present
paper; cf.\ Remark~\ref{infinite-vs-finite-matrices} below).
 Specifically, the implication (iii$^\flat$)\,$\Longrightarrow$\,(v)
is~\cite[Proposition~5.3 and Lemma~7.3]{Pproperf}.
 The implication (i$^\flat$)\,$\Longrightarrow$\,(vi)
is~\cite[Corollary~5.7]{Pproperf}.
 The implication (v)\,$\Longrightarrow\,$(vi) is explained
in~\cite[proof of Theorem~11.1]{Pproperf}.
\end{proof}

 In the rest of this section we discuss some special cases when one can
prove Conjecture~\ref{topologically-perfect-ring-conjecture} and/or some
partial/related results in its direction.

\begin{thm} \label{under-abcd}
 The implication \textup{(vi)} $\Longrightarrow$ \textup{(iv)} holds
under any one of the following three additional assumptions
of\/~\cite[Section~11]{Pproperf}: either \par
\textup{(a)} the ring\/ $\R$ is commutative, or \par
\textup{(b)} the topological ring\/ $\R$ has a countable base of
neighborhoods of zero consisting of two-sided ideals, or \par
\textup{(c)} the topological ring\/ $\R$ has a base of neighborhoods
of zero consisting of two-sided ideals and only a finite number of
semisimple Artinian discrete quotient rings, \par
as well as under the further assumption~\textup{(d)}
formulated in\/~\cite[Section~13]{Pproperf}.
\end{thm}

\begin{proof}
 This is the result of~\cite[Propositions~11.2 and~13.5]{Pproperf}
together with~\cite[Lemma~9.1]{Pproperf} (cf.\ our
Lemma~\ref{H-is-topological-Jacobson}).
\end{proof}

\begin{rem} \label{infinite-vs-finite-matrices}
 The main difference between the generality levels of the expositions
in~\cite{Pproperf} and in the present paper is that the topological
quotient ring $\S=\R/\HH(\R)$ is assumed to be the topological product
of some discrete simple Artinian rings in the paper~\cite{Pproperf}
(see~\cite[Section~9]{Pproperf}).
 This, of course, means rings of \emph{finite-sized} matrices over
skew-fields.
 In the present paper, the same role is played by \emph{topologically}
semisimple topological rings, which are topological products of
the topological rings of \emph{infinite-sized}, row-finite matrices
over skew-fields (see Theorem~\ref{topologically-semisimple-ring}).

 The explanation is that the idea of the proof of the main results
in~\cite{Pproperf} is to deduce the condition~(iv) from~(vi), as in
Theorem~\ref{under-abcd}.
 This is doable under one of the assumptions~(a), (b), (c), or~(d),
which are designed to make this approach work.
 Any one of these four assumptions, under which the main results of
the paper~\cite{Pproperf} are obtained, implies that $\S$ is
the topological product of discrete simple Artinian rings.
 
 It is instructive to consider the particular case when
$\R=\Mat_Y(D)$ is the topological ring of row-finite matrices of some
infinite size $Y$ over a skew-field~$D$.
 Then $\R$ has \emph{no} nonzero proper closed two-sided ideals.
 As $\R$ itself is not discrete, it follows that $\R$ has no nonzero
discrete quotient rings.
 Though not a counterexample to a possible implication
(vi)\,$\Longrightarrow$\,(iv), this simple example seems to suggest
that topological rings with right linear topologies may have too few
discrete quotient rings for the condition~(vi) to be interesting in
the general case.
 That is why we doubt the general validity of 
(vi)\,$\Longrightarrow$\,(iv), and therefore do not include
the condition~(vi) in the list of conditions in
Conjecture~\ref{topologically-perfect-ring-conjecture}.
 (Cf.\ Proposition~\ref{under-base-of-two-sided-ideals} below, which
concerns the case of a topological ring $\R$ with a topology base of
two-sided ideals.)
\end{rem}

\begin{cor} \label{all-equivalent}
 Conjecture~\textup{\ref{topologically-perfect-ring-conjecture}} holds
true under any one of the additional assumptions \textup{(a)},
\textup{(b)}, \textup{(c)}, or~\textup{(d)} of
the paper~\cite{Pproperf}.
 In particular, the conjecture is true for commutative rings\/~$\R$.
\end{cor}

\begin{proof}
 Follows from Theorems~\ref{unconditional} and~\ref{under-abcd}.
\end{proof}

\begin{thm} \label{under-countable-base}
 Conjecture~\textup{\ref{topologically-perfect-ring-conjecture}}
holds true for any complete, separated topological ring\/ $\R$ with
a \emph{countable} base of neighborhoods of zero consisting of
open right ideals.
\end{thm}

\begin{proof}
 In view of Theorem~\ref{unconditional}, it suffices to prove
the implication (v)\,$\Longrightarrow$\,(iv).
 For topological rings $\R$ with a countable base of neighborhoods of
zero, it is provided by
Theorem~\ref{countable-discrete-coperfect-implies-top-perfect}.
\end{proof}

\begin{rem} \label{v-implies-iv=as-question}
 The implication (v)\,$\Longrightarrow$\,(iv) in
Conjecture~\ref{topologically-perfect-ring-conjecture} is true
\emph{if and only if} the answer to Question~\ref{as-main-question} 
is positive.

 Indeed, given a left $A$\+module $M$, one considers its ring of
endomorphisms $\R=\Hom_A(M,M)^\rop$, endowed with the finite topology.
 Conversely, given a complete, separated topological ring $\R$ with
right linear topology, by
Corollary~\ref{topological-rings-are-endomorphism-rings} one can find
an associative ring $A$ and a left $A$\+module $M$ such that
$\R$ is isomorphic to $\Hom_A(M,M)^\rop$ as a topological ring.

 Now, according to Lemma~\ref{endocoperfect-discrete-coperfect},
the right $\R$\+module $M$ is $\Sigma$\+coperfect if and only if
the topological ring $\R$ satisfies condition~(v).
 On the other hand, by
Theorem~\ref{perfect-decomposition-endomorphism-ring}, the left
$A$\+module $M$ has a perfect decomposition if and only if
the topological ring $\R$ satisfies condition~(iv).
 (Cf.\ the proof of
Theorem~\ref{countably-generated-endocoperfect-theorem}.)

 It follows that the whole
Conjecture~\ref{topologically-perfect-ring-conjecture} is
equivalent to a positive answer to Question~\ref{as-main-question}.
\end{rem}

\begin{cor}
 Any endo-$\Sigma$-coperfect left module $M$ over an associative ring
$A$ with a commutative endomorphism ring\/ $\R=\Hom_A(M,M)^\rop$ has
a perfect decomposition.
\end{cor}

\begin{proof}
 Follows from Corollary~\ref{all-equivalent},
Lemma~\ref{endocoperfect-discrete-coperfect}, and
Theorem~\ref{perfect-decomposition-endomorphism-ring},
as explained the preceding remark.
\end{proof}

\begin{rem} \label{ring-summation-easy-counterex}
 Now we can present an alternative argument for the negative answer
to the question posed in~\cite[Section~1.2]{PR}, which allows to avoid
using the theory of direct sum decompositions of modules, as discussed
above in Example~\ref{ring-summation-structure-counterex}.
 Let $\sB$ be a nondiscrete spectral category with a generator $G$
such that the ring of endomorphisms $R=\Hom_\sB(G,G)^\rop$ is
commutative.
 It is explained in Example~\ref{complete-boolean} how such
categories $\sB$ can be produced.
 We would like to show that there does \emph{not} exist a complete,
separated topological ring $\T$ with right linear topology such
that the abelian category $\sB$ is equivalent to $\T\contra$.

 Indeed, suppose that such a topological ring~$\T$ exists.
 Then the category $\sB_\proj$ is topologically agreeable by
Remark~\ref{projective-contramodules-topologically-agreeable}.
 Hence the ring $\R=R=\Hom_\sB(G,G)^\rop$ also acquires a topology
such that there is an equivalence of categories $\sB\cong\R\contra$
taking the projective generator $G\in\sB$ to the free
$\R$\+contramodule with one generator $\R\in\R\contra$
(see the last paragraph of
Example~\ref{ring-summation-structure-counterex}).

 Now if the category $\R\contra$ is split abelian, then all
left $\R$\+contramodules are projective, and in particular, all
flat left $\R$\+contramodules are projective, so condition~(iii)
is satisfied.
 As $\R$ is commutative, Theorem~\ref{under-abcd} applies, providing
the implication (vi)\,$\Longrightarrow$\,(iv).
 The implication (iii)\,$\Longrightarrow$\,(i$^\flat$) is obvious,
and the argument from Theorem~\ref{unconditional} for the implication
(i$^\flat$)\,$\Longrightarrow$\,(vi) can be used.
 All of this is covered by~\cite[Theorem~11.1]{Pproperf}.
 Hence we can conclude that the topological ring $\R$ satisfies~(iv),
i.~e., $\R$ is topologically left perfect.

 Taking into account Corollary~\ref{no-top-T-nilpotent-ideals-cor},
it follows that the topological ring $\R$ is topologically semisimple;
so it is a topological product of discrete fields.
 In particular, $R$ is isomorphic to a product of fields as an abstract
ring, which implies that the category $\sB$ is semisimple.
 The contradiction proves our claim.
\end{rem}

\begin{thm} \label{for-top-coherent}
 Conjecture~\textup{\ref{topologically-perfect-ring-conjecture}}
holds true for any topologically right coherent topological
ring\/~$\R$.
\end{thm}

\begin{proof}
 In view of Theorem~\ref{unconditional}, it suffices to prove any one
of the implications (v)\,$\Longrightarrow$\,(iii$'$) or
(v)\,$\Longrightarrow$\,(iv).
 For topologically right coherent topological rings $\R$, these are
provided by Theorem~\ref{top-coherent-top-coperfect}.
\end{proof}

\begin{cor}
 Let $M$ be an endo-$\Sigma$-coperfect left module over an associative
ring $A$, and let\/ $\R=\Hom_A(M,M)^\rop$ be the (opposite ring to)
the endomorphism ring of $M$, viewed as a topological ring in
the finite topology.
 Assume that the topological ring\/ $\R$ is topologically right coherent.
 Then the $A$\+module $M$ has a perfect decomposition.
\end{cor}

\begin{proof}
 Follows from Theorem~\ref{for-top-coherent},
Lemma~\ref{endocoperfect-discrete-coperfect}, and
Theorem~\ref{perfect-decomposition-endomorphism-ring},
as explained in Remark~\ref{v-implies-iv=as-question}.
 Alternatively, one can apply directly
Theorem~\ref{top-coherent-top-coperfect},
Corollary~\ref{split-direct-limits-contramodule-interpretation}
and Theorem~\ref{angeleri-saorin}.
\end{proof}

 Before we finish this section, let us list two partial or conditional
results in the direction of the main conjecture.

\begin{prop}
 Let\/ $\R$ be a complete, separated topological ring with a right
linear topology.
 Suppose that\/ $\R$ satisfies condition~\textup{(v)}, i.~e., all
discrete right\/ $\R$\+modules are coperfect.
 Then the topological Jacobson radical\/ $\HH(\R)$ of the topological
ring\/ $\R$ is topologically left T\+nilpotent (which is a part of
condition~\textup{(iv)}).
\end{prop}

\begin{proof}
 This is~\cite[Corollary~7.7]{Pproperf}.
\end{proof}

\begin{prop} \label{under-base-of-two-sided-ideals}
 Let\/ $\R$ be a complete, separated topological ring with a base
of neighborhoods of zero consisting of open \emph{two-sided} ideals.
 Then the equivalence \textup{(v)\,$\Longleftrightarrow$\,(vi)} in
Conjecture~\ref{topologically-perfect-ring-conjecture} holds true. 
\end{prop}

\begin{proof}
 This is explained in~\cite[proof of Theorem~11.1]{Pproperf}.
\end{proof}

\end{document}